\documentclass[12pt,reqno]{amsart}
\usepackage{amssymb,graphicx}
\usepackage{amsfonts}
\usepackage{tikz}
\usetikzlibrary{shapes,arrows}
\usetikzlibrary{shapes.geometric,arrows,chains}
\usepackage{mathrsfs}
\usepackage{ifpdf}
\usepackage[hidelinks]{hyperref}
\hypersetup{colorlinks=true,linktocpage=true,
	linkcolor=blue,
	filecolor=blue,
	citecolor = blue,
	urlcolor=red}
\usepackage{setspace}
\usepackage{bm}
\usepackage{float}
\usepackage[margin=1.2in]{geometry}

\usepackage{bbm}
\usepackage{physics}
\usepackage{amsmath}
\usepackage{tikz}
\usepackage{mathdots}
\usepackage{yhmath}
\usepackage{cancel}
\usepackage{color}
\usepackage{siunitx}
\usepackage{array}
\usepackage{multirow}
\usepackage{gensymb}
\usepackage{tabularx}
\usepackage{extarrows}
\usepackage{booktabs}
\usetikzlibrary{fadings}
\usetikzlibrary{patterns}
\usetikzlibrary{shadows.blur}
\allowdisplaybreaks[4]
\theoremstyle{plain}
\newtheorem*{theorem}{Theorem}
\newtheorem{thm}{Theorem}[section]
\newtheorem{lem}[thm]{Lemma}
\newtheorem{cor}[thm]{Corollary}
\newtheorem{definition}[thm]{Definition}

\newtheorem{rmk}[thm]{Remark}

\numberwithin{equation}{section}

\def\a{\alpha}
\def\b{\beta}
\def\B{B(x_0,R)}

\def\B2T{B(x_0,2R)\times[0,T]}
\def\d{\Delta}
\def\e{\varepsilon}
\def\g{\gamma}

\def\M{\mathcal{M}^n}
\def\O{\Omega}
\def\Or{\Omega_r}
\def\P{\Phi}

\def\T{\textbf{T}}
\def\R{\mathbb{R}}
\def\RCD{${\rm RCD}^{\ast}(K,N)$}

\def\S{\mathcal{S}}

\def\x{x^{\star}}
\makeatother
\makeatletter
\@namedef{subjclassname@2020}{\textup{2020} Mathematics Subject Classification}
\makeatother

\begin{document}
	
	\title[]
	{ Global and local properties of solutions of elliptic equations with a nonlinear term involving the product of the function and its gradient}
	
	\author[Zhihao Lu]{Zhihao Lu}
	\address[Zhihao Lu]{School of Mathematics and Statistics, Jiangsu Normal University, Xuzhou 221116, P. R. China}
	\email{zhihaolu@jsnu.edu.cn}

	\begin{abstract}
%Recently, Bidaut-Véron-García-Huidobro-Véron \cite{BGV} researched the local gradient type estimates and Liouville theorems for the following nonlinear equations with a source reaction term  involving the product of the function and its gradient, which contains classical results related to Lane-Emden equation and Hamilton-Jacobi equation. However, on the one hand, except radial positive solutions, the sharp Liouville theorem is widely open; on the other hand, their gradient type estimates cannot be fully established on the indicator region where the Liouville theorem holds. In the present paper, we reconsider these basic questions. Concretely, we first establish an improved Liouville theorem, which is sharp when  and also holds on Riemannian manifolds with nonnegative Ricci curvature. Secondly, by introducing a new method---Bernstein blow-up method, we derive the logarithmic gradient estimate and gradient estimate on the whole indicator region where the Liouville theorem holds, which essentially enhance the corresponding estimates in \cite{BGV}. Naturally, we also obtain the Harnack inequalities on same indicator region.
 We study the global and local properties of positive solutions to the quasi-linear elliptic equation:
	\begin{equation}
	\d u+|\nabla u|^q u^p=0,\quad x\in \O\subset \mathbb{R}^n,\nonumber
\end{equation}
where $q\ge 0$ and $p\in\mathbb{R}$.
Our contributions are twofold:
\begin{enumerate}
	\item[1.]  Based on an optimal and new identity for the modulus squared of the logarithmic gradient, we establish optimal and improved Liouville theorems for global positive solutions, and generalize these findings to the framework of Riemannian manifolds.
	
	\item[2.]Based on a newly discovered mutual control relationship of two nonlinear iterms, for all index pairs \( (p, q) \) where the Liouville theorem holds, we derive several optimal gradient estimates for local positive solutions. As a direct corollary, we obtain the corresponding Harnack inequality.
	
\end{enumerate}

%Our results  substantially  strengthen those already appearing in the literature.

These results strengthen the related conclusions in Bidaut-Véron--García-Huidobro--Véron \cite {BGV} from both global and local perspectives.
 
 %We mention that even for the partial results overlapping with previous ones, our methods of proofs are also new.
	\end{abstract}

%	 By means of the direct Bernstein method, we derive some optimal Liouville theorems for global positive solutions  and  generalize the corresponding results to Riemannian setting. Meanwhile, to derive the optimal local estimates, we have proposed a new method---the Bernstein blow up method. This method combines Bernstein method and Poláčik-Quittner-Souplet's doubling lemma and links the derivation of local gradient-type estimates with the known Liouville theorems for the equation. Eventually, we have obtained several gradient estimates for all index pairs \((p, q)\) where Liouville theorem holds.
	\keywords{nonlinear elliptic equations,  Liouville theorem, local gradient estimate,  Harnack inequality, Riemannian manifolds with curvature bounded below}
	
	\subjclass[2020]{Primary:  35J62,  35B08, 35B09; Secondary:  35A23,  58J05}
	
%	\thanks{School of Mathematics and Statistics, Jiangsu Normal University, Xuzhou 221116, P. R. China}
	
	%\thanks{This paper was typeset using \AmS-\LaTeX}
	
	\maketitle
	{\tableofcontents}
	\section{Introduction}

	\subsection{Background and motivation}\label{ss1}
	In this paper, we mainly study some fine local and global properties of the positive classical solutions to the following 
	 equation on the Euclidean spaces:
	\begin{equation}\label{pqle}
		\d u+|\nabla u|^q u^p=0,\quad x\in\O\subset \mathbb{R}^n,
	\end{equation}
where $q\ge 0$, $p\in\mathbb{R}$.
%Although many results to the study of equation \eqref{pqle} in the Euclidean setting presented here are new, they can be naturally extended to the context of Riemannian manifolds. Specifically, under Ricci curvature bounded below setting, we also establish analogous results for equation \eqref{pqle}.
 In many certain cases, we actually derive the corresponding results for equation \eqref{pqle} on Riemannian manifolds with Ricci curvature bounded below though analogous results are new for Euclidean case.%For deriving the sharp local estimates such as logarithmic gradient estimate and Harnack inequality, we restrict the manifold as $\mathbb{R}^n$. %Simultaneously, we also generalize partial results of equation \eqref{pqle} to the following equation 
%	\begin{equation}\label{ge}
%	\d u+|\nabla u|^q f(u)=0,
%\end{equation}
%where $q\ge 0$ and $f\in C^1(0,\infty)$ with some  growth conditions. 

%\begin{figure}[h]
%	\includegraphics[scale=0.7]{Pro1.png}
%	\caption{Some separated curves in previous progresses (N=6)}\label{F1}
%\end{figure}

Before we state main results, we review some interesting and important studies for equation \eqref{pqle}. Meanwhile, we also analyze the research methods behind these results. For a clearer view, we have summarized them into the following figure.

	\begin{figure}[!htb]

	\tikzset{every picture/.style={line width=0.75pt}} %set default line width to 0.75pt        
	
	\begin{tikzpicture}[x=0.75pt,y=0.75pt,yscale=-1,xscale=1]
		%uncomment if require: \path (0,414); %set diagram left start at 0, and has height of 414
		
		%Shape: Axis 2D [id:dp2562784235403033] 
		\draw  (50.67,336.89) -- (605.67,336.89)(82.66,55.75) -- (82.66,369.75) (598.67,331.89) -- (605.67,336.89) -- (598.67,341.89) (77.66,62.75) -- (82.66,55.75) -- (87.66,62.75) (182.66,331.89) -- (182.66,341.89)(282.66,331.89) -- (282.66,341.89)(382.66,331.89) -- (382.66,341.89)(482.66,331.89) -- (482.66,341.89)(582.66,331.89) -- (582.66,341.89)(77.66,236.89) -- (87.66,236.89)(77.66,136.89) -- (87.66,136.89) ;
		\draw   (189.66,348.89) node[anchor=east, scale=0.75]{1} (289.66,348.89) node[anchor=east, scale=0.75]{2} (389.66,348.89) node[anchor=east, scale=0.75]{3} (489.66,348.89) node[anchor=east, scale=0.75]{4} (589.66,348.89) node[anchor=east, scale=0.75]{5} (79.66,236.89) node[anchor=east, scale=0.75]{1} (79.66,136.89) node[anchor=east, scale=0.75]{2} ;
		%Straight Lines [id:da5759267657203901] 
		\draw  [dash pattern={on 4.5pt off 4.5pt}]  (82.67,237.08) -- (600.67,237.08) ;
		%Straight Lines [id:da5120382948888427] 
		\draw  [dash pattern={on 4.5pt off 4.5pt}]  (82.67,136.75) -- (605.67,135.75) ;
		%Straight Lines [id:da7085509835826891] 
		\draw [color={rgb, 255:red, 50; green, 112; blue, 188 }  ,draw opacity=1 ][fill={rgb, 255:red, 74; green, 144; blue, 226 }  ,fill opacity=1 ]   (59.67,114.42) -- (282.67,336.75) ;
		%Straight Lines [id:da056204043162845] 
		\draw [color={rgb, 255:red, 245; green, 166; blue, 35 }  ,draw opacity=1 ] [dash pattern={on 4.5pt off 4.5pt}]  (182.67,256.42) -- (262.67,336.75) ;
		%Curve Lines [id:da48790554715782286] 
		\draw [color={rgb, 255:red, 245; green, 166; blue, 35 }  ,draw opacity=1 ] [dash pattern={on 4.5pt off 4.5pt}]  (182.67,256.42) .. controls (154.67,225.75) and (117.67,191.08) .. (102.67,136.08) ;
		%Curve Lines [id:da3201158047053496] 
		\draw    (282.67,336.75) .. controls (249.67,313.42) and (155.67,244.42) .. (121.67,136.42) ;
		%Curve Lines [id:da448378894966204] 
		\draw [color={rgb, 255:red, 208; green, 2; blue, 27 }  ,draw opacity=1 ]   (282.67,336.75) .. controls (220.67,272.08) and (226.67,245.08) .. (557.67,241.08) ;
		%Curve Lines [id:da7817364319566846] 
		\draw [color={rgb, 255:red, 144; green, 19; blue, 254 }  ,draw opacity=1 ]   (134.67,136.75) .. controls (143.67,167.75) and (201.67,260.75) .. (268.67,321.75) ;
		%Straight Lines [id:da2755819843261147] 
		\draw [color={rgb, 255:red, 245; green, 166; blue, 35 }  ,draw opacity=1 ] [dash pattern={on 4.5pt off 4.5pt}]  (82.67,237.08) -- (182.67,336.75) ;

		\draw (65,156.48) node  [font=\footnotesize,color={rgb, 255:red, 0; green, 0; blue, 0 }  ,opacity=1 ,rotate=-358.41]  {$\frac{n+3}{n-1}$};
		% Text Node
		\draw (61,50.75) node [anchor=north west][inner sep=0.75pt]   [align=left] {q};
		% Text Node
		\draw (600,349.75) node [anchor=north west][inner sep=0.75pt]   [align=left] {p};
		% Text Node
		\draw (254,344.15) node [anchor=north west][inner sep=0.75pt]  [font=\tiny]  {$\frac{n+3}{n-1}$};
		% Text Node
		\draw (64,339.15) node [anchor=north west][inner sep=0.75pt]    {$O$};
		% Text Node
		\draw (118,117.15) node [anchor=north west][inner sep=0.75pt]    {$I$};
		% Text Node
		\draw (137,117.15) node [anchor=north west][inner sep=0.75pt]    {$II$};
		% Text Node
		\draw (85,117.82) node [anchor=north west][inner sep=0.75pt]    {$III_{1}$};
		% Text Node
		\draw (55,98.82) node [anchor=north west][inner sep=0.75pt]    {$IV$};
		% Text Node
		\draw (416,249.82) node [anchor=north west][inner sep=0.75pt]    {$V$};
		% Text Node
		\draw (100,276.15) node [anchor=north west][inner sep=0.75pt]    {$III_{2}$};

	\end{tikzpicture}

	\caption{Some separated curves in previous progresses ($n=6$)}\label{F1}
\end{figure}
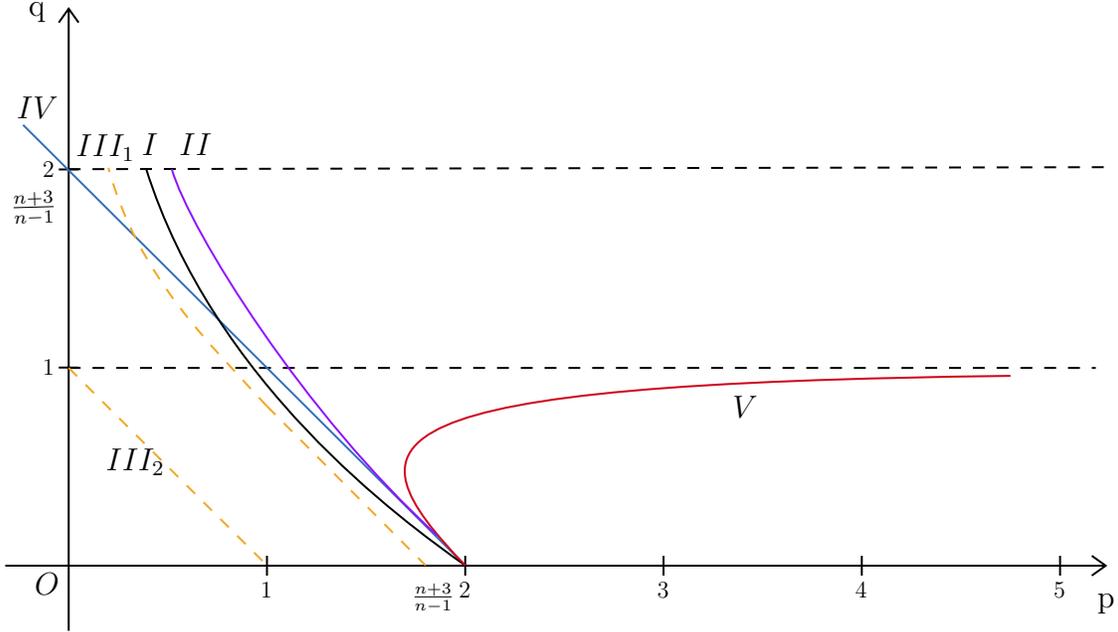

  Bidaut-Véron-García-Huidobro-Véron \cite[Theorem D]{BGV} proved that the red curve $V$ in Figure \ref{F1} is the critical curve that determines whether equation \eqref{pqle} has a nontrivial radially symmetric positive solution in $\mathbb{R}^n$. When $n\ge 3$, the equation of curve $V$ is given by 
$$
p+q-1=\frac{(2-q)^2}{(1-q)(n-2)},\qquad q\in[0,1).
$$
Then \cite[Theorem D]{BGV} states that  there exist non-constant radial positive solutions of \eqref{pqle} in $\mathbb{R}^n$ if and only if $0 \leq q<1$ and
$$
p+q-1\ge \frac{(2-q)^2}{(1-q)(n-2)}.
$$
That is to say, the equation \eqref{pqle} has nontrivial radial positive symmetric solutions if and only if indicator $(p,q)$ falls on the lower right corner of the red curve $V$.

When the symmetry condition is completely removed, it becomes even more challenging to determine the indicator region where the Liouville theorem of \eqref{pqle} holds. Specifically, inspired by \cite[Theorem D]{BGV}, one would naturally put forward the following conjecture.

\textbf{Conjecture:} Let $\mathbb{R}^2_{+}:=\{(p,q)\in \mathbb{R}^2:q\ge 0\}$ and  define the subdomain of $\mathbb{R}^2_{+}$
\begin{eqnarray}\label{NL}
	\mathscr{D}(n):=
	\begin{cases}
		\emptyset \qquad\qquad  \qquad\qquad \qquad\qquad\qquad\qquad\qquad\qquad\qquad   n= 2,\\
		\\
		\left\{(p,q)\in \mathbb{R}^2_{+}:q\in[0,1),\,\,p+q-1\ge \frac{(2-q)^2}{(1-q)(n-2)}.\right\} \quad  \,\,\,\, n\ge 3,
	\end{cases}
\end{eqnarray}
and define the Liouville domain of the equation \eqref{pqle} on $\mathbb{R}^n$ as follows.
\begin{eqnarray}\label{LD}
	\mathscr{D}_L(n)=\left\{\begin{array}{l|l}
		(p,q) \in \mathbb{R}^2_{+} & \begin{array}{l}
			\text{any positive classical solutions of \eqref{pqle} on } \\
			\text{ $\mathbb{R}^n$  with indices $p,q$ must be constant.}
		\end{array}
	\end{array}\right\}.
\end{eqnarray}
Then the question is that $\mathscr{D}_L(n)=\mathbb{R}^2_{+}\setminus \mathscr{D}(n)$?
\vspace{6pt}
%However, these methods are ineffective for equation \eqref{pqle} when $q>0$.

Generally speaking, this question is widely open.
For some special and important case, there are some nice results. For $q=0$ case, using integral method (or integral Bernstein method), Gidas-Spruck \cite{GS} first obtained the complete Liouville theorem, which states that the nonnegative classical solutions of \eqref{pqle} with $(1,p_S(n))$ must be trivial. Here, $p_S(n)$ is given by 
\begin{eqnarray}
	p_S(n)=
	\begin{cases}
		\infty\qquad \qquad\qquad\quad n=2;\nonumber\\
		\frac{n+2}{n-2}\qquad\qquad\qquad\,\,\, n\ge 3.
	\end{cases}
\end{eqnarray}
Their celebrated results are also reproved by moving plane or moving sphere method, e.g., \cite{CL,GNN,LZ}, see also the references therein.  Recently, Lu and Wu have also established the Gidas–Spruck type Liouville theorem on manifolds via the direct Bernstein method \cite{LU1,LU2,LU3,W}. Another interesting case is $p=0$. For this case, using direct Bernstein method, Lions \cite{Lions2} yields that all  classical solutions of \eqref{pqle} with $q>1$ are constant. However, for general case, the Liouville theorem without any assumptions for solutions to equation \eqref{pqle} is hard.  Recently, there are some remarkable progresses. For $q>2$ and $p>0$, 
Filippucci-Pucci-Souplet \cite{FPS} obtained that all bounded classical solutions of \eqref{pqle} is constant. Later on, using Keller--Osserman comparison and bootstrap method, Bidaut-Véron \cite[Theorem 1.1, Theorem 1.5]{B} completely removed the boundness condition and generalized the Liouville theorem to more general quasilinear problem. Based on the work \cite{GS,B}, the essential hard range for index is  $q\in(0,2)$.

%. And  Bidaut-Véron's Liouville theorem is valid for larger index range: $p+q-1>0$ and $p\le 0$ or $q\ge 2$. In fact, \cite{B} considered
 In an earlier and interesting work \cite[Theorem C]{BGV}, Bidaut-Véron-García-Huidobro-Véron established the Liouville theorem for a subdomain of $\mathbb{R}^2_{+}\setminus \mathscr{D}(n)$. Concretely speaking, using integral Bernstein method, they obtained that if $p\ge 0$, $q\in[0,2)$ and $(p,q)$ satisfies $G(p,q)<0$, then all the positive solutions of \eqref{pqle} is constant. Here, $$G(p,q)=((n-1)^2q+n-2)p^2+p(n(n-1)q^2-(n^2+n-1)q-n-2)-nq^2.$$
 Especially, for $q=0$ case, their result recovers the Gidas-Spruck's result. In Figure \ref{F1}, the black curve $I$ is $G(p,q)=0$ with $n=6$. Equivalently,  the bottom left corner of curve $I$ belongs to $\mathscr{D}_L(n)$. Notice that curve $I$ is strictly below the ideal curve $V$ except $q=0$ case. Very recently, using a refined integral identity and integral Bernstein method, Ma-Wu \cite[Theorem 1,3,Theorem 1.4]{MW}  improved the  Bidaut-Véron-García-Huidobro-Véron's Liouville theorem (see also \cite{DSWZ}). Especially, Ma-Wu's Liouville theorem is optimal when $p\ge 0$, $p+q-1>0$ and $q\in[0,\frac{1}{n-1}]$. That is, 
when $p\ge 0$, $p+q-1>0$, $q\in[0,\frac{1}{n-1}]$ and
$$
p+q-1< \frac{(2-q)^2}{(1-q)(n-2)},
$$
then $(p,q)\in \mathscr{D}_L(n)$. For $q\in(\frac{1}{n-1},2)$, they prove that if $p\ge 0$, $p+q-1>0$ and $\mathbb{H}(p,q)<0$, then $(p,q)\in \mathscr{D}_L(n)$, where
$$
\mathbb{H}(p, q):=p^2+\left(\frac{n-1}{n-2} q-\frac{n^2-3}{(n-2)^2}\right) p+\frac{1-(n-1) q}{(n-2)^2} .
$$
Visually speaking, their separated curve is purple curve $II$ in the Figure \ref{F1}. %In fact, the superlinear condition is redundant.
%Very recently, via Moser's iteration method, He-Hu-Wang \cite[Theorem 1.2]{HHW} established logarithmic gradient estimate for \eqref{pqle} with $p+q<\frac{n+3}{n-1}$ and $q\ge 0$, and so the corresponding Liouville theorem is obtained.  \vspace{3pt}

To summarize, up to now, aforementioned conjecture is open for the following region: $q\in(0,2)$, $p\ge 0$ and $(p,q)$ is located between curve $II$ and curve $V$ in Figure \ref{F1}.
%$\bullet$ $q\in(0,2)$, $p+q-1>\frac{n+3}{n-1}$ and $p<0$;\\
 
\vspace{2mm}
Therefore, one of the objectives of this paper is to improve the existing Liouville theorem.   
Before we go futher, we mention that the main difficulty of derive the optimal Liouville theorem of \eqref{pqle} is the existence of gradient item factor. Concretely, it causes that the important moving plane or sphere method fails in the present case and integral Bernstein method has become the mainstream choice currently.. However, up to now, the classical integral estimate from Gidas-Spruck \cite{GS} have not been fully generalized and applied successfully to obtain complete Liouville theorem for \eqref{pqle} (see \cite{BGV} and \cite{MW}). In the present work, we adopt the direct Bernstein method to deal with this question. In contrast to the case when \(q = 0\), we find an essential refinement for dealing with auxiliary functions because the classical differential identity is not optimal for present case (see \cite{LU1, LU3} for more details  $q=0$ case and subsection \ref{ss3} below for $q>0$ case).
\vspace{2mm}

In addition, the index region that the local estimates (such as gradient type estimates) hold is obviously smaller than the region where the known Liouville theorem holds. In \cite[Theorem B]{BGV}, Bidaut-Véron-García-Huidobro-Véron  derived a gradient type estimate as follows. 
If $p+q-1>0$, $q\in[0,2)$ and \\
$\bullet$ $p\ge 1$ and $p+q-1<\frac{n+3}{n-1}$ or\\
$\bullet$ $p\in[0,1)$ and $p+q-1<\frac{(p+1)^2}{p(n-1)}$,\\
then there exist $a=a(n,p,q)>0$ and $C=C(n,p,q)>0$ such that any positive solutions of \eqref{pqle} on $B(0,R)$ satisfy
$$
\sup\limits_{B(0,\frac{R}{2})}|\nabla u^a|\le CR^{-1-\frac{a(2-q)}{p+q-1}}.
$$
In Figure \ref{F1}, when $(p,q)$ lies between the yellow dashed curve $III_1$ and curve $III_2$, the aforementioned estimate holds. Subsequently, He--Hu--Wang \cite[Theorem 1.2]{HHW} established a logarithmic gradient estimate for equation \eqref{pqle} under the conditions $p+q<\frac{n+3}{n-1}$ and $q\ge 0$ (see also \cite{HWW,LU1,LU30} for the case when $q=0$). To date, these local gradient-type estimates constitute all the known results of this kind.

It is obvious that local gradient estimates is stronger than the corresponding Liouville theorem. However, from classical case $q=0$, they should be equivalent to each other. Indeed, in this case, the logarithmic gradient estimate, Harnack inequality and Liouville theorme are valid (see \cite[Theorem 1.6]{LU1}). Besides, Gidas-Spruck \cite{GS0} and Poláčik-Quittner-Souplet \cite{PQS} provide a philosophy that local estimates can be derived from Liouville theorems and so they are equivalent. Therefore, we pose the following natural question:
\vspace{1mm}

\textbf{Question:} For $n\ge 2$, let $\mathscr{D}_L(n)$ be defined by \eqref{LD}. For any $(p,q)\in \mathscr{D}_L(n)$, does these local gradient estimates and Harnack inequality hold? 
\vspace{1mm}

The another objective of this paper is to give a positive answer on this question.

\subsection{Main results}
First, we give the notation that will be used throughout the text.

\vspace{3mm}
	% 调整表格行间距（1.8为系数，越大行间距越宽）
	\renewcommand{\arraystretch}{1.5}
	\begin{tabular}{|p{2.5cm}|p{11.5cm}|} % 固定列宽，避免内容挤压
		\hline
		$(\M, g)$ & $n$-dimensional Riemannian manifold with metric $g$ \\
		\hline
		$\mathrm{Ric}$ & Ricci tensor on $(M^n, g)$ \\
		\hline
		$B(x_0, R)$ & the geodesic ball with center $x_0$ and radius $R$ \\
		\hline
		$\nabla f, df, \nabla^2 f$ & the gradient, differential and Hessian  of the $C^2$ function $f$ \\
		\hline
		$\mathrm{dist}(x, E)$ & the distance between $x$ and the subset $E$ in $\M$ \\
		\hline
		$\mathbbm{1}_E$ & the indicator function of the set $E$ \\
		\hline
		$\Omega_r(u)$ & regular set of $C^1$ function $u$ in $\Omega$: $\Omega_r(u) = \{x \in \Omega : |\nabla u|(x) > 0\}$ \\
		\hline
		$|T|$ & the norm of the tensor $T$ under the Riemannian metric $g$ \\
		\hline
	\end{tabular}

	\vspace{3mm}
	Now, we state main results. The first one is a partial answer for the Conjecture in the subsection \ref{ss1}, which essentially improves the known results.
	
	\begin{thm}\label{m1}
		Let $u$ be a nonnegative classical solution of \eqref{pqle} on $\mathbb{R}^n$ with $q\ge 0$ and $p\in \mathbb{R}$. Let $\mathscr{D}(n)$ be defined by \eqref{NL}. If any of the following conditions holds:\\
		(1) $n=2$;\\
		(2) $q\ge \frac{5}{3}$;\\
		(3) $q\in[0,1-\frac{1}{\sqrt{n-1}}]$ and $(p,q)\in \mathbb{R}^2_{+}\setminus \mathscr{D}(n)$ with $n\ge3$;\\
		(4) $q>0$ and $p+q\le \frac{n+2}{n-2}$ with $n\ge3$;\\
		(5) $p\le\frac{3\sqrt{n+6}}{2(n-2)}$  with $n\ge3$,\\
		then $u$ is constant. 
	\end{thm}
	
	The conditions $(1)$--$(3)$ are optimal. Especially, case (3) contains the classical Gidas-Spruck's result for $q=0$ and it also covers (4) when $q\in[0,1-\frac{1}{\sqrt{n-1}}]$. Regarding the linear region, the condition (4) is sharp ($n\ge 3$), i.e., 
	$$
	\frac{n+2}{n-2}=\sup\{l\in\mathbb{R}:p+q\le l, (p,q)\in \mathscr{D}_L(n)\},
	$$
	where $\mathscr{D}_L(n)$ is given by \eqref{LD}.
	This is due to the fact that the line $p+q=\frac{n+2}{n-2}$ (the  line $IV$ in the Figure \ref{F1}) is exactly the tangent to the curve $V$ at  point $(\frac{n+2}{n-2},0)$. From Theorem \ref{m1}, we see that the cases $q>0$ and $q=0$ are clearly different. In $q=0$ case,  equation \eqref{pqle} has positive solutions when $p=\frac{n+2}{n-2}$; and this fact is invalid for $p+q=\frac{n+2}{n-2}$ with $q>0$. Last, the condition (5) contains the Lions' classical result $p=0$ and $q>1$ and also covers the condition (4) when $q\ge 1$.

	In fact, in the non-optimal case, we have proven that the region where the Liouville theorem holds is larger than the region described by conditions (4) and (5). Concretely, this region is expressed by some inequlities in Definition \ref{set}. Its specific plot is as follows:

	% Pattern Info
	
	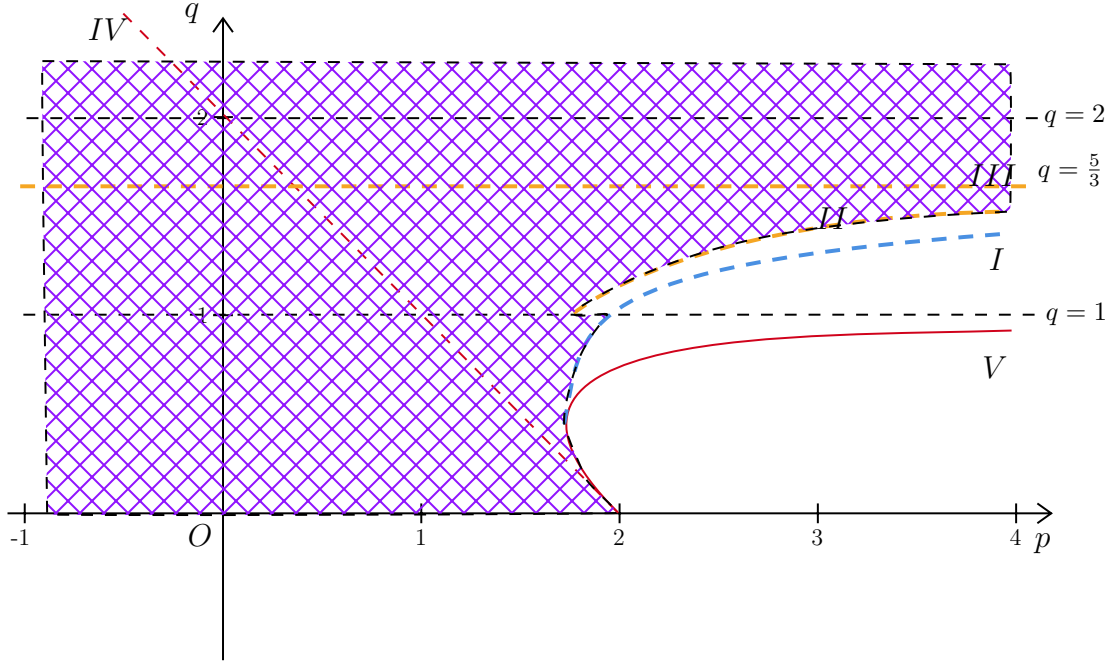
\begin{figure}[!htb]

% Pattern Info

\tikzset{
	pattern size/.store in=\mcSize, 
	pattern size = 5pt,
	pattern thickness/.store in=\mcThickness, 
	pattern thickness = 0.3pt,
	pattern radius/.store in=\mcRadius, 
	pattern radius = 1pt}
\makeatletter
\pgfutil@ifundefined{pgf@pattern@name@_gyjn2hfro}{
	\pgfdeclarepatternformonly[\mcThickness,\mcSize]{_gyjn2hfro}
	{\pgfqpoint{0pt}{0pt}}
	{\pgfpoint{\mcSize}{\mcSize}}
	{\pgfpoint{\mcSize}{\mcSize}}
	{
		\pgfsetcolor{\tikz@pattern@color}
		\pgfsetlinewidth{\mcThickness}
		\pgfpathmoveto{\pgfqpoint{0pt}{\mcSize}}
		\pgfpathlineto{\pgfpoint{\mcSize+\mcThickness}{-\mcThickness}}
		\pgfpathmoveto{\pgfqpoint{0pt}{0pt}}
		\pgfpathlineto{\pgfpoint{\mcSize+\mcThickness}{\mcSize+\mcThickness}}
		\pgfusepath{stroke}
}}
\makeatother
\tikzset{every picture/.style={line width=0.75pt}} %set default line width to 0.75pt        

\begin{tikzpicture}[x=0.75pt,y=0.75pt,yscale=-1,xscale=1]
	%uncomment if require: \path (0,439); %set diagram left start at 0, and has height of 439
	
	%Shape: Axis 2D [id:dp4329897875813441] 
	\draw  (45.4,303.14) -- (571.8,303.14)(154.09,53.28) -- (154.09,377.28) (564.8,298.14) -- (571.8,303.14) -- (564.8,308.14) (149.09,60.28) -- (154.09,53.28) -- (159.09,60.28) (254.09,298.14) -- (254.09,308.14)(354.09,298.14) -- (354.09,308.14)(454.09,298.14) -- (454.09,308.14)(554.09,298.14) -- (554.09,308.14)(54.09,298.14) -- (54.09,308.14)(149.09,203.14) -- (159.09,203.14)(149.09,103.14) -- (159.09,103.14) ;
	\draw   (261.09,315.14) node[anchor=east, scale=0.75]{1} (361.09,315.14) node[anchor=east, scale=0.75]{2} (461.09,315.14) node[anchor=east, scale=0.75]{3} (561.09,315.14) node[anchor=east, scale=0.75]{4} (61.09,315.14) node[anchor=east, scale=0.75]{-1} (151.09,203.14) node[anchor=east, scale=0.75]{1} (151.09,103.14) node[anchor=east, scale=0.75]{2} ;
	%Straight Lines [id:da6229616054259954] 
	\draw [line width=0.75]  [dash pattern={on 4.5pt off 4.5pt}]  (55,103.68) -- (567,103.68) ;
	%Straight Lines [id:da9058824549941648] 
	\draw [color={rgb, 255:red, 208; green, 2; blue, 27 }  ,draw opacity=1 ] [dash pattern={on 4.5pt off 4.5pt}]  (103.8,50.88) -- (353.4,302.88) ;
	%Straight Lines [id:da5617969123519926] 
	\draw [line width=0.75]  [dash pattern={on 4.5pt off 4.5pt}]  (53.4,202.88) -- (565.4,202.88) ;
	%Curve Lines [id:da5765925554381958] 
	\draw [color={rgb, 255:red, 208; green, 2; blue, 27 }  ,draw opacity=1 ]   (353.4,302.88) .. controls (258.2,207.68) and (445.4,214.08) .. (551.8,210.88) ;
	%Straight Lines [id:da2585148569955078] 
	\draw [color={rgb, 255:red, 245; green, 166; blue, 35 }  ,draw opacity=1 ][line width=1.5]  [dash pattern={on 5.63pt off 4.5pt}]  (51.8,138.08) -- (563.8,138.08) ;
	%Curve Lines [id:da7387450903582022] 
	\draw [color={rgb, 255:red, 245; green, 166; blue, 35 }  ,draw opacity=1 ][line width=1.5]  [dash pattern={on 5.63pt off 4.5pt}]  (331,202.08) .. controls (371,172.08) and (457.4,151.68) .. (551,150.88) ;
	%Curve Lines [id:da7429117527320178] 
	\draw [color={rgb, 255:red, 74; green, 144; blue, 226 }  ,draw opacity=1 ][line width=1.5]  [dash pattern={on 5.63pt off 4.5pt}]  (326.2,258.08) .. controls (335.8,206.08) and (328.6,178.88) .. (547.8,162.08) ;
	%Shape: Polygon Curved [id:ds02828704486376976] 
	\draw  [pattern=_gyjn2hfro,pattern size=9pt,pattern thickness=0.75pt,pattern radius=0pt, pattern color={rgb, 255:red, 144; green, 19; blue, 254}][dash pattern={on 4.5pt off 4.5pt}] (63,74.88) .. controls (63.8,74.88) and (550.2,76.48) .. (551,76.48) .. controls (551.8,76.48) and (551.5,150.88) .. (551,150.88) .. controls (550.5,150.88) and (477.4,153.28) .. (424.08,164.3) .. controls (370.76,175.31) and (331.8,198.88) .. (331,202.08) .. controls (330.2,205.28) and (348.6,201.28) .. (347.8,202.88) .. controls (347,204.48) and (339.68,212.47) .. (335,222.08) .. controls (330.32,231.7) and (325.17,252.45) .. (326.2,258.08) .. controls (327.23,263.72) and (335.8,282.88) .. (337.4,286.08) .. controls (339,289.28) and (353.4,302.05) .. (353.4,302.88) .. controls (353.4,303.72) and (65.4,304.48) .. (65.4,303.68) .. controls (65.4,302.88) and (62.2,74.88) .. (63,74.88) -- cycle ;
	
	% Text Node
%	\draw (721,21) node    {$0$};
	% Text Node
%	\draw (701,71) node    {$0$};
	% Text Node
	\draw (134.8,307.4) node [anchor=north west][inner sep=0.75pt]    {$O$};
	% Text Node
	\draw (562,311.8) node [anchor=north west][inner sep=0.75pt]    {$p$};
	% Text Node
	\draw (133.2,44.2) node [anchor=north west][inner sep=0.75pt]    {$q$};
	% Text Node
	\draw (538.8,167.8) node [anchor=north west][inner sep=0.75pt]    {$I$};
	% Text Node
	\draw (451.6,147.4) node [anchor=north west][inner sep=0.75pt]    {$II$};
	% Text Node
	\draw (528.4,125.6) node [anchor=north west][inner sep=0.75pt]    {$III$};
	% Text Node
	\draw (535.6,222.4) node [anchor=north west][inner sep=0.75pt]    {$V$};
	% Text Node
	\draw (84.4,52) node [anchor=north west][inner sep=0.75pt]    {$IV$};
	% Text Node
	\draw (563.6,120.4) node [anchor=north west][inner sep=0.75pt]  [font=\footnotesize]  {$q=\frac{5}{3}$};
	% Text Node
	\draw (567.6,195.6) node [anchor=north west][inner sep=0.75pt]  [font=\footnotesize]  {$q=1$};
	% Text Node
	\draw (566.8,95.6) node [anchor=north west][inner sep=0.75pt]  [font=\footnotesize]  {$q=2$};

\end{tikzpicture}

	\caption{present Liouville domain of equation \eqref{pqle} ($n=6$)}\label{F02}
\end{figure}
	The purple grid area in the plot represents the region where the Liouville theorem is valid. When $q\in[0,1]$, it is bounded by the blue curve $I$; when  $q>1$, it is bounded by the yellow curve $II$ and $III$ ($q=\frac{5}{3}$).
	Now, we give a remark for Figure \ref{F02}.
	
	\begin{rmk}
\rm 	Notice that	curve $I$ and curve $II$ are incompatible when $q>1$. Technically, this incompatibility arises primarily from our use of different auxiliary functions tailored to each case. If we use the auxiliary function in $q\le 1$, we need a priori $L^{\infty}$ estimate of solution at the maximal point and fortunately, this key fact holds automatically in our arguments. However, when $q>1$, we do not have this estimate and hence we search another auxiliary function which does not need the priori estimate. However, the trade-off is that we lose some of the regions where the Liouville theorem holds. 
	\end{rmk}
	
	Under the assumption of boundedness, we obtain a Liouville theorem on a larger index range.

	\begin{thm}\label{m2}
		Let $u$ be a nonnegative bounded classical solution of \eqref{pqle} on $\mathbb{R}^n$ with $q\ge 0$ and $p\in \mathbb{R}$. If $p,q$ satisfy any one of the conditions in Theorem \ref{m1} or $q\ge\frac{3}{2}$, 
		then $u$ is constant. 
	\end{thm}
	
	Actually, under the assumption of boundedness, the Liouville domain that we have proved below is larger than the union $q\ge\frac{3}{2}$ and the Liouville domain without boundedness assumption. 
	Specifically, it is given by the blue grid area in the figure below. Here, we mention that the following blue curve $I$ is same as curve $I$ in Figure \ref{F02}.

	\begin{figure}[!htb]
		
	% Pattern Info
	
	\tikzset{
		pattern size/.store in=\mcSize, 
		pattern size = 5pt,
		pattern thickness/.store in=\mcThickness, 
		pattern thickness = 0.3pt,
		pattern radius/.store in=\mcRadius, 
		pattern radius = 1pt}
	\makeatletter
	\pgfutil@ifundefined{pgf@pattern@name@_c3vk5ey3o}{
		\pgfdeclarepatternformonly[\mcThickness,\mcSize]{_c3vk5ey3o}
		{\pgfqpoint{0pt}{0pt}}
		{\pgfpoint{\mcSize}{\mcSize}}
		{\pgfpoint{\mcSize}{\mcSize}}
		{
			\pgfsetcolor{\tikz@pattern@color}
			\pgfsetlinewidth{\mcThickness}
			\pgfpathmoveto{\pgfqpoint{0pt}{\mcSize}}
			\pgfpathlineto{\pgfpoint{\mcSize+\mcThickness}{-\mcThickness}}
			\pgfpathmoveto{\pgfqpoint{0pt}{0pt}}
			\pgfpathlineto{\pgfpoint{\mcSize+\mcThickness}{\mcSize+\mcThickness}}
			\pgfusepath{stroke}
	}}
	\makeatother
	
	% Pattern Info
	
	\tikzset{
		pattern size/.store in=\mcSize, 
		pattern size = 5pt,
		pattern thickness/.store in=\mcThickness, 
		pattern thickness = 0.3pt,
		pattern radius/.store in=\mcRadius, 
		pattern radius = 1pt}
	\makeatletter
	\pgfutil@ifundefined{pgf@pattern@name@_9o0n8qavt}{
		\pgfdeclarepatternformonly[\mcThickness,\mcSize]{_9o0n8qavt}
		{\pgfqpoint{0pt}{0pt}}
		{\pgfpoint{\mcSize}{\mcSize}}
		{\pgfpoint{\mcSize}{\mcSize}}
		{
			\pgfsetcolor{\tikz@pattern@color}
			\pgfsetlinewidth{\mcThickness}
			\pgfpathmoveto{\pgfqpoint{0pt}{\mcSize}}
			\pgfpathlineto{\pgfpoint{\mcSize+\mcThickness}{-\mcThickness}}
			\pgfpathmoveto{\pgfqpoint{0pt}{0pt}}
			\pgfpathlineto{\pgfpoint{\mcSize+\mcThickness}{\mcSize+\mcThickness}}
			\pgfusepath{stroke}
	}}
	\makeatother
	
	% Pattern Info
	
	\tikzset{
		pattern size/.store in=\mcSize, 
		pattern size = 5pt,
		pattern thickness/.store in=\mcThickness, 
		pattern thickness = 0.3pt,
		pattern radius/.store in=\mcRadius, 
		pattern radius = 1pt}
	\makeatletter
	\pgfutil@ifundefined{pgf@pattern@name@_4mhispr4r}{
		\pgfdeclarepatternformonly[\mcThickness,\mcSize]{_4mhispr4r}
		{\pgfqpoint{0pt}{0pt}}
		{\pgfpoint{\mcSize}{\mcSize}}
		{\pgfpoint{\mcSize}{\mcSize}}
		{
			\pgfsetcolor{\tikz@pattern@color}
			\pgfsetlinewidth{\mcThickness}
			\pgfpathmoveto{\pgfqpoint{0pt}{\mcSize}}
			\pgfpathlineto{\pgfpoint{\mcSize+\mcThickness}{-\mcThickness}}
			\pgfpathmoveto{\pgfqpoint{0pt}{0pt}}
			\pgfpathlineto{\pgfpoint{\mcSize+\mcThickness}{\mcSize+\mcThickness}}
			\pgfusepath{stroke}
	}}
	\makeatother
	\tikzset{every picture/.style={line width=0.75pt}} %set default line width to 0.75pt        
	
	\begin{tikzpicture}[x=0.75pt,y=0.75pt,yscale=-1,xscale=1]
		%uncomment if require: \path (0,402); %set diagram left start at 0, and has height of 402
		
		%Straight Lines [id:da702474492441356] 
		\draw [color={rgb, 255:red, 74; green, 144; blue, 226 }  ,draw opacity=1 ] [dash pattern={on 4.5pt off 4.5pt}]  (75.67,208.33) -- (593.67,208.33) ;
		%Straight Lines [id:da8869015167418075] 
		\draw [color={rgb, 255:red, 74; green, 144; blue, 226 }  ,draw opacity=1 ] [dash pattern={on 4.5pt off 4.5pt}]  (75.67,108) -- (255.67,107.66) -- (598.67,107) ;
		%Straight Lines [id:da2330616020906333] 
		\draw [color={rgb, 255:red, 208; green, 2; blue, 27 }  ,draw opacity=1 ][fill={rgb, 255:red, 58; green, 96; blue, 140 }  ,fill opacity=1 ][line width=1.5]  [dash pattern={on 5.63pt off 4.5pt}]  (172.67,81.43) -- (399.67,308) ;
		%Curve Lines [id:da27461560020913023] 
		\draw [color={rgb, 255:red, 208; green, 2; blue, 27 }  ,draw opacity=1 ]   (399.67,308) .. controls (361.67,269) and (345.67,243.67) .. (380.67,233) .. controls (415.67,222.33) and (516.67,222.67) .. (579.67,217.67) ;
		%Shape: Axis 2D [id:dp05912285427311237] 
		\draw  (61.67,308) -- (585.67,308)(199.67,59) -- (199.67,345.67) (578.67,303) -- (585.67,308) -- (578.67,313) (194.67,66) -- (199.67,59) -- (204.67,66) (299.67,303) -- (299.67,313)(399.67,303) -- (399.67,313)(499.67,303) -- (499.67,313)(99.67,303) -- (99.67,313)(194.67,208) -- (204.67,208)(194.67,108) -- (204.67,108) ;
		\draw   (306.67,320) node[anchor=east, scale=0.75]{1} (406.67,320) node[anchor=east, scale=0.75]{2} (506.67,320) node[anchor=east, scale=0.75]{3} (106.67,320) node[anchor=east, scale=0.75]{-1} (196.67,208) node[anchor=east, scale=0.75]{1} (196.67,108) node[anchor=east, scale=0.75]{2} ;
		%Curve Lines [id:da7416491455861269] 
		\draw [color={rgb, 255:red, 74; green, 144; blue, 226 }  ,draw opacity=1 ][line width=1.5]  [dash pattern={on 5.63pt off 4.5pt}]  (359.67,251) .. controls (375.67,201.09) and (382.67,187.09) .. (593.67,168.67) ;
		%Straight Lines [id:da7562448323715267] 
		\draw [color={rgb, 255:red, 74; green, 144; blue, 226 }  ,draw opacity=1 ][line width=2.25]  [dash pattern={on 6.75pt off 4.5pt}]  (65.67,157.09) -- (608.67,157.43) ;
		%Shape: Polygon [id:ds4001149060934861] 
		\draw  [color={rgb, 255:red, 74; green, 144; blue, 226 }  ,draw opacity=1 ][pattern=_c3vk5ey3o,pattern size=9pt,pattern thickness=0.75pt,pattern radius=0pt, pattern color={rgb, 255:red, 74; green, 144; blue, 226}][dash pattern={on 0.84pt off 2.51pt}] (199.67,107.67) -- (399.67,308) -- (61.67,308.67) -- (75.67,308.67) -- (75.67,195.67) -- (71.67,81.09) -- (172.67,81.43) -- cycle ;
		%Shape: Polygon Curved [id:ds008587548730889072] 
		\draw  [color={rgb, 255:red, 74; green, 144; blue, 226 }  ,draw opacity=1 ][pattern=_9o0n8qavt,pattern size=9pt,pattern thickness=0.75pt,pattern radius=0pt, pattern color={rgb, 255:red, 74; green, 144; blue, 226}][dash pattern={on 0.84pt off 2.51pt}] (199.67,107.67) .. controls (199.67,106.76) and (582.67,106.76) .. (582.67,107.76) .. controls (582.67,108.76) and (582.67,172.43) .. (582.67,171.09) .. controls (582.67,169.76) and (445.42,182.29) .. (407.67,196.09) .. controls (369.92,209.9) and (365.67,229.09) .. (361.67,242.09) .. controls (357.67,255.09) and (361.67,260.76) .. (366.67,268.76) .. controls (371.67,276.76) and (378.67,284.76) .. (381.67,288.76) .. controls (384.67,292.76) and (399.67,308.06) .. (399.67,308) .. controls (399.67,307.94) and (199.67,108.57) .. (199.67,107.67) -- cycle ;
		%Shape: Polygon [id:ds3314705206261752] 
		\draw  [pattern=_4mhispr4r,pattern size=9pt,pattern thickness=0.75pt,pattern radius=0pt, pattern color={rgb, 255:red, 74; green, 144; blue, 226}][dash pattern={on 0.84pt off 2.51pt}] (172.67,81.43) -- (199.67,107.67) -- (582.67,107.76) -- (581.67,80.09) -- (172.67,81.43) -- cycle ;
		
		% Text Node
		\draw (193,40) node [anchor=north west][inner sep=0.75pt]   [align=left] {q};
		% Text Node
		\draw (574,314) node [anchor=north west][inner sep=0.75pt]   [align=left] {p};
		% Text Node
		\draw (179,310.4) node [anchor=north west][inner sep=0.75pt]    {$O$};
		% Text Node
		\draw (538,175.4) node [anchor=north west][inner sep=0.75pt]    {$I$};
		% Text Node
		\draw (539,137.4) node [anchor=north west][inner sep=0.75pt]    {$II$};
		% Text Node
		\draw (583.67,129.73) node [anchor=north west][inner sep=0.75pt]  [font=\small]  {$q=\frac{3}{2}$};
		% Text Node
		\draw (541,95.07) node [anchor=north west][inner sep=0.75pt]    {$III$};
		% Text Node
		\draw (532,225.65) node [anchor=north west][inner sep=0.75pt]    {$V$};
		% Text Node
		\draw (583.67,83.49) node [anchor=north west][inner sep=0.75pt]    {$q=2$};
		% Text Node
		\draw (155,65.4) node [anchor=north west][inner sep=0.75pt]    {$IV$};

	\end{tikzpicture}

\caption{Liouville domain of equation \eqref{pqle} with boundedness assumption ($n=6$)}\label{F03}
\end{figure}
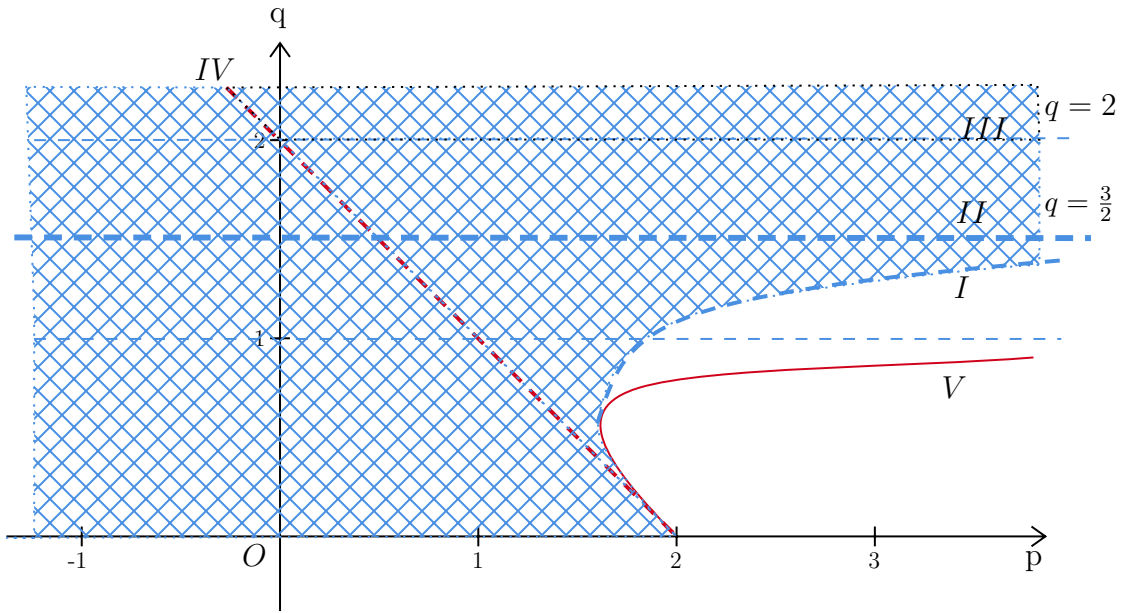

	For the clear description of  Liouville domain in the Figure \ref{F02} and Figure \ref{F03} respectively, we give the following definition of admissible set.
	\begin{definition}[Admissible set]\label{set}
		Let $n\ge 3$, we define two regions. The first one is $$\mathbb{A}(n)=\left\{(p,q)\in \mathbb{R}^2_{+}:p+q-1<\mathcal{L}(n,q), \,\, 1-\frac{1}{\sqrt{n-1}}< q< \frac{3}{2}\right\},$$
		where
		\begin{equation}
			\mathcal{L}(n,q)=\sup\left\{\frac{2-q}{n-2}+\frac{y(n-2)+2}{(n-2)(2-y)}:y\in\mathbb{D}(n,q)\right\},\nonumber
		\end{equation}
		\begin{eqnarray}
			\mathbb{D}(n,q):=\{y\in(0,2):	y \,\,\text{satisfies}  \,\,\eqref{12} \},\nonumber
		\end{eqnarray}
		
		\begin{equation}\label{12}
			\left\{
			\begin{aligned}
				&x = \frac{2 - ny}{2(n-1)},\quad 2y \left( \frac{1}{\sqrt{n}+1} + x \right) + qy^2 > 0, \\
				&\frac{2(x+1)^2}{n} + y(q - 2x - 1) - \frac{qy^2}{2} - 2x^2 > 0, \\
				&4 \left( \frac{2(x+1)^2}{n} + y(q - 2x - 1) - \frac{qy^2}{2} - 2x^2 \right) \\
				&+ 2y \left( \frac{1}{\sqrt{n}+1} + x \right) + (q-1)y^2 > 0.
			\end{aligned}
			\right.
		\end{equation}
	The second one is
		$$\mathbb{B}(n)=\left\{(p,q)\in \mathbb{R}^2_{+}:p+q-1<\mathcal{H}(n,q), \,\, 1< q< \frac{5}{3}\right\},$$
		where
		\begin{equation}\label{l0}
			\mathcal{H}(n,q)=\sup\left\{\frac{2-q}{n-2}+\frac{y(n-2)+2}{(n-2)(2-y)}:y\in\mathbb{E}(n,q)\right\},\nonumber
		\end{equation}
		\begin{eqnarray}
			\mathbb{E}(n,q):=\{y\in(0,2):	y \,\,\text{satisfies}  \,\,\eqref{13} \},\nonumber
		\end{eqnarray}
		
		\begin{equation}\label{13}
			\left\{
			\begin{aligned}
				&x = \frac{2 - ny}{2(n-1)},\quad 2y \left( \frac{1}{\sqrt{n}+1} + x \right) + qy^2 > 0, \\
				&\frac{2(x+1)^2}{n} + y(q - 2x - 1) - \frac{qy^2}{2} - 2x^2 > 0, \\
				&4 \left( \frac{2(x+1)^2}{n} + y(q - 2x - 1) - \frac{qy^2}{2} - 2x^2 \right) \\
				&+ 2y \left( \frac{1}{\sqrt{n}+1} + x \right) + \left(\frac{3}{2}q-2\right)y^2 > 0.
			\end{aligned}
			\right.
		\end{equation}
	Then we call the following sets as admissible set ($\mathscr{D}(n)$ is given by \eqref{NL}):

	$$
	\mathbb{BL}(n)=\left\{(p,q)\in \mathbb{R}^2_{+},q\ge\frac{3}{2}\right\}\bigcup \mathbb{A}(n)\bigcup\left\{(p,q)\in \mathbb{R}^2_{+}\setminus \mathscr{D}(n):q\in[0,1-\frac{1}{\sqrt{n-1}}]\right\};
	$$
	
		$$
	\mathbb{L}(n)=\left\{(p,q)\in \mathbb{R}^2_{+},q\ge\frac{5}{3}\right\}\bigcup \mathbb{B}(n)\bigcup\left(	\mathbb{BL}(n)\cap (\R\times[0,1])\right).
	$$
	\end{definition}

	Then $\mathbb{L}(n)$ and $\mathbb{BL}(n)$ are the Liouville domains that what we are going to prove below without and with boundedness assumption respectively. Concretely, we have %In addition, our approach of deriving Liouville theorems can be  extended to the Riemannian setting without extra effort.	We collect them as follows. $(\M,g)$ be an n-dimensional Riemannian manifold with $Ric\ge 0$ and

	\begin{thm}\label{rm1}
		Let  $u$ be a nonnegative classical solution of \eqref{pqle} on $\mathbb{R}^n$ with $q\ge 0$ and $p\in \mathbb{R}$. The admissible set $\mathbb{L}(n)$ and $\mathbb{BL}(n)$ are given by Definition \ref{set}.  If one of the following two conditions is satisfied:\\
		$(1)$  $n=2$ or $(p,q)\in \mathbb{L}(n) $ with $n\ge 3$;\\
		$(2)$ $u$ is bounded and $(p,q)\in \mathbb{BL}(n) $ with $n\ge 3$,\\
		then $u$ is constant. 
	\end{thm}
	
	We have actually established the above Liouville theorem for the more general case of Riemannian manifolds.\emph{ An important point we would like to emphasize is that our method of deriving Theorem \ref{rm1} is completely different from the existing methods. In the regions where it overlaps with each of the previous results, we have adopted different approaches and techniques.} That is to say, the proof of Theorem \ref{rm1} is self-contained. For the detailed framework of the proof of Theorem \ref{rm1}, the involved techniques, and the generalization to Riemannian setting, the readers can refer to Section \ref{framework}.

	\vspace{2mm}

	Except Liouville theorem, we derive the local gradient type estimates to equation \eqref{pqle} with all indices where the Liouville theorem holds  (i.e. $(p,q)\in \mathscr{D}_L(n)$).

	\begin{thm}\label{LFG}
		Let $\O\subsetneq \mathbb{R}^n$ be a domain and $u$ be a positive solution  of \eqref{pqle} with indices $p,q$ on $\O$. If $(p,q)\in \mathscr{D}_L(n)$, where $\mathscr{D}_L(n)$ is defined by \eqref{LD},  then 
		\begin{eqnarray}\label{lgefromL}
			\left(\frac{|\nabla u|^2}{u^2}+|\nabla u|^qu^{p-1}\right)(x)\le \frac{C(n,p,q)}{{\rm dist}^2(x,\partial\O)} \qquad \text{on\,\, $\O$}.
		\end{eqnarray}
	\end{thm}

	From Theorem \ref{LFG} we immediately have a local gradient type estimate, which provides an explicit value for the constant $a$ in \cite[Theorem B]{BGV}.

	\begin{thm}\label{geofpqle}
		Let $\O\subsetneq \mathbb{R}^n$ be a domain and $u$ be a positive solution  of \eqref{pqle} with indices $p,q$ on $\O$. If $(p,q)\in \mathscr{D}_L(n)$, where $\mathscr{D}_L(n)$ is defined by \eqref{LD},  then  there exists $a=a(p,q)\in\mathbb{R}$ such that
		\begin{eqnarray}\label{gte}
			|\nabla u^a|(x)\le C(n,p,q)\left({\rm dist}(x,\partial\O)\right)^{-1-a\frac{2-q}{p+q-1}} \qquad \text{on\,\, $\O$}.
		\end{eqnarray}
		Especially, if $p+q>1$ and $q>0$, then we can choose $a=\frac{p+q-1}{q}>0$; further, if $p>1$, we can choose $a=1$.
	\end{thm}

	The Harnack inequality of equation \eqref{pqle}
	is a direct corollary of  logarithmic gradient estimate \eqref{lgefromL}.

	\begin{thm}\label{hkofpqle}
		Let  $u$ be a positive solution  of \eqref{pqle} on $B(x_0,2R)\subset\mathbb{R}^n$ with indices $p,q$. If $(p,q)\in \mathscr{D}_L(n)$,  where $\mathscr{D}_L(n)$ is defined by \eqref{LD},   then 
		\begin{eqnarray}\label{HK0}
			\sup\limits_{B(x_0,R)}u\le C(n,p,q)\inf\limits_{B(x_0,R)} u.
		\end{eqnarray}
	\end{thm}
	
	Combining Theorem \ref{m1}, Theorem \ref{LFG}, Theorem \ref{geofpqle} and Theorem \ref{hkofpqle}, we have 
	
	\begin{thm}\label{zong}
	Let $u$ be a nonnegative classical solution of \eqref{pqle} on $\O\subsetneq\mathbb{R}^n$ with $q\ge 0$ and $p\in \mathbb{R}$. If $n=2$ or $(p,q)\in \mathbb{L}(n) $ with $n\ge 3$, 
	then $u$ satisfies the gradient type estimates \eqref{lgefromL}, \eqref{gte} and Harnack inequality \eqref{HK0}. 
	\end{thm}

	Before we go further, we give a remark for Theorem \ref{LFG}--Theorem \ref{zong}.
	\begin{rmk}
	\rm	(A) Theorem \ref{LFG} establishes the equivalence between the Liouville theorem for the equation \eqref{pqle} and local gradient estimates.  Gidas-Spruck \cite{GS} and Poláčik-Quittner-Souplet \cite{PQS} also provide such philosophy for semilinear case. However, for equation \eqref{pqle} in $q>0$ case,  such connection has not been derived for equation \eqref{pqle} by classical blow-up arguments due to the lack of a priori $L^{\infty}$ estimates of the blow-up sequences.\\ %Concretely, one can see more details in subsection \ref{ss3}.
	(B) Theorem \ref{zong} essentially strengthened the previous results. The best known result for gradient estimate is due to \cite[Theorem B]{BGV}, which derives the inequality \eqref{gte} for $(p,q)$ between the curve $III_1$ and $III_2$ in Figure \ref{F1}\\% When $p+q<\frac{n+3}{n-1}$ (for $p+q>1$, this region is slightly less than the region between curve $III_1$ and $III_2$), via Moser's iteration method, \cite{HHW} also gives the logarithmic gradient estimate \eqref{lgefromL}.  \\
		(C) when $\frac{2-q}{n-2}<p+q-1$, $q\in[0,2)$ and $(p,q)\in\mathscr{D}_L(n)$, equation \eqref{pqle} has solution of the form $u(x)=\Lambda |x|^{\frac{2-q}{p+q-1}}$ in $\mathbb{R}^n\setminus\{0\}$, where $\Lambda=\Lambda(n,p,q)>0$. In this case, estimates \eqref{lgefromL} and \eqref{gte} are optimal ($\O$ is chose as $\mathbb{R}^n\setminus\{0\}$).
	\end{rmk}

%By direct computation, if $q>2-\frac{1}{\sqrt{n}+1}$, then $x=2$ satisfies \eqref{712} and \eqref{713} and so there exists $\delta>0$ such that $(2-\delta,2)\subset \mathbb{D}(n,q)$. Therefore, $l_0=\infty$.

%Now, we present the estimates for $1-\frac{1}{\sqrt{n-1}}\le q< 2$. Due to limitations in our proof process, we divide it into two situations: $q\le 1$ and $q\in(1,2)$. When $q\in [1-\frac{1}{\sqrt{n-1}},1]$, we have same estimate as Theorem \ref{kt}. When $q\in(1,2)$, we have a weaker estimate. 

%	\left\{(p,q)\in\mathbb{R}_{+}^2|\text{any positive classical solutions of \eqref{pqle} on  $\mathbb{R}^n$ must be constant}\right\}

\subsection{New contributions and ideas}\label{ss3}
	In this subsection, we will elaborate on the main contributions and ideas of this article.
	
	$\bullet$ \textbf{Bochner inequality is not optimal and its remedy.} For exploring the local and global properties of Lane-Emden equation ($q=0$ in \eqref{pqle}) or Yamabe equation ($\Delta u+\lambda u+\mu u^{\frac{n+2}{n-2}}=0$) on manifolds or Euclidean spaces, many analysts and differential geometers provided the following transformation: if $u$ is a positive $C^3$ function on manifolds, define
	\begin{equation}
		v=u^{-\b}, \qquad\b\neq 0\nonumber
	\end{equation}
	and use the Bochner formula 
	$$
\frac{1}{2}	\Delta |\nabla v|^2=|\nabla^2 v|^2+\left\langle\nabla v,\nabla \Delta v\right\rangle+Ric(\nabla v,\nabla v)
	$$
	for the new function $v$, such as Obata \cite{O}, Gidas--Spruck \cite{GS}, Bidaut-Véron--Véron \cite{BV} and Serrin--Zou \cite{SZ}. One major advantage of this transformation is that it can be applied to equations on manifolds and even metric measure spaces.  Especially, when considering subcritical or critical problems of Lane-Emden equation, one usually uses the basic inequality $|\nabla^2 v|^2\ge\frac{1}{n}(\Delta v)^2$ ($n$ is the dimension of manifolds) and derives the useful Bochner inequality
		$$
	\frac{1}{2}	\Delta |\nabla v|^2\ge \frac{1}{n}(\Delta v)^2+\left\langle\nabla v,\nabla \Delta v\right\rangle+Ric(\nabla v,\nabla v).
	$$
	In fact, the discarded item in Bochner inequality is
	$$
	\left|\nabla^2 v-\frac{\Delta v}{n}g\right|^2,
	$$
	where $g$ is the Riemannian metric of the manifolds. We express this item in the form of u as follows:
	\begin{equation}\label{di}
		\b^2u^{-2\b}\left|\textbf{X}_{\b}-\frac{tr(\textbf{X}_{\b})}{n}g\right|^2,
	\end{equation}
	where $$\textbf{X}_{\b}:=\frac{\nabla^2 u}{u}-(1+\b)d\ln u\otimes d\ln u,$$
	$d\ln u$ is the exterior differential of $\ln u$ and $tr(\textbf{X}_{\b})$ is the trace of tensor $\textbf{X}_{\b}$.

	Now, we restrict the manifolds to Euclidean spaces. Notice that  a standard global solution of critical Lane-Emden equation \eqref{pqle} with $q=0$ and $p=\frac{n+2}{n-2}$ is ($n\ge 3$)
	$$
	u(x)=\left[\frac{ \sqrt{n(n-2)}}{1+\left|x\right|^2}\right]^{\frac{n-2}{2}},\qquad x\in\mathbb{R}^n .
	$$
	Substituting such a solution to the formula \eqref{di} and choosing $\b=\frac{2}{n-2}$, one will find that the discarded item vanishes. This inspires us when studying critical Lane-Emden equation on $\mathbb{R}^n$ with $n\ge 3$, Bochner inequality is optimal because the discarded item vanishes. Simultaneously, when we consider subcritical Lane-Emden equation, this fact inspires us that if we choose $\b$ sufficiently close to $\frac{2}{n-2}$, such discarded item should be almost optimal.
	
	Then we explore the equation \eqref{pqle} in $q>0$ case. When $q\in(0,1)$, the critical condition of \eqref{pqle} becomes the curve $$p+q-1=\frac{(2-q)^2}{(1-q)(n-2)}.$$
	 By \cite[Theorem D]{BGV}, we have a global solution of \eqref{pqle} as follows ($n\ge 3$):
	\begin{equation}\label{ss}
		u(x)=\left(K +|x|^{\frac{2-q}{1-q}}\right)^{-\frac{(N-2)(1-q)}{2-q}},\qquad x\in\mathbb{R}^n ,
	\end{equation}
	where $K=K(n,q)>0$. If we substitute the solution to the discarded item \eqref{di}, then we find that there is no $\b\neq 0$ that makes this expression equal to zero. This key fact inspires us that \emph{Bochner inequality (with possible transformation $v=u^{-\b}$) cannot be optimal at critical case!} Therefore, we must return to Bochner formula and find new way to deriving optimal inequality for equation \eqref{pqle}.
	
	After some experimental calculations, we find the following singular second-order symmetric covariant tensor with two parameters:
	$$
	\T(\b,\sigma)= \frac{\nabla^2 u}{u}-[(1+\beta)+\sigma Z] d \ln u \otimes d \ln u,
	$$
	where $Z:=|\nabla u|^{q-2}u^{p+q+1}$ defines on the open set $\{x\in \M:|\nabla u|>0\}$.
	
	If we choose $u$ as \eqref{ss} ($\M=\mathbb{R}^n$), $\b=\frac{2}{n-2}$ and $\sigma=\frac{-q}{n-(n-1)q}$ with $q\in(0,1)$, then the nonnegative item
	$$
	\left|\T(\b,\sigma)-\frac{tr(\T(\b,\sigma))}{n}  g\right|^2
	$$
	vanishes on  $\mathbb{R}^n\setminus\{0\}$, where $tr(\T(\b,\sigma))$ is the trace of tensor $\T(\b,\sigma)$. This important fact inspires us that after we use Bochner formula, we should discard such item to ensure the optimality of the inequality. This observation is key to proving the optimal Liouville theorem for equation \eqref{pqle}. Concretely, we give its details in Section \ref{S2}.

	$\bullet$ \textbf{Combine Bernstein method and blow up arguments.}
	For deriving the local estimate of nonlinear  elliptic equations on the domain of Euclidean spaces, analysts usually use the blow up arguments. Generally speaking, one assumes the expected local estimates fail and yields a sequence of functions, and finally finds a subsequence that converges  a global solution of certain elliptic equation (in some sense) which contradicts with some known Liouville theorems of such equation. For positive solutions of Lane-Emden equation, via blow up arguments, Gidas-Spruck \cite{GS0} and Dancer \cite{D}  obtained its universal bounds on bounded domain with Dirichlet boundary condition and without boundary condition respectively.  Afterwards, Poláčik-Quittner-Souplet  found an important doubling lemma \cite[Lemma 5.1]{PQS}, and based on it, they systematically derive universal gradient estimate and $L^{\infty}$ estimate for nonlinear elliptic equations and  systems. However, this classic method has not been successfully applied to equation \eqref{pqle} with $q>0$ yet. First, we recall the classical results and arguments for Lane-Emden equation.%\footnote{Such estimates also can be obtained by integral estimate or direct Bernstein method, see Serrin-Zou \cite{SZ} or Lu \cite{LU1,LU2}. For $\M=\mathbb{R}^n$, via blow up arguments, in the present paper, we give another proof of \cite[Theorem 1.6]{LU1}, which provides the local logarithmic gradient estimate and $L^{\infty}$ estimate for equation \eqref{pqle} with  $q=0$ and $p<p_S(n)$, where $p_S(n)=\infty$ if $n\le 2$ and $p_S(n)=\frac{n+2}{n-2}$ if $n\ge 3$.}.

	\begin{theorem}\cite[Theorem 2.3]{PQS}
	Let $1<p<p_S(n)$, and let $\Omega \neq \mathbb{R}^n$ be a domain of $\mathbb{R}^n$. There exists $C=C(n, p)>0$ such that any nonnegative solution $u$ of \eqref{pqle} with $q=0$ in $\Omega$ satisfies
\begin{equation}\label{uefle}
	u(x)+|\nabla u|^{2 /(p+1)}(x) \leq C \operatorname{dist}^{-2 /(p-1)}(x, \partial \Omega), \quad x \in \Omega.
\end{equation}	
	\end{theorem}
	
	\begin{proof}
		For proving estimate \eqref{uefle}, we assume it fails. Then there exist sequences $\Omega_k, u_k, y_k \in \Omega_k$ such that $u_k$ solves \eqref{pqle} (with $q=0$) on $\Omega_k$ and the functions
		
		$$
		M_k:=u_k^{(p-1) / 2}+\left|\nabla u_k\right|^{(p-1) /(p+1)}, \quad k=1,2, \cdots
		$$
		satisfy
		
		$$
		M_k\left(y_k\right)>2 k \operatorname{dist}^{-1}\left(y_k, \partial \Omega_k\right) .
		$$
		By the doubling lemma \cite[Lemma 5.1]{PQS}, it follows that there exists $x_k \in \Omega_k$ such that
		
		$$
		M_k\left(x_k\right) \geq M_k\left(y_k\right), \quad M_k\left(x_k\right)>2 k \operatorname{dist}^{-1}\left(x_k, \partial \Omega_k\right)
		$$
		and
		
		$$
		M_k(z) \leq 2 M_k\left(x_k\right), \quad\left|z-x_k\right| \leq k M_k^{-1}\left(x_k\right)
		$$
		Now, we rescale $u_k$ by setting
		
		$$
		v_k(y):=\lambda_k^{2 /(p-1)} u_k\left(x_k+\lambda_k y\right), \quad|y| \leq k, \text { with } \lambda_k=M_k^{-1}\left(x_k\right) .
		$$
		The function $v_k$ solves \eqref{pqle} with $q=0$ on the ball $B(0,k)$ and they also satisfy
		\begin{equation}\label{ei}
			\left[v_k^{(p-1) / 2}+\left|\nabla v_k\right|^{(p-1) /(p+1)}\right](0)=\lambda_k M_k\left(x_k\right)=1
		\end{equation}
		and
		\begin{equation}\label{ue}
			\left[v_k^{(p-1) / 2}+\left|\nabla v_k\right|^{(p-1) /(p+1)}\right]_{L^{\infty}(B(0,k))} \leq 2.
		\end{equation}
		
		By using elliptic $W^{2,p}$-estimates  and Sobolev imbeddings (see \cite{GT}), we deduce that some subsequence of $v_k$ converges in $C_{\text {loc }}^1\left(\mathbb{R}^n\right)$ to a classical solution $v \geq 0$ of Lane-Emden equation in $\mathbb{R}^n$. Moreover, $\left[v^{(p-1) / 2}+|\nabla v|^{(p-1) /(p+1)}\right](0)=1$ by \eqref{ei}, so that $v$ is nontrivial. This contradicts the Gidas-Spruck's classical Liouville theorem \cite[Theorem 1.1]{GS} and so proves the desired Theorem.
	\end{proof}

	 Then, we dissect this exquisite proof and take a look at where the difficulties lie when we consider \eqref{pqle} with $q>0$.
	 First, the blow up quantity $M(u)$ should have the following good properties:
	 \begin{spacing}{1.1}
	 \noindent	(i) it is (directly or indirectly) related to our desired estimates;\\
	 	(ii) it is rescaling invarint such that identity \eqref{ei} and inequality \eqref{ue} hold at same time;\\
	 	(iii) inequality \eqref{ue} provides the ``compactness" of blow up sequence and  identity \eqref{ei} provides ``nontrivialness" of ``all possible limit states".
	 \end{spacing}

	 Notice that equation \eqref{pqle} still is invariant by rescaling or translation. And we can construct rescaling-invariant blow up quantity (assume $p+q-1>0$ and $q\in(0,2)$)
	 $$
	 M(u)=u^{\frac{p+q-1}{2-q}}+|\nabla u|^{\frac{p+1}{p+q-1}}.
	 $$
	 Then we can check that all the processes in the proof just now are also valid for present case. But, for the limit function $v$,   the identity $$M(v)(0)=\left(v^{\frac{p+q-1}{2-q}}+|\nabla v|^{\frac{p+1}{p+q-1}}\right)(0)=1$$ cannot derive the ``nontrivialness" of $v$, which 
	 causes the entire proof to collapse. In fact, due to the existence of constant solutions, we cannot expect universal $L^{\infty}$ estimate for positive solutions. Therefore, for using blow up arguments, we only construct rescaling invariant quantity related to gradient of solutions (and also solutions). However, if we choose blow up quantity $M(u)$ only related gradient of solutions (such as $\sum_{i=1}^{k}u^{\a_i}|\nabla u|^{\g_i}$ with $\g_i>0$ and $k\ge 1$), ``compactness" of blow up sequence is difficult to obtain because of lack of local $L^{\infty}$ estimate of the sequence.

	  Now, we provide our solution for it. Inspired by \cite[Theorem 1.6]{LU1}, we construct the following rescaling invariant blow up function 
	$$
	F(u)=\frac{|\nabla u|}{u}+|\nabla u|^{\frac{q}{2}}u^{\frac{p-1}{2}} .
	$$
	One heuristic reason is that when $q=0$ and $p<p_S(n)$, by \cite[Theorem 1.6]{LU1},
	 $$F(u)(x)\le C(n,p){\rm dist}^{-1}(x,\partial\O) \quad x \in \Omega.$$  Another more important reason is that the uniform boundedness of function sequences $|\nabla\ln v_k|$ ($v_k$ is the rescaling of  $u_k$) automatically  provide the important Harnack inequality. Next, we just need to prove the ``compactness'' of real number sequence  $\{v_k(0)\}$. Concretely speaking, it simultaneously has positive upper and lower bounds. For proving the key point, we find that for any solution $u$ of \eqref{pqle} with $p\in\mathbb{R}$ and $q\ge 0$, the items
	 $$|\nabla\ln u|\qquad and\qquad |\nabla u|^{\frac{q}{2}}u^{\frac{p-1}{2}}$$
	  can be  controlled by each other. We summarize this vital fact in Section \ref{S6} and call them mutual control lemmas.
	 Via the mutual control lemmas,  we exclude  the possibility of $\liminf\limits_{k\to \infty}|\nabla\ln v_k|(0)=0$ or $\liminf\limits_{k\to \infty}|\nabla v_k|^{\frac{q}{2}}v_k^{\frac{p-1}{2}}(0)=0$ and derive that the limits of any subsequence must be positive. Together with the Harnark inequality, this proves that there exists a subsequence of $\{v_k\}$ $C^1_{loc}$ converges to a global positive solution $v$ of \eqref{pqle}. Finally, the identity 
	$$
	\left(\frac{|\nabla v|}{v}+|\nabla v|^{\frac{q}{2}}v^{\frac{p-1}{2}}\right)(0)=1
	$$
	yields the ``nontrivialness" of solution and this contradicts with the (weaker) Liouville theorem. Finally, we derive the desired estimate.
	 %In fact, for the contradiction, we just need a weaker Liouville theorem. Concretely, we give related details in Section \ref{S7}.  

	%the standard cut-off functions on Riemannian manifolds with Ricci curvature bounded below and the fundamental lemma 

	\subsection{Organization of the paper} The rest of this paper is organized as follows. In Section \ref{framework}, we provide the framework of proving Liouville theorem and its generalization. In Section \ref{S2},  we provide a elliptic equations of the modulus square of logarithmic gradient of the positive solutions with two parameters. In addition to deriving the $L^{\infty}$ estimates of logarithmic gradient in Section \ref{S3}, it serves as a foundation of obtaining elliptic equations of more complex auxiliary functions related to solutions in Section \ref{S4}.  In Section \ref{S3}, using Lemma \ref{l2}, we prove the logarithmic gradient estimate part in Theorem \ref{t10} and Liouville theorem for $q>2$ case.  In Section \ref{S4},  we derive the elliptic equations for more general auxiliary functions related solutions and its gradients.
	 In Section \ref{S5}, using the elliptic equations in Section \ref{S4}, we have proven the estimates in Section \ref{framework}, thus completing the proof of Liouville theorem. In the second part of this article, which is composed by Section \ref{S6} and Section \ref{S7}, we provide a new method to study the local gradient type  estimates  of the positive solutions of \eqref{pqle} on the domains of Euclidean spaces. Based on Bernstein method ,  we obtain the important mutual control lemmas in Section \ref{S6}. In Section \ref{S7} we  derive the local estimates via the blow up argument and the (weaker) Liouville theorem. Last, in appendix, we collect the purely algebraic calculations involved in the auxiliary functions and compare our results with the previously known results.

	\section{The framework of Liouville theorems}\label{framework}
	In this section, we list the estimates required for proving Liouville theorem. In fact, we have established these estimates for equation \eqref{pqle} on Riemannian manifolds under Ricci curvature bounded below setting.

	First, we generalize Theorem \ref{rm1} to Riemannian manifolds with nonnegative Ricci curvature.
		\begin{thm}\label{lrm}
		Let  $u$ be a non-negative classical solution of \eqref{pqle} on $\M$ with $Ric\ge 0$. The admissible set $\mathbb{L}(n)$ and $\mathbb{BL}(n)$ are given by Definition \ref{set}.  If one of the following two conditions is satisfied:\\
		$(1)$  $n=2$ or $(p,q)\in \mathbb{L}(n) $ with $n\ge 3$;\\
		$(2)$ $u$ is bounded and $(p,q)\in \mathbb{BL}(n) $ with $n\ge 3$,\\
		then $u$ is constant. 
	\end{thm}
	
	When $q=2$ and $p>-1$, Theorem \ref{lrm} is a direct corollary of Liouville theorem for harmonic functions. Let $w=\int_0^u e^{\frac{s^{p+1}}{p+1}}ds$, then $w$ is a positive harmonic function and so a constant by Yau's Liouville theorem (see Yau \cite{Y}). Therefore, $u$ is constant. 	When $q=2$ and $p\le-1$, Theorem \ref{lrm} is implied by the following gradient estimate, which guarantees Liouville theorem when \( p + q <\frac{n+3}{n-1} \).

		\begin{thm}\label{t10}
		Let $(\M,g)$ be an n-dimensional Riemannian manifold with $Ric\ge-Kg$ and $K\ge 0$. Let $u$ be a positive solution of \eqref{pqle} on $B(x_0,2R)$ with $q\ge 0$ and $p+q<\frac{n+3}{n-1}$, then
		\begin{eqnarray}
			\sup\limits_{B(x_0,R)}\left(\frac{|\nabla u|^2}{u^2}+|\nabla u|^qu^{p-1}\right)\le C(n,p,q)\left(\frac{1}{R^2}+K\right).\nonumber
		\end{eqnarray}
	\end{thm}
	
Yau type gradient estimate part in Theorem \ref{t10} was also obtained by He-Hu-Wang \cite{HHW} using the Moser iteration method, but the proof we present below is new.
	From Theorem \ref{t10}, we obtain the weaker gradient estimate of the power of solutions. 
	
	\begin{thm}\label{t20}
		Let $(\M,g)$ be an n-dimensional Riemannian manifold with $Ric\ge-Kg$ and $K\ge 0$. If  $u$ be a positive solution of \eqref{pqle} on $B(x_0,2R)$ with $q>0$ and $p+q<\frac{n+3}{n-1}$, then there exists $a=a(p,q)\in\mathbb{R}$ such that
		\begin{eqnarray}
			\sup\limits_{B(x_0,R)}|\nabla u^a|\le C(n,p,q)\left(\frac{1}{R^2}+K\right)^{\frac{1}{2}+a\frac{2-q}{2(p+q-1)}}.\nonumber
		\end{eqnarray}
		Especially, if $p+q>1$ and $q>0$, then we can choose $a=\frac{p+q-1}{q}>0$; further, if $p>1$, we can choose $a=1$.
	\end{thm}
	When $q>2$, using Bernstein method, we derive a strong control lemma, which It plays a crucial role in proving Theorem \ref{lrm} for the case when \( q > 2 \). 
		\begin{thm}\label{sc}
		Let $(\M,g)$ be an n-dimensional Riemannian manifold with $Ric\ge-Kg$ and $K\ge 0$. Let $u$ be a positive solution of \eqref{pqle} on $B(x_0,2R)$ with $q> 2$ , then
		\begin{eqnarray}
			\sup\limits_{B(x_0,R)}\left(|\nabla u|^qu^{p-1}-\frac{2nq(p+q-1)}{q-2}\frac{|\nabla u|^2}{u^2}\right)\le C(n,p,q)\left(\frac{1}{R^2}+K\right).\nonumber
		\end{eqnarray}
	\end{thm}
From	Theorem \ref{sc}, we have the following global estimate in Ricci nonnegative case.
	\begin{cor}\label{sc0}
			Let $(\M,g)$ be an n-dimensional Riemannian manifold with $Ric\ge 0$. Let $u$ be a positive solution of \eqref{pqle} on $\M$ with $q> 2$ and $p+q\ge 1$, then
			\begin{eqnarray}
				\sup\limits_{\M}|\nabla u|^{q-2}u^{p+1}\le \frac{2nq(p+q-1)}{q-2}.\nonumber
			\end{eqnarray}
	\end{cor}
	The similar result as Corollary \ref{sc0} in the Euclidean case was obtained in Bidaut-Véron \cite{B} by Keller--Osserman comparison method and bootstrap method. The difference in our estimates is that the upper bound is computable, while hers is qualitative.
Starting from Corollary \ref{sc0}, as same as in \cite{B}, we can derive Theorem \ref{lrm} for the case when \( q > 2 \) by utilizing the properties of super-harmonic and sub-harmonic functions.

	\vspace{2mm}

	Next, we focus on various gradient estimates for the case when \( 0\le q < 2 \). In the three theorems below, we assume \(0 \leq q < 2\) by default.
	For the two-dimensional case, we derive a gradient-type estimate that imposes no restrictions on \(p\).
	\begin{thm}\label{2d}
		Let $(\mathcal{M}^2,g)$ be a 2-dimensional Riemannian manifold with $Ric\ge 0$. Assume that $u$ be a positive solution of \eqref{pqle} on $B(x_0,2R)$ with $q\in[0,2)$ and $p\in\mathbb{R}$. Then for any $0<r\le R$, there exists $\b=\b(p,q)\ge 0$ such that
		\begin{equation}\label{2de}
			\sup\limits_{B(x_0,R)}u^{-\b}|\nabla\ln u|^{2}\le \frac{C(p,q)\left(\min\limits_{\partial B(x_0,r)}u\right)^{-\b}}{Rr}.\nonumber
		\end{equation}
	\end{thm}

In the theorems below, we assume \( n \geq 3 \) and \( p + q \geq \frac{n+3}{n-1} \geq \frac{n}{n-2} \) (these results are automatically implied by Theorem \ref{t10} otherwise). Under this condition, we derive the gradient-type estimate for equation \eqref{pqle}.
	
	\begin{thm}\label{t27}
		Let $(\M,g)$ be an n-dimensional Riemannian manifold with $Ric\ge 0$. Assume that $u$ be a positive solution of \eqref{pqle} on $B(x_0,2R)$ with $(p,q)\in  \mathbb{L}(n)$ and  $n\ge 3$,  where $\mathbb{L}(n)$ is given by Definition \ref{set}.
		Then for $0<r\le R$, there exists a $\b=\b(n,p,q)\in[0,\frac{2}{n-2})$  such that
		\begin{eqnarray}\label{17}
			\sup\limits_{B(x_0,R)}u^{-(1-\frac{q}{2})\b}|\nabla\ln u|^{2-q}\le C\left(R^{-(2-(n-2)\b)(1-\frac{q}{2})}\left(r^{n-2}\min_{\partial B(x_0,r)}u\right)^{-\left(1-\frac{q}{2}\right)\b}+R^{-\eta}\right),\nonumber
		\end{eqnarray}
		where $C=C(n,p,q)$ and $\eta=\frac{2-q}{p+q-1}\left(p+q-1-\left(1-\frac{q}{2}\right)\b\right)$.
	\end{thm}
	%\begin{rmk}
	%	When $n\ge 3$, $\e_1=2-(n-2)\b$ with $\b\in[0,\frac{2}{n-2})$ and $\a(n)=n-2$. 
	%\end{rmk}

	For $(p,q)\in \mathbb{BL}(n)$, then we have the following weaker gradient estimate, which implies the Liouville theorem for bounded solutions. 
	\begin{thm}\label{t28}
		Let $(\M,g)$ ($n\ge 3$) be an n-dimensional Riemannian manifold with $Ric\ge 0$. If  $u$ be a positive solution of \eqref{pqle} on $B(x_0,2R)$ with $(p,q)\in \mathbb{BL}(n)$, where $\mathbb{BL}(n)$ is given by Definition \ref{set}.
		Then for $0<r\le R$, there exist $\b=\b(n,p,q)\in[0,\frac{2}{n-2})$ and $\rho=\rho(n,p,q)\in(0,1)$ such that %$(1-\rho)(p+q-1)-(1-\frac{q}{2})\b>0$ and 
		\begin{eqnarray}\label{180}
			&&\sup\limits_{B(x_0,R)}\left(u^{-(1-\frac{q}{2})\b}|\nabla\ln u|^{2-q}+u^{(1-\rho)(p+q-1)-(1-\frac{q}{2})\b}|\nabla\ln u|^{(2-q)\rho}\right)\nonumber\\
			&&\le C\left(R^{-(2-(n-2)\b)(1-\frac{q}{2})}\left(r^{n-2}\min_{\partial B(x_0,r)}u\right)^{-\left(1-\frac{q}{2}\right)\b}+R^{-(2-q)\rho}\Vert u\Vert^{(1-\rho)(p+q-1)-(1-\frac{q}{2})\b}_{L^{\infty}(B(0,2R))}\right),\nonumber
		\end{eqnarray}
		where $C=C(n,p,q)$.
	\end{thm}
	
	%Before we go further, we give a remark for these estimates appeared in Theorem \ref{kt}, Theorem \ref{kt0} and Theorem \ref{wkt0}.

%	\begin{rmk}\label{r1}
	%	We use the direct Bernstein method to derive these estimates. For estimate \eqref{17} and \eqref{180}, we use the same auxiliary function; and for estimate \eqref{18}, we use another function. Their main differece lies in the selection of $\rho$ (concretely, see the auxiliary function \eqref{g}). For proving  Theorem \ref{kt} and Theorem \ref{wkt0}, the parameter $\rho$ is chose sufficiently small. For proving Theorem \ref{kt0}, we choose $\rho=1$.	The main reason for this situation is that when $q\le 1$, we automatically obtain the universal $L^{\infty}$ estimate of solution at the maximal value point of related auxiliary function and this fact fails in our proof when $q>1$ and we need the extra $L^{\infty}$ estimate of solutions in Theorem \ref{wkt0}.
%	\end{rmk}

	\section{A fundamental lemma}\label{S2}
	In this section, we derive the elliptic equation for the squared norm of the logarithmic gradient associated with the positive solution to equation \eqref{pqle}. 
	Let $u$ be a positive solution of \eqref{pqle} on $\O\subset\M$, define 
	\begin{equation}\label{21}
		H:=|\nabla \ln u|^2,\quad l:=p+q-1,\quad L:=|\nabla \ln u|^q u^{l}
	\end{equation}
	 and define
\begin{equation}\label{22}
	 \quad Z:=\frac{L}{H} \qquad\text{on} \quad \O_r(u).
\end{equation}

Then we give the equation of $H$, which is the core of this section.
%The following lemma is the core of this section, which  plays an important role in deriving the Liouville theorem. %From the proof of Theorem \ref{Liou}, one will see that directly applying the classical Bochner formula is not optimal for equation \eqref{pqle}. 
\begin{lem}\label{l2}
	Let $u$ be a positive solution of \eqref{pqle} on $\O\subset\M$ and define $H,l,L,Z$ as \eqref{21} and \eqref{22} respectively. Then for any real number $\b$ and real-valued function $\sigma$ on $\O$, we have
	\begin{eqnarray}\label{23}
		 \Delta H&=&A(\beta) H^2+A(\sigma) L^2+B(\beta,\sigma) H L\nonumber\\
		 &&+[2(\beta-1) H+(2 \sigma-q) L]\langle\nabla\ln u, \nabla \ln H\rangle\nonumber\\
		 &&	+2\left|\T(\b,\sigma)-\frac{tr(\T(\b,\sigma))}{n}  g\right|^2+2 Ric(\nabla \ln u, \nabla\ln u)
	\end{eqnarray}
on  $\Or(u)$, where $A,B$ are two functions defined on $\mathbb{R}$ and $\mathbb{R}^2$ respectively and $\T(\b,\sigma)$ is a second order symmetric covariant tensor, whose definitions are as follows.
\begin{eqnarray}
	 A(x):&=&\frac{2}{n}(1+x)^2-2 x^2, \label{A}\\
	B(x, y):&=&\frac{4}{n}(1+x)(1+y)-4 x y-2 l,\label{B} \\
	\T(\b,\sigma)&=&u^{-1} \nabla^2 u-[(1+\beta)+\sigma Z] d \ln u \otimes d \ln u .
\end{eqnarray}
\end{lem}
\begin{proof}
	By Leibniz rule, Bochner formula and the definitions of $H,l,L$, we have
		\begin{eqnarray}\label{27}
		\Delta H&=&u^{-2} \Delta|\nabla u|^2+|\nabla u|^2 \Delta\left(u^{-2}\right)+2\left\langle\nabla\left(u^{-2}\right), \nabla|\nabla u|^2\right\rangle \nonumber\\
		&=&2 u^{-2}\left(|\nabla u|^2+\langle\nabla \Delta u, \nabla u\rangle+Ric(\nabla u, \nabla u)\right)\nonumber\\
		&&+\left(-2 u^{-3} \Delta u+6 u^{-4}|\nabla u|^2\right)|\nabla u|^2 -4 u^{-3}\left\langle\nabla u, \nabla|\nabla u|^2\right\rangle \nonumber\\
		&=&2 u^{-2}\left(\left|\nabla^2 u\right|^2+Ric(\nabla u, \nabla u)\right)-2 u^{-2}\left\langle\nabla\left(|\nabla u|^q u^p\right), \nabla u\right\rangle \nonumber\\
		&& +\left(2 u^{-3}\left| \nabla u\right|^q u^p+6 u^{-4}|\nabla u|^2\right)|\nabla u|^2-4 u^{-3}\left\langle\nabla u, \nabla|\nabla u|^2\right\rangle \nonumber\\
		&=&2 u^{-2}\left(\left|\nabla^2 u\right|^2+Ric(\nabla u, \nabla u)\right)-2 (p-1)|\nabla \ln u |^2 \cdot|\nabla u|^q u^{p-1} \nonumber\\
		&&-q u^{p-2}|\nabla u|^{q-2}\left\langle\nabla|\nabla u|^2, \nabla u\right\rangle+6|\nabla \ln u|^4-4 u^{-3}\left\langle\nabla u, \nabla|\nabla u|^2\right\rangle \nonumber\\
		&=&2 u^{-2}\left(\left|\nabla^2 u\right|^2+Ric(\nabla u, \nabla u)\right)-2 l H L \nonumber\\
		&&-2 H^2-(4 H+q L)\langle\nabla \ln H, \nabla \ln u\rangle
	\end{eqnarray}
on $\O$. Notice that for any real number $\b$ and real-valued function $\sigma$ on $\O$, we have
	\begin{eqnarray}\label{28}
		2 u^{-2}\left|\nabla^2 u\right|^2&= & 2\left|u^{-1} \nabla^2 u-[(1+\beta)+\sigma Z] d \ln u \otimes d\ln u\right|^2 \nonumber\\
		&& -2[(1+\beta) H+\sigma L]^2+4 u^{-3}((1+ \beta)+\sigma Z) \nabla^2 u(\nabla u, \nabla u)\nonumber \\
		&= & 2\left|u^{-1} \nabla^2 u-[(1+ \beta)+\sigma Z] d \ln u \otimes d \ln u\right|^2 \nonumber\\
		&& -2[(1+\beta) H+\sigma L]^2+2 u^{-3}((1+\beta)+\sigma Z)\left\langle\nabla|\nabla u|^2, \nabla u\right\rangle\nonumber \\
		&= & 2\left|u^{-1} \nabla^2 u-[(1+\beta)+\sigma Z] d \ln u \otimes d \ln u\right|^2-2[(1+\beta) H+\sigma L]^2\nonumber \\
		&& +2((1+\beta) H+\sigma L)\langle\nabla \ln u, \nabla \ln H\rangle+4((1+\beta) H+\sigma L) H
	\end{eqnarray}
	on $\Or(u)$.	Substituting \eqref{28} into \eqref{27} yields
	\begin{eqnarray}\label{29}
		\Delta H&=&  2\left|u^{-1} \nabla^2 u-[(1+ \beta)+\sigma Z] d \ln u \otimes d \ln u\right|^2-2 \beta^2 H^2-2 \sigma^2 L^2 \nonumber\\
		&&-(4 \beta \sigma +2l) H L +[2(\beta-1) H+(2 \sigma-q) L]\langle\nabla \ln H, \nabla \ln u\rangle 
	\end{eqnarray}
	on $\Or(u)$. If we define the second order covariant tensor (on $\Or(u)$)
	$$
	\T(\b,\sigma)=u^{-1} \nabla^2 u-[(1+\beta)+\sigma Z] d \ln u \otimes d \ln u,
	$$
	then \eqref{29} is equivalent to
	\begin{eqnarray}\label{210}
			\Delta H&=&  2\left|\T(\b,\sigma)-\frac{tr(\T(\b,\sigma))}{n} g\right|^2+\frac{2}{n}(t r \T(\b,\sigma))^2-2 \beta^2 H^2-2 \sigma^2 L^2\nonumber \\
		& &-(4 \beta \sigma +2l) H L+[2(\beta-1) H+(2 \sigma-q) L]\langle\nabla\ln H, \nabla \ln u\rangle 
	\end{eqnarray}
	on $\Or$, where $tr(\T(\b,\sigma))$ is the trace of the tensor $\T(\b,\sigma)$ under the Riemannian metric. Direct computation gives $$tr(\T(\b,\sigma))=-(1+\b)H-(1+\sigma)L,$$ then substituting it into \eqref{210} yields our desired equation \eqref{23}.

\end{proof}
	
	\begin{rmk}
		In most cases, \(\sigma\) is chosen to be a constant function; however, when considering the a priori estimates for equation \eqref{pqle} in the case where \( 1 < q < 2 \), we need \( \sigma \) to be a simple function.

	\end{rmk}
	
	\section{Proof of Theorem \ref{t10} and \ref{sc}}\label{S3}
	In this section, we prove Theorems \ref{t10} and \ref{sc}. Furthermore, using Theorem \ref{sc}, we derive the part of Theorem \ref{lrm} in \( q > 2 \) case.  
	 First,  by classical Laplacian comparison theorem and  the standard construction argument, we have the good  cut-off function.

	\begin{lem}[Existence of cut-off function]\label{cut}
		Let $(\M,g)$ be an $n$-dimensional complete Riemannian manifold with $Ric\ge-Kg$ and $K\ge 0$. For  any $R>0$, there exists a Lipschtiz function $\Phi(x)$ in $B(x_0,2R)$ such that\\
		(i) $\Phi(x)=\phi(d(x_0,x))$, where  $d(x_0,\cdot)$ is the distance function from $x_0$ and  $\phi$ is a non-increasing function on $[0,\infty)$ and
		\begin{eqnarray}
			\Phi(x)=
			\begin{cases}
				1\qquad\qquad\qquad\,\text{if}\qquad x\in B(x_0,R)\nonumber\\
				0\qquad\qquad\qquad\,\text{if}\qquad x\in B(x_0,2R)\setminus B(x_0,\frac{3}{2}R).
			\end{cases}
		\end{eqnarray}
		(ii) On $\{x\in B(x_0,2R):\Phi(x)>0\}$,
		\begin{equation}
			\frac{|\nabla \Phi|}{\Phi^{\frac{1}{2}}}\le\frac{C}{R}.\nonumber
		\end{equation}
		(iii)
		\begin{equation}
			\Delta \Phi\ge-\frac{C\sqrt{nK}}{R} \coth\left(R \sqrt{\frac{K}{n}}\right)-\frac{C}{R^2}\ge-C(n)\frac{1+\sqrt{K}R}{R^2}\nonumber
		\end{equation}
		holds on $B(x_0,2R)$ in the distribution sense and pointwise outside cut locus of $x_0$. Here, $C$ is a universal constant.
	\end{lem}

	Next, using Lemma \ref{l2}, we derive the following elliptic inequality. 
	\begin{lem}\label{l3}
		Let $(\M,g)$ be an n-dimensional Riemannian manifold with $Ric\ge-Kg$ and $K\ge 0$. Let $u$ be a positive solution of \eqref{pqle} on $B(x_0,R)\subset\M$ with $q\ge 0$ and $p+q<\frac{n+3}{n-1}$ and $H,l,L$ are defined by \eqref{21}. Then there exist $\b=\b(n)>0$ and $\e=\e(n,p,q)>0$ such that
		\begin{eqnarray}
			\Delta H\ge \e\left(H^2+ L^2\right)-2KH+[2(\beta-1) H-q L]\langle\nabla\ln u, \nabla \ln H\rangle\nonumber
		\end{eqnarray}
	on $\{x\in B(x_0,R): |\nabla u|(x)>0\}$.
	\end{lem}

	\begin{proof}
		Choosing $\b=\frac{2}{n-1}$, $\sigma=0$ and $\O=B(x_0,R)$ in Lemma \ref{l2} and using the curvature condition $Ric\ge -Kg$, we have
			\begin{eqnarray}\label{32}
			\Delta H&\ge&\frac{2}{n} H^2+\frac{2}{n} L^2+\left(\frac{4(n+1)}{n(n-1)}-2l\right) H L-2KH\nonumber\\
			&&+[2(\beta-1) H-q L]\langle\nabla\ln u, \nabla \ln H\rangle
		\end{eqnarray}
		on $\{x\in B(x_0,R): |\nabla u|(x)>0\}$.
		By basic inequality, for any $\e\in (0,2)$ we have
		\begin{equation}\label{33}
		\frac{2-\e}{n}\left(H^2+L^2\right)\ge \frac{4-2\e}{n}HL.
		\end{equation}
		Substituting \eqref{33} into \eqref{32} yields
		\begin{eqnarray}\label{34}
			\Delta H&\ge&\e\left(H^2+L^2\right)+\left(\frac{4(n+1)}{n(n-1)}+\frac{4-2\e}{n}-2l\right) H L-2KH\nonumber\\
			&&+[2(\beta-1) H-q L]\langle\nabla\ln u, \nabla \ln H\rangle\nonumber
		\end{eqnarray}
		on $\{x\in B(x,R): |\nabla u|(x)>0\}$. Due to $l<\frac{4}{n-1}$, we can fix a $\e\in(0,1)$ such that 
		$$
		\frac{4(n+1)}{n(n-1)}+\frac{4-2\e}{n}-2l>0.
		$$
		Then such $\e$ is what we need and the proof is complete.
		
	\end{proof}
	
%	\begin{rmk}
%		The choice of $\b$ and $\sigma$ is optimal and in this case, we make $l=p+q-1$ as large as possible. For $q=0$, these arguments are appeared in \cite[Section 3]{LU1}.
%	\end{rmk}

	By Lemma \ref{l3}, we give a proof of Theorem \ref{t10}. Due to Lemma \ref{control2} in Section \ref{S6} (the proof of Lemma \ref{control2} is self-contained in Section \ref{S6}), we just prove the logarithmic gradient estimate part in Theorem \ref{t10}.
	
	\begin{proof}[Proof of Theorem \ref{t10}]

		Let $u$ be a positive solution of \eqref{pqle}, and $H,l,L$ are defined by \eqref{21}.	Let	$\P(x)$ be an undetermined cut-off function as in Lemma \ref{cut}, we define
		\begin{equation}
			A(x)=\P(x)H(x),\qquad\text{on $B(x_0,2R)$}.
		\end{equation}
		 By chain rule, we have
		\begin{equation}\label{A1}
			\d A=\frac{\d \P}{\P}A+2\left\langle\nabla \P,\nabla H\right\rangle+\P\d H,
		\end{equation}
		on $N:=\{x\in B(x_0,2R):A(x)>0\,\, \text{and}\,\, x\notin cut(x_0)\}$, where
		$cut(x_0)$ is the cut locus of $x_0$.
		 Without loss of generality, we assume $A(x^{\star})=\max\limits_{B(x_0,2R)} A>0$ and the maximum value point $\x$ is outside of cut locus of $x_0$ (see, e.g., Yau \cite{Y} or Cheng-Yau \cite{CY}).
		At the point $\x$, by  Lemma \ref{l3}, there exist $\e=\e(n,p,q)\in(0,1)$ and $\b=\frac{2}{n-1}>0$ such that
		\begin{eqnarray}\label{37}
			0\ge\d A&\ge&\left(\frac{\d \P}{\P}-2\frac{|\nabla\P|^2}{\P^2}\right)A+\e\P\left(H^2+L^2\right)-2KA\nonumber\\
			&& +\P[2(\beta-1) H-q L]\langle\nabla\ln u, \nabla \ln H\rangle\nonumber\\
			&\ge&\left(\frac{\d \P}{\P}-2\frac{|\nabla\P|^2}{\P^2}\right)A+\e\P \left(H^2+L^2\right)-2KA\nonumber\\
			&& -[2(\beta-1) H-q L]\langle\nabla\ln u, \nabla \P\rangle,
		\end{eqnarray}
	where we use $\nabla A(\x)=\left(\nabla\P H+\P\nabla H\right)(\x)=0$.
		Cauchy-Schwarz inequality and Young inequality derive
		\begin{eqnarray}
				2(1-\b)H\left\langle\nabla \ln u,\nabla \P\right\rangle&\ge&-\frac{1}{2}\e\P H^2-\frac{8}{\e}\frac{|\nabla\P|^2}{\P^2}A,\nonumber\\
				qL\langle\nabla\ln u, \nabla \P\rangle\ge -qLH^{\frac{1}{2}}|\nabla\P|&=&-\left(\sqrt{\e}\P^{\frac{1}{2}}L\right)\left(\sqrt{\e}\P^{\frac{1}{4}}H^{\frac{1}{2}}\right)\left(\frac{q}{\e}\frac{|\nabla\P|}{\P^{\frac{3}{4}}}\right)\nonumber\\
				& \ge&-\frac{\e}{2}\P L^2-\frac{\e^2}{4}\P H^2-\frac{1}{4}\left(\frac{q}{\e}\right)^4\frac{|\nabla\P|^4}{\P^3}.\nonumber
		\end{eqnarray}
		So, substitute these inequalities into \eqref{37}, we have 
		\begin{eqnarray}\label{216}
			0\ge\left(\frac{\d \P}{\P}-\left(2+\frac{8}{\e}\right)\frac{|\nabla\P|^2}{\P^2}\right)A+\frac{1}{4}\e\P H^2-2KA-\frac{1}{4}\left(\frac{q}{\e}\right)^4\frac{|\nabla\P|^4}{\P^3}.
		\end{eqnarray}
		Multiplying both sides of \eqref{216} by $\P(\x)$ and  using the property of cut-off function in Lemma \ref{cut}, we see 
		\begin{equation}
			\max\limits_{B(x_0,2R)} A=A(\x)\le C(n,p,q)\left(\frac{1}{R^2}+K\right).\nonumber
		\end{equation}
		Furthermore, we have the desired estimate and finish the proof.

	\end{proof}

The remainder of this section will be devoted to proving Theorem \ref{sc} and Theorem \ref{lrm} in the case $q>2$. First, we need an elliptic inequality for the auxiliary function.

\begin{lem}\label{l43}
	Let $(\M,g)$ be an $n$-dimensional manifold with $Ric\ge-Kg$ and $u$ be a positive solution of \eqref{pqle} with $q>2$ in the domain $\O\subset\M$,  and define $H,l,L$ as \eqref{21} and the function $Z$ in the domain $\O_r(u)$ as \eqref{22}. For any $d\ge 0$, we define $M=L-dH$ on $\O_r(u)$ and have
	\begin{eqnarray}
		\d M&\ge&\frac{2}{n}\left(\frac{q}{2}-1\right)Z(H+L)^2-\left(l+\frac{2ql}{(q-2)d}\right)L^2\nonumber\\
		&&-(qL-2dH)K-(qZ+2)\left\langle\nabla M,\nabla\ln u\right\rangle+\frac{2qlL}{qL-2dH}\left\langle\nabla M,\nabla\ln u\right\rangle\nonumber
	\end{eqnarray}
	on the set $\{x\in \O: L-dH>0\}$.	
\end{lem}

\begin{proof}
By the definitions of 	$H,l,L$ and $M$, using the Leibniz rule, we immediately obtain 
\begin{eqnarray}
	\d M&=&\left(\frac{q}{2}Z-d\right)\d H+\frac{q}{2}\left(\frac{q}{2}-1\right)L|\nabla\ln H|^2\nonumber\\
	&&+l((l-1)H-L)L+qlL\left\langle\nabla \ln H,\nabla\ln u\right\rangle\nonumber
\end{eqnarray}	
and 
\begin{equation}\label{47}
\left(\frac{q}{2}L-dH\right)\left\langle\nabla \ln H,\nabla\ln u\right\rangle=\left\langle\nabla M,\nabla\ln u\right\rangle-lHL
\end{equation}
on $\O_r(u)$.

	Choosing $\b=\sigma=0$ in the equation of $\d H$ (recall Lemma \ref{l2}) and using the curvature condition $Ric\ge -Kg$, we have 
		\begin{eqnarray}\label{48}
		\Delta H\ge\frac{2}{n}(H+L)^2-2lHL-(2 H+q L)\langle\nabla\ln u, \nabla \ln H\rangle-2KH
	\end{eqnarray}
	on  $\Or(u)$. Notice that in the subset  $\{x\in \O: L-dH>0\}$, we have $$\frac{q}{2}Z-d\ge \left(\frac{q}{2}-1\right)Z>0.$$ Substituting this inequality, \eqref{47} and \eqref{48} into the equation of $\d M$ yields the desired elliptic inequality.

\end{proof}

\begin{proof}[Proof of Theorem \ref{sc}]
		Let $u$ be a positive solution of \eqref{pqle} with $q>2$ on $B(x_0,2R)$,  and define $H,l,L$ as \eqref{21} and the function $Z$ in the open subset $\{x\in B(x_0,2R):|\nabla u|>0\}$ as \eqref{22}. We always assume that $l=p+q-1>0$ or Theorem \ref{sc} is implied by Theorem \ref{t10} and the following Lemma \ref{control2}.
	Define the auxiliary function
	\[
	U = \Phi M
	\]
	on $B(x_0,2R)$, where $M=L-dH$ with undetermined $d>0$. Without loss of generality, \( U \) gets its positive maximal value in \( B(x_0, 2R) \) at point \( x^* \). 
	So, \( x^* \in \{ x \in B(x_0, 2R) : L - dH > 0 \} \).
	Then at the point $x^*$, by Leibniz rule and Lemma \ref{l43}, we have
\begin{eqnarray}\label{49}
	0 &\geq& \Delta U(x^*) = \left( \Delta \Phi - 2 \frac{|\nabla \Phi|^2}{\Phi} \right) M + \Phi \Delta M \nonumber\\
	&\geq& \left( \Delta \Phi - 2 \frac{|\nabla \Phi|^2}{\Phi} \right) M + \frac{2}{n} \left( \frac{q}{2} - 1 \right) \Phi Z (H + L)^2  \nonumber\\
	&&- l \Phi L^2 \left( 1 + \frac{2 q l}{(q - 2) d} \right) - K \Phi (q L - 2 d H) \nonumber\\
	&&+ (2 + q Z) M \langle \nabla \Phi, \nabla \ln u \rangle - \frac{2 q l L M}{q L - 2 d H} \langle \nabla \Phi, \nabla \ln u \rangle
\end{eqnarray}
	
	We deal with the gradient items as follows. For any $\varepsilon_1,\varepsilon_2>0$, we have
	\[
	\begin{split}
		(2 + q Z) M \langle \nabla \Phi, \nabla \ln u \rangle &\geq - (2 + q Z) M |\nabla \Phi| H^{\frac{1}{2}} \\
	&\geq - \left( \frac{2}{d} + q \right) Z L H^{\frac{1}{2}} |\nabla \Phi| \\
	&\geq - \left( \frac{2}{d} + q \right) Z \left( \varepsilon_1 \Phi L^2 + \frac{1}{4 \varepsilon_1} \frac{|\nabla \Phi|^2}{\Phi} H \right)
	\end{split}
	\]
	and
	\[
	\begin{split}
		- \frac{2 q l L M}{q L - 2 d H} \langle \nabla \Phi, \nabla \ln u \rangle &\geq - \frac{2 q l L M}{(q - 2) L} |\nabla \Phi| H^{\frac{1}{2}} \\
		&\geq - \frac{2 q l}{q - 2} \left( \varepsilon_2 \Phi L^2 + \frac{1}{4 \varepsilon_2} \frac{|\nabla \Phi|^2}{\Phi} H \right).
	\end{split}
	\]
	Now, we fix \[ d  = \frac{2 n q l}{q - 2}= \frac{2 n q (p + q - 1)}{q - 2}>0 \]
	and derive
	\[
	\begin{split}
		&\Phi \left[ \frac{1}{n} \left( \frac{q}{2} - 1 \right) Z (H + L)^2 - l L^2 \left( 1 + \frac{2 q l}{(q - 2) d} \right) \right]\\
		 &\geq \Phi \left[ \frac{1}{n} \left( \frac{q}{2} - 1 \right) d L^2 - l L^2 \left( 1 + \frac{2 q l}{(q - 2) d} \right) \right] \\
		&\geq \Phi \left( q l - l \left( 1 + \frac{1}{n} \right) \right) \geq 0.
	\end{split}
	\]
	Then, we fix \( \varepsilon_1=\varepsilon_1(n,p,q), \varepsilon_2=\varepsilon_2(n,p,q) > 0 \) such that
	\[
	\left( \frac{2}{d} + q \right) \varepsilon_1 \leq \frac{1}{4 n} \left( \frac{q}{2} - 1 \right)d\quad \text{and}\quad \frac{2 q l}{q - 2} \varepsilon_2 \leq \frac{1}{4 n} \left( \frac{q}{2} - 1 \right) d.
	\]
	Substituting these inequalities into \eqref{49} yields
	\begin{eqnarray}\label{4100}
			0 &\geq& \frac{1}{2 n} \left( \frac{q}{2} - 1 \right)\Phi Z (H + L)^{2} - q \Phi K L \nonumber\\
			&&- C(n, p, q) \frac{|\nabla \Phi|^2}{\Phi} (H + L) + \left( \Delta \Phi - 2 \frac{|\nabla \Phi|^2}{\Phi} \right) M.
	\end{eqnarray}
	Using the property of the cut-off function \( \Phi \) in Lemma \ref{cut}, in \eqref{4100}, we have
	\[
	0 \geq \frac{1}{2n} \left( \frac{q}{2} - 1 \right) d (H + L) \Phi - C(n, p, q) \left( K + \frac{1}{R^2} \right),
	\]
	which implies
	\[
	\Phi (H + L)(x^*) \leq C(n, p, q) \left( K + \frac{1}{R^2} \right).
	\]
		Therefore,
	\[
	\Phi M(x^*) \leq C(n, p, q) \left( K + \frac{1}{R^2} \right)
	\]
and
	\[
	\sup_{B(x_0, R)} (L - dH) \leq \sup_{B(x_0, 2R)} \Phi M = \Phi M(x^*) \leq C(n, p, q) \left( K + \frac{1}{R^2} \right). 
	\]
We finish the proof.	
	
	\end{proof}

For proving Theorem \ref{lrm} in the case $q>2$, we also need the properties of super-harmonic and sub-harmonic functions on the  Riemannian manifolds with non-negative Ricci curvature.

\begin{lem}\cite[Corollary 5.8]{So}\label{l45}
	 ($L^\epsilon$-estimate) Let $(\M,g)$ be an $n$-dimensional Riemannian manifolds with \( \text{Ric} \geq 0 \). Let  \( u  \) be a non-negative super-harmonic function on $B(x_0,2R)$.
	Then
	\[
	\left(\frac{1}{|B(x_0,R)|} \int_{B(x_0,R)} u^\epsilon(x) dx \right)^{\frac{1}{\epsilon}} \leq C \inf_{B(x_0,R)} u ,
	\]
	where the constants \( \epsilon \in (0, 1) \) and \( C > 0 \) depend only on $n$.
	
\end{lem}
Lemma \ref{l45} is a special case of Soojung \cite[Corollary 5.8]{So}, which claim more general result for $p$-laplacian equation on manifolds with curvature bounded below.

\begin{lem}\label{l46}
	(Upper bound of sub-harmonic function) Let $(\M,g)$ be an $n$-dimensional Riemannian manifolds with \( \text{Ric} \geq 0 \). Let  \( u  \) be a non-negative sub-harmonic function on $B(x_0,2R)$.
	Then for any $s\in(0,2]$, 
	\[
\sup_{B(x_0,R)} u \le	C(n,s)\left(\frac{1}{|B(x_0,R)|} \int_{B(x_0,R)} u^s(x) dx \right)^{\frac{1}{s}}.
	\]
	
\end{lem}

In the Euclidean case, Lemma \ref{l46} is a special case of Malý--Ziemer \cite[Corollary 3.10]{MZ}. Specifically, if we replace the Sobolev inequality in Euclidean spaces with 	Saloff-Coste's Sobolev inequality (\cite[Theorem 3.1]{Sa}) on manifolds, the proof in \cite{MZ} can be adapted without modification to yield Lemma \ref{l46}.  

\vspace{1mm}
Now, we give  a proof of Theorem \ref{lrm} in the case $q>2$.

\begin{proof}[Proof of Theorem \ref{lrm} in the case $q>2$]
First, we assume $l=p+q-1> 0$ or the Liouville theorem is implied by Theorem \ref{t10}.	Let \( u \) be a positive solution of \eqref{pqle} on \( \M\). Define \( \ell = \inf_{\M} u \) and \( u_{\ell} = u - \ell \). We may assume that  \( u_{\ell} > 0 \); otherwise, it follows from the strong maximum principle that \( u \) is  constant. Now, we fix \( s= 1+ \frac{2 n q (p + q - 1)}{q - 2}> 1\). Then using Theorem \ref{sc}, we derive
	\[
	\begin{split}
		\Delta(u_{\ell}^s) &= s u_{\ell}^{s - 1} \left[ (s - 1) \frac{|\nabla u_{\ell}|^2}{u_{\ell}} - |\nabla u_{\ell}|^q u_{\ell}^p \right] \\
		&\ge  s u_{\ell}^{s - 1} \left[ \frac{2 n q (p + q - 1)}{q - 2} \frac{|\nabla u|^p}{u} - |\nabla u|^q u^p \right] \geqslant 0
	\end{split}
	\]
	on \( \M \). So \( u_{\ell}^s \) is a sub-harmonic function on \( \M \) and $u_{\ell}$ is a super-harmonic function. Then, using Lemma \ref{l45} and Lemma \ref{l46}, we have (the $\epsilon$ is same as in Lemma \ref{l45})
	\begin{eqnarray}
	\sup_{B(x_0,R)} u_{\ell}^{s}&\le&  C(n,s)\left(\frac{1}{|B(x_0,R)|} \int_{B(x_0,R)} u_{\ell}^{\epsilon}(x) dx \right)^{\frac{s}{\epsilon}}\nonumber\\
	&\le& C(n, s) \inf_{B(x_0,R)} u_{\ell}^s \to 0\qquad \text{as}\quad R\to \infty,\nonumber
	\end{eqnarray}
	which yields a contradiction. So we have \( u \equiv \ell \) and finish the proof.
\end{proof}

\section{Equations of the auxiliary functions}\label{S4} 

In this section, we construct several auxiliary functions of solution to \eqref{pqle} and compute their elliptic equations.

Let $u$ be a positive solution of \eqref{pqle} in the domain $\O\subset\M$,  and define $H,l,L$ as \eqref{21} and function $Z$ in the domain $\O_r(u)$ as \eqref{22}. For undetermined real numbers $k,d,\gamma$, we define the auxiliary function 
\begin{equation}\label{F}
	F:=u^{k}H^{\frac{q\gamma}{2}}\left(H+dL\right):=I+II\quad\text{on the domain}\,\,\,\O_r(u),
\end{equation}
where
\begin{equation}\label{1a2}
	I:=u^kH^{1+\frac{q\gamma}{2}},\qquad II:=du^{k+l}H^{\frac{q(1+\g)}{2}}.
\end{equation}
When
$1+\frac{q\gamma}{2}\ge 0$ and $ \g\ge  -1,$
 $F\in C(\O)\cap C^{\infty}(\O_r(u))$.

\vspace{2mm}

The following lemma gives the elliptic equation of $F$.% which is key to prove the desired Liouville theorem.

\begin{lem}\label{kl}
	Let $u$ be a positive solution of \eqref{pqle} with $q\in[0,2)$ in the domain $\O$,  and define the functions $H,l,L, Z$ in the domain $\O_r(u)$ as \eqref{21} and  \eqref{22} respectively. For $k\in\R, d\ge 0$ and $\gamma\ge -1$, the auxiliary function $F$ is given by \eqref{F}. Then for any real numbers $\b,\delta$ and real-valued functions $\sigma,\tau$ on $\O$, we have
	\begin{eqnarray}
		\d F
		&=&\frac{JW\left(\S_1H^5+\S_2H^4L+\S_3H^3L^2+\S_4H^2L^3+\S_5HL^4+\S_6L^5\right)}{\left(\left(1+\frac{q\g}{2}\right)H+\frac{dq}{2}(1+\g)L\right)^{3}}\nonumber\\
		&&+\frac{JW\left(\left(1+\frac{q\g}{2}\right)N(\b,\sigma)H+\frac{dq}{2}(1+\g)N(\delta,\tau)L\right)}{\left(1+\frac{q\g}{2}\right)H+\frac{dq}{2}(1+\g)L}\nonumber\\
		&& +H W\left\{\frac{q\g}{2}\left(1+\frac{q \g}{2}\right) \frac{(1+d Z)^2}{J^2} |\nabla \ln F|^2\right.\nonumber\\
		&&\left.+2k\left(1+\frac{q \g}{2}\right) \frac{1+dZ}{J}\langle\nabla\ln F, \nabla \ln u\rangle\right.\nonumber\\
		&&\left.-q \g\left(1+\frac{q \g}{2}\right) \frac{1+dZ}{J^2}(k(1+dZ)+l d Z) \langle\nabla \ln F, \nabla \ln u\rangle\right. \nonumber\\
		&& \left.+\left(1+\frac{q \g}{2}\right)\left(2(\beta-1)+(2\sigma-q) Z\right) \frac{1+ d Z}{J} \langle\nabla \ln F, \nabla \ln u\rangle\right\}\nonumber\\
		&&+dLW\left\{\frac{q(1+\g)}{2}\left(\frac{q(1+\g)}{2}-1\right)\frac{(1+dZ)^2}{J^2}|\nabla\ln F|^2\right.\nonumber\\
		&&\left.+(k+l)q(1+\g)\frac{1+dZ}{J}\langle\nabla \ln F, \nabla \ln u\rangle\right.\nonumber\\
		&&\left.-q(1+\g)\left(\frac{q(1+\g)}{2}-1\right)\frac{1+dZ}{J^2}\left(k(1+dZ)+ldZ\right)\langle\nabla \ln F, \nabla \ln u\rangle
		\right.\nonumber\\
		&&\left.+\left(2(\delta-1)+(2\tau-q) Z\right)\frac{q(1+\g)}{2}\frac{1+dZ}{J}\langle\nabla \ln F, \nabla \ln u\rangle\right\}
	\end{eqnarray}
on $\O_r(u)$, where
\begin{equation}
	W=u^kH^{\frac{q\g}{2}},\qquad J=\left(1+\frac{q\g}{2}\right)+\frac{dq}{2}(1+\g)Z,\nonumber
\end{equation}
the function $N(\cdot,\cdot):\mathbb{R}^2\to\mathbb{R}$ is defined by
\begin{equation}\label{N}
	N(x,y)=2\left|\T(x,y)-\frac{tr(\T(x,y))}{n}  g\right|^2+2 Ric(\nabla \ln u, \nabla\ln u)
\end{equation}
and the symmetry covariant 2-tensor $\T$ given by
\begin{equation}\label{T}
	\T(x,y)=u^{-1} \nabla^2 u-[(1+x)+y Z] d \ln u \otimes d \ln u .
\end{equation}
The coefficients $\S_1,\cdots,\S_6$ are polynomial functions of $k,\g,d,\b,\delta,\sigma,\tau$ and  their expressions are given in Lemma \ref{a1} of Appendix A.

\end{lem}

\begin{proof}
	
	Let $I,II$ defined by \eqref{1a2}, then Leibniz rule of Laplacian directly derives
	\begin{eqnarray}\label{d1}
		\Delta I &=& \Delta u^k \cdot H^{1+\frac{q\g}{2}} + \Delta(H^{1+\frac{q\g}{2}}) u^k + 2k\left(1+\frac{q\g}{2}\right)u^{k-1}H^{\frac{q\g}{2}} \langle \nabla u, \nabla H \rangle \nonumber  \\
		&=& \left( ku^{k-1} \Delta u + k(k-1)u^{k-2} |\nabla u|^2 \right)  H^{1+\frac{q\g}{2}} \nonumber \\
		&&\quad + \left[ \left(1+\frac{q\g}{2}\right)H^{\frac{q\g}{2}}  \Delta H +\left(1+\frac{q\g}{2}\right)\frac{q\g}{2}  H^{\frac{q\g}{2} - 1}  |\nabla H|^2 \right] u^k \nonumber \\
		&&\quad + 2k \left(1+\frac{q\g}{2}\right) H^{1+\frac{q\g}{2}}  u^k \langle \nabla \ln u, \nabla \ln H \rangle\nonumber  \\
		&=& \left(-ku^{k+l} H^{\frac{q\g}{2}} + k(k-1)u^k H \right) H^{1+\frac{q\g}{2}} \nonumber \\
		&&\quad + \left[ \left(1+\frac{q\g}{2}\right)H^{\frac{q\g}{2}}  \Delta H +\left(1+\frac{q\g}{2}\right)\frac{q\g}{2}  H^{1+\frac{q\g}{2}}  |\nabla \ln H|^2 \right] u^k \nonumber \\
		&&\quad + 2k \left(1+\frac{q\g}{2}\right) H^{1+\frac{q\g}{2}}  u^k \langle \nabla \ln u, \nabla \ln H \rangle
	\end{eqnarray}
and
		\begin{eqnarray}\label{d2}
				\Delta II &=& d \left[ (k+l)u^{k+l} \Delta u + (k+1)(k+\ell-1)u^{k+\ell-2} |\nabla u|^2 \right] H^{\frac{q(k+l)}{2}} \nonumber\\
			&& + u^{k+l} \left( \frac{q}{2}(1+\g) H^{\frac{q(1+\g)}{2} - 1} \Delta H +  \frac{q}{2}(1+\g) \left( \frac{q}{2}(1+\g)-1 \right) H^{\frac{q(1+\g)}{2} - 2} |\nabla H|^2 \right) \nonumber\\
			&&+ q(k+l)(1+\g)u^{k+l-1} H^{\frac{q(1+\g)}{2} - 1} \langle \nabla u, \nabla H \rangle
		\end{eqnarray}
	$\text{on } \Omega_r(u).$
	
	Notice that on the regular set $\O_r(u)$, we have
	\begin{align*}
		\nabla F &= ku^k H^{1+\frac{q\g}{2}} \nabla \ln u + \left(1+\frac{q\g}{2}\right) u^k H^{1+\frac{q\g}{2}} \nabla \ln H \\
		& + d \left( (k+l)u^{k+l} H^{\frac{q(1+\g)}{2}} \nabla \ln u + \frac{q}{2}(1+\g) u^{k+l} H^{\frac{q(1+\g)}{2}} \nabla \ln H \right) \\
		&= kF \nabla \ln u + l II\nabla \ln u + \left( \left(1+\frac{q\g}{2}\right) I + \frac{q}{2}(1+\g) II \right) \nabla \ln H,
	\end{align*}
	
	\begin{align*}
	II = (1+dZ)^{-1} dZ F,
	\end{align*}
	and so
\begin{eqnarray}
	\nabla F &=& \left( k + l (1+dZ)^{-1} dZ \right) F \nabla \ln u \nonumber\\
	&& + \left(  \left(1+\frac{q\g}{2}\right)(1+dZ)^{-1} + \frac{q}{2}(1+\g) (1+dZ)^{-1} dZ \right) F \nabla \ln H,\nonumber
\end{eqnarray}
which is equivalent to
\begin{equation}\label{grh}
	\nabla \ln H = J^{-1} (1+ dZ) \nabla \ln F - J^{-1} \left[ k + (k+l) dZ \right] \nabla \ln u,
\end{equation}
		where 
	\begin{equation}\label{j}
		J := \left( 1+\frac{q\g}{2} \right) + \frac{dq}{2}(1+\g) Z.
	\end{equation}
	From \eqref{grh}, we obtain 
	\begin{eqnarray}\label{410}
	|\nabla \ln H|^2 &=& J^{-2} (1+ dZ)^2 |\nabla \ln F|^2 + J^{-2} \left[ k + (k+l) dZ \right]^2 H \nonumber\\
	&& - 2J^{-2} (1+ dZ) (k+(k+l) dZ) \langle \nabla \ln F, \nabla \ln u \rangle
	\end{eqnarray}
and
	\begin{equation}\label{411}
			\langle \nabla \ln H, \nabla \ln u \rangle = J^{-1} (1+ dZ) \langle \nabla \ln F, \nabla \ln u \rangle - J^{-1} \left[ k + (k+l) dZ \right] H.
	\end{equation}
	Now, we define 
	\begin{equation}\label{w}
		W=u^kH^{\frac{q\g}{2}}.
	\end{equation}
	Then combining \eqref{d1}, \eqref{j}, \eqref{410}, \eqref{411} and \eqref{w}, we have
	\begin{eqnarray}\label{413}
		\Delta I 
		&=&\left(1+\frac{q\gamma}{2}\right) W \Delta H +W \left\{ k \left( -HL + (k-1)H^2 \right)\right. \nonumber\\
		&&\quad \left. + \frac{q\gamma}{2} \left(1+\frac{q\gamma}{2}\right) J^{-2} \left[ k + (k+l) dZ \right]^2 H^2 \right.\nonumber \\
		&&\qquad \left. - 2k \left( 1 + \frac{q\gamma}{2} \right) J^{-1} \left[ k + (k+l) dZ \right] H^2 \right\}\nonumber \\
		&&\quad + HW \left\{ \frac{q\gamma}{2} \left(1+\frac{q\gamma}{2}\right) J^{-2} (1+dZ)^2 |\nabla \ln F|^2 \right.\nonumber\\  &&\quad\left.-q\gamma \left(1+\frac{q\gamma}{2}\right) J^{-2} (1+dZ) (k + (k+l) dZ) \langle \nabla \ln F, \nabla \ln u \rangle \right.\nonumber \\
		&&\qquad \left. + 2k \left( 1 + \frac{q\gamma}{2} \right) J^{-1} (1+dZ) \langle \nabla \ln F, \nabla \ln u \rangle \right\}.
	\end{eqnarray}
	Combining \eqref{d2}, \eqref{j}, \eqref{410}, \eqref{411} and \eqref{w}, we have
	\begin{eqnarray}\label{414}
		\Delta II &=& \frac{dq}{2} (1+\gamma) WZ \Delta H + d W\left\{(k+l)(-L + (k+l-1) H)\right.\nonumber \\
		&&\left.+ \frac{q}{2} (1+\gamma) \left(\frac{q}{2}(1+\gamma) - 1\right) HL  J^{-2} (1+dZ)^2 |\nabla \ln F|^2 \right. \nonumber\\
		&&\left.-q(1+\gamma)(k+l)J^{-1} (k+(k+l) dZ) HL\right\}\nonumber\\
		&& + dLW \left\{ \frac{q}{2}(1+\gamma) \left( \frac{q}{2}(1+\gamma) - 1 \right) J^{-2} (1+dZ)^2 |\nabla \ln F|^2 \right. \nonumber\\
		&&- q(1+\gamma) \left( \frac{q}{2}(1+\gamma) - 1 \right) (1+dZ) (k+(k+l)dZ) J^{-2} \langle \nabla \ln F, \nabla \ln u \rangle \nonumber\\
		&&+ \left. q(k+l)  (1+\gamma) J^{-1} (1+dZ) \langle \nabla \ln F, \nabla \ln u \rangle \right\} .
	\end{eqnarray}
	Using Lemma \ref{l2} and identity \eqref{411}, for any $\b\in \mathbb{R}$ and real function $\sigma$, we have 
\begin{eqnarray}\label{415}
	\Delta H &=& A(\beta) H^2 + A(\sigma) L^2 + B(\beta, \sigma) HL + N(\beta, \sigma)\nonumber\\
	&&+ \left[ 2(\beta - 1)H + (2 \sigma-q)L \right]J^{-1}\times\nonumber\\
	&& \left[ (1+ dZ) \langle \nabla \ln F, \nabla \ln u \rangle -  (k + (k+l) dZ) H \right],
\end{eqnarray}
	where $A,B$ are defined by \eqref{A} and \eqref{B} respectively and  $N(\b,\sigma)$ is defined by \eqref{N}.
	Substituting \eqref{415} into \eqref{413}, then we have 
	\begin{eqnarray}\label{416}
		\Delta I& = & \left(1 + \frac{q\g}{2}\right) W \left[ A(\sigma) H^2+A(\sigma)L^2 + B(\beta, \sigma) HL + N(\beta, \sigma) \right] \nonumber\\
		&& + W \left\{ k \left( -HL + (k-1)H^2 \right) + \frac{q\g}{2} \left(1 + \frac{q\g}{2}\right) J^{-2} \left[ k + (k + l) dZ \right] H^2 \right. \nonumber\\
		&& \left. - 2k \left(1 + \frac{q\g}{2}\right) J^{-1} \left( k + (k + l) dZ \right) H^2 \right.\nonumber\\
		&&\left.- \left(1 + \frac{q\g}{2}\right) \left( 2(\beta - 1) H + (2\sigma - q) L \right) J^{-1} \left( k + (k + l) dZ \right) H \right\}\nonumber \\
		&& + HW \left\{ \frac{q\g}{2} \left(1 + \frac{q\g}{2}\right) J^{-2} \left( 1+dZ \right)^2  | \nabla \ln F|^2 \right.\nonumber\\
		&&\left.-q\g \left(1 + \frac{q\g}{2}\right) J^{-2} \left( 1+dZ \right) \left( k + (k + l) dZ \right) \left\langle \nabla \ln F, \nabla \ln u \right\rangle \right.\nonumber \\
		&& \left. + 2k \left(1 + \frac{q\g}{2}\right) J^{-1} \left(1+dZ \right) \left\langle \nabla \ln F, \nabla \ln u \right\rangle \right.\nonumber\\
		&&\left.+ \left(1 + \frac{q\g}{2}\right) \left( 2(\beta - 1) + (2 \sigma-q) Z \right) J^{-1} \left( 1+dZ \right) \left\langle \nabla \ln F, \nabla \ln u \right\rangle \right\}.
	\end{eqnarray}
	Replacing $\b,\sigma$ by $\delta,\tau$ in \eqref{415}, we see
	\begin{eqnarray}\label{417}
		\Delta H &=& A(\delta) H^2 + A(\tau) L^2 + B(\delta \tau) HL + N(\delta, \tau)\nonumber\\
		&&+ \left[ 2(\delta - 1)H + (2 \tau-q)L \right]J^{-1}\times\nonumber\\
		&& \left[ (1+ dZ) \langle \nabla \ln F, \nabla \ln u \rangle -  (k + (k+l) dZ) H \right].
	\end{eqnarray}
	And substituting \eqref{417} into the formula \eqref{414} yields
	\begin{eqnarray}\label{418}
		\Delta II &=& \frac{dq}{2} (1+\gamma) W Z \left[ A(\delta) H^2 + A(\tau) L^2 + B(\delta, \tau) HL + N(\delta, \tau) \right]\nonumber\\
		&&+ dWL \left\{ (k+l)(-L + (k + l - 1) H)\right. \nonumber\\
		&&\left.+ \frac{q}{2} (1+\gamma) \left( \frac{q}{2} (1+\gamma) - 1 \right) H J^{-2} (k + (k+l) dZ)^2\right.\nonumber\\
		&&\left.- q (k+l)(1+\gamma) J^{-1} (k + (k+l) dZ) H\right.\nonumber\\
		&&\left.- \frac{q}{2} (1+\gamma) \left[ 2(\delta-1) H + (2\tau-q) L \right] J^{-1} (k + (k+l) dZ)\right\}\nonumber\\
		&&+ dWL \left\{ \frac{q}{2} (1+\gamma) \left( \frac{q}{2} (1+\gamma) - 1 \right) J^{-2} (1+dZ)^2 \left| \nabla \ln F \right|^2 \right.\nonumber\\
		&&\left. - q (1+\gamma) \left( \frac{q}{2} (1+\gamma) - 1 \right) (1+dZ) (k + (k+l) dZ) J^{-2} \left\langle \nabla \ln F, \nabla \ln u \right\rangle \right.\nonumber\\
		&&\left.+ q (k+l)(1+\gamma) J^{-1} (1+dZ) \left\langle \nabla \ln F, \nabla \ln u \right\rangle\right.\nonumber\\
		&&\left.+ \frac{q}{2} (1+\gamma) \left[ 2(\delta-1) + (2\tau-q) Z \right] J^{-1} (1+dZ) \left\langle \nabla \ln F, \nabla \ln u \right\rangle \right\}.
	\end{eqnarray}

	Combining \eqref{416} and \eqref{418}, we represent $\Delta F$ as follows.	
	\begin{equation}
		\Delta F=\textbf{nonnegetive items}+\textbf{dominant terms}+\textbf{gradient items},\nonumber
	\end{equation}
	where
	\begin{equation}
		\textbf{nonnegetive items}=\frac{JW\left(\left(1+\frac{q\g}{2}\right)N(\b,\sigma)H+\frac{dq}{2}(1+\g)N(\delta,\tau)L\right)}{\left(1+\frac{q\g}{2}\right)H+\frac{dq}{2}(1+\g)L},\nonumber
	\end{equation}

	\begin{eqnarray}
		\textbf{dominant terms}
		&=&JW\left(\left(1+\frac{q\g}{2}\right)H+\frac{dq}{2}(1+\g)L\right)^{-3}\times\nonumber\\
		&& \left\{ \left[ \left(1 + \frac{q\g}{2}\right) A(\beta)H + \frac{dq}{2} (1+\gamma) A(\delta)L \right] H^2\times\right.\nonumber\\
		&& \left[ \left(1 + \frac{q\g}{2}\right) H + \frac{dq}{2} (1+\gamma) L \right]^2  \nonumber\\
		&& + \left[ \left(1 + \frac{q\g}{2}\right) B(\beta, \sigma) H + \frac{dq}{2} (1+\gamma)  B(\delta, \tau)L \right] HL\times\nonumber\\
		&& \left[ \left(1 + \frac{q\g}{2}\right) H + \frac{dq}{2} (1+\gamma) L \right]^2\nonumber \\
		&& + \left[ \left(1 + \frac{q\g}{2}\right) A(\sigma) H + \frac{dq}{2} (1+\gamma) A(\tau) L \right]L^2\times\nonumber\\
		&& \left[ \left(1 + \frac{q\g}{2}\right) H + \frac{dq}{2} (1+\gamma) L \right]^2 \nonumber\\
		&& + k(-L + (k-1)H) H^2 \left[ \left(1 + \frac{q\g}{2}\right) H + \frac{dq}{2} (1+\gamma) L \right]^2 \nonumber\\
		&& + \frac{q\g}{2} \left(1+\frac{q\g}{2}\right)H^3 (kH + (k+l)dL)^2\nonumber \\
		&& - 2k \left(1 + \frac{q\g}{2}\right) H^3 (kH + (k+l)dL)\times\nonumber\\
		&& \left( \left(1 + \frac{q\g}{2}\right) H + \frac{dq}{2} (1+\gamma) L \right)\nonumber \\
		&& - \left(1 + \frac{q\g}{2}\right) \left( 2(\beta-1) H + (2\sigma-q) L \right) H^2 \times\nonumber\\
		&& (kH + (k+l)dL)\left( \left(1 + \frac{q\g}{2}\right) H + \frac{dq}{2} (1+\gamma) L \right) \nonumber\\
		&& + d(k+l) HL (-L + (k+l-1)H)\times\nonumber\\
		&& \left( \left(1 + \frac{q\g}{2}\right) H + \frac{dq}{2} (1+\gamma) \right)^2 \nonumber\\
		&& + \frac{dq}{2} (1+\gamma) \left( \frac{q}{2} (1+\gamma) - 1 \right) H^2 L (kH + (k+l)dL)^2 \nonumber\\
		&& - dq (k+l) (1+\gamma) H^2 L (kH + (k+l)dL)\times\nonumber\\
		&& \left( \left(1 + \frac{q\g}{2}\right) H + \frac{dq}{2} (1+\gamma) L \right)\nonumber \\
		&& - \frac{dq}{2} (1+\gamma) HL \left( 2(\delta-1) H + (2\tau-q) L \right) \times\nonumber\\
		&& \left.(kH + (k+l)dL) \left( \left(1 + \frac{q\g}{2}\right) H + \frac{dq}{2} (1+\gamma) L \right) \right\},\nonumber
	\end{eqnarray}
	\begin{eqnarray}
	\textbf{gradient items}	&=&H W\left\{\frac{q\g}{2}\left(1+\frac{q \g}{2}\right) \frac{(1+d Z)^2}{J^2} |\nabla \ln F|^2\right.\nonumber\\
	&&\left.+2k\left(1+\frac{q \g}{2}\right) \frac{1+dZ}{J}\langle\nabla\ln F, \nabla \ln u\rangle\right.\nonumber\\
	&&\left.-q \g\left(1+\frac{q \g}{2}\right) \frac{1+dZ}{J^2}(k(1+dZ)+l d Z) \langle\nabla \ln F, \nabla \ln u\rangle\right. \nonumber\\
	&& \left.+\left(1+\frac{q \g}{2}\right)\left(2(\beta-1)+(2\sigma-q) Z\right) \frac{1+ d Z}{J} \langle\nabla \ln F, \nabla \ln u\rangle\right\}\nonumber\\
	&&+dLW\left\{\frac{q(1+\g)}{2}\left(\frac{q(1+\g)}{2}-1\right)\frac{(1+dZ)^2}{J^2}|\nabla\ln F|^2\right.\nonumber\\
	&&\left.+(k+l)q(1+\g)\frac{1+dZ}{J}\langle\nabla \ln F, \nabla \ln u\rangle\right.\nonumber\\
	&&\left.-q(1+\g)\left(\frac{q(1+\g)}{2}-1\right)\frac{1+dZ}{J^2}\times\right.\nonumber\\
	&&\left.\left(k(1+dZ)+ldZ\right)\langle\nabla \ln F, \nabla \ln u\rangle
	\right.\nonumber\\
	&&\left.+\left(2(\delta-1)+(2\tau-q) Z\right)\frac{q(1+\g)}{2}\frac{1+dZ}{J}\langle\nabla \ln F, \nabla \ln u\rangle\right\}.\nonumber
	\end{eqnarray}
	If we combine like terms of \textbf{dominant terms}, we can express it as 
	$$
	\frac{JW\sum_{i=1}^{6}\S_iH^{6-i}L^{i-1}}{\left(\left(1+\frac{q\g}{2}\right)H+\frac{dq}{2}(1+\g)L\right)^{3}},
	$$
	where $\S_i$ $(1\le i\le 6)$ are polynomial functions of $k,\g,d,\b,\delta,\sigma,\tau$. Then we complete the proof.

\end{proof}

%Before we give the proof of Lemma \ref{kl}, 
We provide a remark to illustrate the optimality of Lemma \ref{kl}.

\begin{rmk}\label{opt}

	When \( q \in [0,1) \) and \( l = p + q - 1 = \frac{(2-q)^2}{(1-q)(n-2)} \), if we take \( u \) to be the radially symmetric ground states on $\R^n$ with $n\ge 3$ 
	$$
	u(x)=\left(\frac{(1-q)(n-2)^{q-1}}{n-(n-1)q}+|x|^{\frac{2-q}{1-q}}\right)^{-\frac{(1-q)(n-2)}{2-q}},
	$$
	then $\O_r(u)=\R^n\setminus\{0\}$. If we set $$k=-(1-\frac{q}{2})\b,\quad \b=\frac{2}{n-2},\quad \g=-1\,\,\,\text{and}\,\,\, d=\frac{n-2}{(1-q)(n-(n-1)q)},$$  it can be shown that the function $F$ given by \eqref{F} is a positive constant.

	In the right hand side of the equation of $\d F$, we set $$\delta=\b=\frac{2}{n-2},\quad \sigma=\tau\equiv-\frac{q}{n-(n-1)q},$$ and we find that $\S_i\equiv 0$ for $i=1,2,\cdots,6$. Simultaneously, the non-negative items $N(\b,\sigma)$ and $N(\delta,\tau)$ both vanish. Therefore, the equation of $\d F$
	is optimal if we discard the items $N(\b,\sigma)$ and $N(\delta,\tau)$ (in fact, when we use the equation of auxiliary function $F$ to derive the estimate of $F$, we always discard the items $N(\b,\sigma)$ and $N(\delta,\tau)$).
	%   However, if \( u \) is taken to be a positive constant, \( F \) also reduces to a constant. This scenario cannot arise when \( q = 0 \), as the equation admits no positive solutions in this case.  
	
\end{rmk}

When $\g=-1$ and $k=-\left(1-\frac{q}{2}\right)\b$, the  following  lemmas provide the lower bound estimates of the leading coefficients $\S_i$ appearing in Lemma \ref{kl}.

\begin{lem}\label{cl}
	Let $\S_1,\cdots,\S_6$ be defined as in Lemma \ref{a1}. If we set $\g=-1$ and $k=-\left(1-\frac{q}{2}\right)\b$, then $\S_i=0$ for $ i=4,5,6$ and $\S_i=\left(1-\frac{q}{2}\right)^3S_i$, where
	\begin{eqnarray}
		S_1&:=& \left(\frac{2}{n}(1+\b)-\b\right)(1+\b),\\
		S_2&:=&  (B(\beta ,\sigma)+(2\sigma+1-q)\b)\nonumber\\
		&&+d\left(1-\frac{q}{2}\right)^{-1} \left(l+1+\left(\frac{q}{2}-1\right)\b\right)\left(l+\left(\frac{q}{2}-1\right)\b\right) ,\\
		S_3&:=& A(\sigma)-\frac{d^2 q}{2}  \left(1-\frac{q}{2}\right)^{-2} \left(l+\left(\frac{q}{2}-1\right)\b\right)^2\nonumber\\
		&&-d \left(1-\frac{q}{2}\right)^{-1} \left(l+\left(\frac{q}{2}-1\right)\b\right) (2 \sigma+1 -q),
	\end{eqnarray}
the functions $A,B$ are defined as  \eqref{A} and \eqref{B} in Lemma \ref{l2} respectively.
\end{lem}

\begin{proof}
By Lemma \ref{a1}, direct computations end the proof.
\end{proof}

\begin{lem}\label{p1}
Let $S_1,S_2,S_3$ be defined as in Lemma \ref{cl}. For the following cases:\\
(1) $n=2$ and $q\in[0,2)$ and $l=p+q-1<\infty$;\\
(2) $n\ge 3$ and  $q\in [0,1)$ and $l=p+q-1\in[\frac{2}{n-2},\frac{(2-q)^2}{(1-q)(n-2)})$;\\ 
(3) $n\ge 3$ and $q\in[1,2)$ and $l=p+q-1\in[\frac{2}{n-2},\infty)$;\\
there exist $\b\in[0,\frac{2}{n-2})$, $d\ge 0$ and constant function which almost depend on $n,p,q$ such that
$$
S_i>0,\qquad i=1,2,3.
$$
\end{lem}

\begin{proof}
First, we consider $n=2$ case. Fix $\sigma=d=0$ and $\b=\max(\frac{2l}{3-q},0)$, then  $S_2=2$, $S_3=1$ and $S_1>0$. This proves case (1).

Second, we always assume $n\ge 3$ as follows.
 For any $\b\in[0,\frac{2}{n-2})$, $S_1>0$. Define
$$
\theta=\frac{d\left(l+\left(\frac{q}{2}-1\right)\b\right)}{1-\frac{q}{2}},
$$
we have $\theta>0$ if and only if $d>0$ under the assumption that $\b\in[0,\frac{2}{n-2})$, $q\in[0,2)$ and $l\ge\frac{2}{n-2}$. Notice that $S_3>0$ is equivalent to 
\begin{equation}\label{425}
	A(\sigma)+(q-1-2\sigma)\theta-\frac{q}{2}\theta^2>0.
\end{equation}
We choose (in fact, such selection of $\sigma$ is optimal)
$$\sigma\equiv\frac{2-n\theta}{2(n-1)},$$
then 
\begin{eqnarray}
	&&A(\sigma)+(q-1-2\sigma)\theta-\frac{q}{2}\theta^2\nonumber\\
	&\equiv&\left(\frac{n}{2(n-1)}-\frac{q}{2}\right)\theta^2+\left(q-1-\frac{2}{n-2}\right)\theta+\frac{2}{n-1}.
\end{eqnarray}
So, if $p\in[0,1],$ \eqref{425} is equivalent to
$$
\theta<\frac{2}{n-(n-1)q} \quad \text{or} \quad \theta>2;
$$
if 
$p\in(1,\frac{n}{n-1}),$ \eqref{425} is equivalent to
$$
\theta>\frac{2}{n-(n-1)q} \quad \text{or} \quad \theta<2;
$$
if 
$p=\frac{n}{n-1},$ \eqref{425} is equivalent to
$ \theta<2;
$\\
if 
$p\in(\frac{n}{n-1},2),$ \eqref{425} is equivalent to
$$
\frac{2}{n-(n-1)q}<\theta<2.
$$

Now, we divide the remaining arguments into two parts.

$\bullet$ If case (2) holds, then we choose $$\sigma_0\equiv\frac{2-n\theta_0}{2(n-1)}=\frac{-q}{n-(n-1)q}\quad (\rm{i.e.}\,\, \theta_0=\frac{2}{n-(n-1)q})$$ and $\b_0=\frac{2}{n-2}$, then 
$S_2>0$ is equivalent to $$l<\frac{(2-q)^2}{(1-q)(n-2)}.$$
By continuity, we can find $\sigma=\frac{2-n\theta}{2(n-1)}$ with $\theta<\frac{2}{n-(n-1)q}$ and $\b<\b_0$ such that $S_2>0$. At the present, we see $S_3>0$ and $S_1>0$.

$\bullet$ If case (3) holds, then we choose $$\sigma_0\equiv\frac{2-n\theta_0}{2(n-1)}=-1\quad ( \rm{i.e.}\,\,\theta_0=2)$$ and $\b=0$, then 
$S_2=2>0$ and $S_3=0$. By continuity,  we can find $\sigma=\frac{2-n\theta}{2(n-1)}$ with $\theta<2$ and $\b=0$ such that $S_1,S_2,S_3$ are positive.

\end{proof}

Now, we define a new auxiliary function related to solution. First, we fix $\g=-1$ and $k=-\left(1-\frac{q}{2}\right)\b$ in the definition of $F$ (recall \eqref{F}), where $\beta\in\mathbb{R}$ is undetermined. That is, 
\begin{equation}\label{F0}
	F := u^{-\left(1 -\frac{q}{2}\right)\beta} \left( H^{1 - \frac{q}{2}} + du^l \right)\quad \text{on}\, \,\, \O,
\end{equation}
where $H,l$ is defined by \eqref{21}. On the regular set $\O_r(u)$, we define
$
W:= u^{-\left(1 - \frac{q}{2}\right)\beta} H^{-\frac{q}{2}}
$
and $L,Z$ as \eqref{21} and \eqref{22} respectively. For $0\le q< 2$, we define the following two-parameter pertubation of $F$:
\begin{equation}\label{g}
	G:=F\times\left(\frac{H^{1-\frac{q}{2}}}{H^{1-\frac{q}{2}}+\varepsilon u^l}\right)^\rho \quad \text{on}\, \,\, \O,
\end{equation}
where $F$ is given by \eqref{F0} and $\varepsilon, \rho > 0$. On $\O_r(u)$, we see
$$
G =W(H+dL)\left( \frac{H}{H + \varepsilon L} \right)^\rho.
$$ 
Formally, the function \( G \) converges pointwise to the function \( F \) on $\O_r(u)$ whether \( \rho \) or \( \varepsilon \) tends to $0^{+}$.  
For deriving Liouville theorem of equation \eqref{pqle}, we need to calculate the elliptic equation of $G$. Before we deal with it, we  give a beneficial remark.

\begin{rmk}\label{r4}
(A) The simplest auxiliary function is $F$ given by \eqref{F} with $\g=-1$ and $k=-\left(1-\frac{q}{2}\right)\b$. For this function, using Lemma \ref{p1}, one see that if $(p,q)\in \mathscr{D}_L(n)$ (recall the definition \eqref{LD}), then the leading coefficients $\S_1,\S_2,\S_3$ can be positive simultaneously as long as we choose suitable parameters $d,\b,\sigma$. However, this important fact can not yield the desired Liouville theorem. There are several reasons as follows:
\begin{itemize}
	\item in this case, we can not ensure that the maximal value point of $F$ locates at $\O_r(u)$ and so we can not use Lemma \ref{kl};
	
	\item the technical issue is that when we deal with gradient items by basic inequalities or Cauchy--Schwarz inequality, except in the case of \( q = 0 \), the gradient terms cannot be canceled out by the dominant terms;
	
	\item formally, a key distinction from the case of \( q = 0 \) lies in the fact that the auxiliary function \( F \) fails to distinguish between positive constant solutions and radially symmetric ground states—for in both cases, \( F \) reduces to a positive constant.  
	
\end{itemize}	
In summary, we introduce above perturbation to $F$
in an attempt to overcome the aforementioned difficulties.\\
(B) If we choose $F$ given by \eqref{F} with $\g>-1$ and $k=-\left(1+\frac{q\g}{2}\right)\b$, we also can derive a Liouville theorem in $q\le 1$ case because in this case, the $L^{\infty}$ estimate  of solution at the maximal value of $F$ is automatically obtained and when $\g\to -1^{+}$, the positivity for coefficient functions are still valid.  However, we miss this important fact in $1<q<2$ case. So, we use another type pertubation $G$ (defined by \eqref{g} with $\rho=1$).
Finally, we can derive the decay estimate of $G$ without the priori $L^{\infty}$ estimate of solution. One can see concrete details  in Section \ref{S5}.
\end{rmk}

Now, we present the elliptic equation of $G$.

\begin{lem}\label{eg}
	Let $u$ be a positive solution of \eqref{pqle} with $q\in[0,2)$ in the domain $\O$,  and define the functions $H,l,L, Z$ in the domain $\O_r(u)$ as \eqref{21} and  \eqref{22} respectively. For real numbers $d\ge 0,\b\in\mathbb{R}$, $k=-\left(1-\frac{q}{2}\right)\b$ and $\gamma=-1$, we define $F$ as \eqref{F0} and the auxiliary function $G$ as \eqref{g} with parameters $\e,\rho>0$. Then for any real numbers $\delta$ and real functions $\sigma,\tau$, we have
	\begin{eqnarray}\label{529}
		\Delta G&=&\left(1-\frac{q}{2}\right)W\left( \frac{H}{H + \varepsilon L} \right)^\rho \left(N(\b,\sigma)+\frac{\rho\e(H+dL)Z}{H+\e L}N(\delta,\tau)\right)\nonumber\\
		&&+W\left( \frac{H}{H + \varepsilon L} \right)^\rho\left\{\frac{\left(1-\frac{q}{2}\right)\sum_{j=1}^{10}I_jH^{10-j}L^{j-1}}{H(H+\e L)^2(H^2+\e HL+\rho\e HL+\rho\e dL^2)^2}\right.\nonumber\\
		&& -\frac{q}{2} \left(1 - \frac{q}{2}\right)^{-1} (H + dL) \left( P_1^2 |\nabla \ln G|^2 + 2 P_1 Q_1 \langle \nabla \ln G, \nabla \ln u \rangle \right) \nonumber\\
		&& + q \left(1 - \frac{q}{2}\right)^{-1} (H + dL) (k + (k+l)dZ) P_1 \langle \nabla \ln G, \nabla \ln u \rangle \nonumber\\
		&& + 2k (H + dL) P_1 \langle \nabla \ln G, \nabla \ln u \rangle \nonumber\\
		&& + (2(\beta-1)H + (2\sigma-q)L)(1 + dZ) P_1 \langle \nabla \ln G, \nabla \ln u \rangle \nonumber\\
		&& + 2\rho \frac{\varepsilon L}{H + \varepsilon L} (H + dL) \left( P_1 P_2 |\nabla \ln G|^2 + (P_1 Q_2 + P_2 Q_1) \langle \nabla \ln G, \nabla \ln u \rangle \right)\nonumber \\
		&& + \rho (\rho - 1) (H + dL) \left( \frac{\varepsilon L}{H + \varepsilon L} \right)^2 \left( P_2^2 |\nabla \ln G|^2 + 2 P_2 Q_2 \langle \nabla \ln G, \nabla \ln u \rangle \right) \nonumber\\
		&& + \rho (H + dL) \frac{2\varepsilon HL}{(H + \varepsilon L)^2} \left( -P_2 P_3 |\nabla \ln G|^2 - (P_2 Q_3 + P_3 Q_2) \langle \nabla \ln G, \nabla \ln u \rangle \right) \nonumber\\
		&& + 2\rho (H + dL)\left( \frac{\varepsilon L}{H + \varepsilon L} \right)^2  \left( -P_2 P_4 |\nabla \ln G|^2 - (P_2 Q_4 + P_4 Q_2) \langle \nabla \ln G, \nabla \ln u \rangle \right) \nonumber\\
		&& + \left(1 - \frac{q}{2}\right) \rho\e \frac{H + dL}{H + \varepsilon L} Z (2(\delta-1)H + (2\tau - q)L) P_3 \langle \nabla \ln G, \nabla \ln u \rangle \nonumber\\
		&& + \frac{q}{2} \left(1 - \frac{q}{2}\right) \rho\e \frac{H + dL}{H + \varepsilon L} L \left( P_3^2 |\nabla \ln G|^2 + 2 P_3 Q_3 \langle \nabla \ln G, \nabla \ln u \rangle \right) \nonumber\\
		&&\left. + ql\rho\e \frac{H + dL}{H + \varepsilon L} L P_3 \langle \nabla \ln G, \nabla \ln u \rangle\right\}
	\end{eqnarray}
on $\O_r(u)$, where $I_j$ ($1\le j\le 10$) are  functions on $d,\e,\rho,\b,\delta,\sigma,\tau$, whose expressions are given in Lemma \ref{a2}, $N(\cdot,\cdot)$ is defined by \eqref{N} and  functions $P_i$ and $Q_i$ ($1\le i\le4 $) is given by

$$
P_1 = \frac{1 + \varepsilon Z}{1 + \varepsilon Z + \rho \varepsilon Z + \rho \varepsilon dZ^2},\qquad
Q_1 = \frac{\rho \varepsilon (k+l) Z (1 + dZ)}{1 + \varepsilon Z + \rho \varepsilon Z + \rho \varepsilon dZ^2},
$$ 

$$
P_2 = \frac{(1 + dZ)(1 + \varepsilon Z)}{1 + \varepsilon Z + \rho \varepsilon Z + \rho \varepsilon dZ^2},\quad
Q_2 = (k+l)P_2,\quad P_3 = \left(1 - \frac{q}{2}\right)^{-1} P_2,
$$ 

$$
Q_3 = \left(1 - \frac{q}{2}\right)^{-1} \left[ -k + (k+l)Z \frac{\rho \varepsilon (1 + dZ) - d(1 + \varepsilon Z)}{1 + \varepsilon Z + \rho \varepsilon Z + \rho \varepsilon dZ^2} \right],
$$ 

$$
P_4 = \frac{q}{2} \left(1 - \frac{q}{2}\right)^{-1} P_2,
\qquad
Q_4 = l + \frac{q}{2} Q_3.
$$

\end{lem}

\begin{proof}
	For simplicity, we assume that all formulas below are calculated in the regular set $\O_r(u)$. By the definition of $G$ and  Leibniz rule, we have
	\begin{eqnarray}\label{428}
		\Delta G = \Delta F \cdot \left(\frac{H}{H+\e L}\right)^\rho + F \Delta \left(\frac{H}{H+\e L}\right)^\rho + 2 \left\langle \nabla F, \nabla \left(\frac{H}{H+\e L}\right)^\rho \right\rangle.
	\end{eqnarray}
	Notice that 
	\begin{eqnarray}\label{429}
		\nabla \left(\frac{H}{H+\e L}\right) &= &\frac{1}{ H+\e L}\nabla H - \frac{H}{(H+\e L)^2} \nabla (H+\e L) \nonumber\\
		&=& \frac{\e L}{(H+\e L)^2} \nabla H - \frac{\e H}{(H+\epsilon L)^2} \nabla L \nonumber\\
		&=& \frac{\e H L}{(H+\e L)^2} (\nabla \ln H - \nabla \ln L),
	\end{eqnarray}
	so
	\begin{equation}\label{430}
		2 \left\langle \nabla F, \nabla \left(\frac{H}{H+\e L}\right)^\rho \right\rangle = 2 \rho \left(\frac{H}{H+\e L}\right)^\rho \cdot \frac{\e L}{H+\e L} \left\langle \nabla F, \nabla \ln H - \nabla \ln L \right\rangle.
	\end{equation}
	Again, by Leibniz rule, we see
	\begin{eqnarray}\label{431}
		\Delta \left(\frac{H}{H+\e L}\right)^\rho &=& \rho \left(\frac{H}{H+\e L}\right)^{\rho-1} \Delta \left(\frac{H}{H+\e L}\right)\nonumber\\
		&& + \rho (\rho-1) \left(\frac{H}{H+\e L}\right)^{\rho-2} \left| \nabla \left(\frac{H}{H+\e L}\right) \right|^2.
	\end{eqnarray}
	Then we concretely compute these two items. Using \eqref{429}, we have 
	\begin{eqnarray}\label{432}
		&&\rho (\rho-1) \left(\frac{H}{H+\e L}\right)^{\rho-2} \left| \nabla \left(\frac{H}{H+\e L}\right) \right|^2\nonumber \\
		&=& \rho (\rho-1) \left(\frac{H}{H+\e L}\right)^\rho \left(\frac{\e L}{H+\e L}\right)^2 \left| \nabla \ln H - \nabla \ln L \right|^2
	\end{eqnarray}
and
	\begin{eqnarray}\label{433}
		\Delta \left(\frac{H}{H+\e L}\right)& =& \Delta H \cdot (H+\e L)^{-1}  + 2 \langle \nabla H, \nabla (H+\e L)^{-1} \rangle \nonumber\\
		&&+ H \left(-(H+\e L)^{-2} \Delta (H+\e L) + 2(H+\e L)^{-3} |\nabla (H+\e L)|^2 \right)\nonumber\\
		&=& \frac{\e L}{(H+\e L)^2} \Delta H - \frac{\e H}{(H+\e L)^2} \Delta L \nonumber\\
		&&+ 2H(H+\e L)^{-3} |\nabla (H+\e L)|^2 - 2(H+\e L)^{-2} \langle \nabla H, \nabla (H+\e L) \rangle \nonumber\\
		&=& \frac{\e L}{(H+\e L)^2} \Delta H - \frac{\e H}{(H+\e L)^2} \Delta L\nonumber\\
		&& + 2(H+\e L)^{-3} \left\langle \nabla (H+\e L), H \nabla (H+\e L) - (H+\e L) \nabla H \right\rangle \nonumber\\
		&=& \e (H+\e L)^{-2} (L\Delta H - H \Delta L) \nonumber\\
		&&+ 2 \e H L (H+\e L)^{-3} \left\langle \nabla (H+\e L), \nabla \ln L - \nabla \ln H \right\rangle.
	\end{eqnarray}
	Substituting \eqref{432} and \eqref{433} into \eqref{431} first, then substituting \eqref{431} and \eqref{430} into \eqref{428} yields
	\begin{eqnarray}\label{434}
		\Delta G = \left(\frac{H}{H+\varepsilon L}\right)^\rho &\times&\left[ \Delta F +  \frac{2\rho\varepsilon L}{H+\varepsilon L} \cdot F \left\langle \nabla \ln F, \nabla \ln H - \nabla \ln L \right\rangle \right.\nonumber \\
		&&+ \left. \rho (\rho-1) F \left(\frac{\varepsilon L}{H+\varepsilon L}\right)^2 \left| \nabla \ln H - \nabla \ln L \right|^2 \right. \nonumber\\
		&&+ \left. \rho F \left( \frac{\varepsilon Z}{H+\varepsilon L} \Delta H - \frac{\varepsilon}{H+\varepsilon L} \Delta L \right)\right.\nonumber\\
		&&\left.+\rho F \frac{2 \varepsilon L}{(H+\varepsilon L)^2} \left\langle \nabla (H+\varepsilon L), \nabla \ln L - \nabla \ln H \right\rangle  \right].
	\end{eqnarray}
	From Lemma \ref{kl} and Lemma \ref{cl}, we have 
	\begin{eqnarray}\label{435}
		\Delta F &=& \left(1 - \frac{q}{2}\right) W \left[ S_1 H^2 + S_2 HL + S_3 L^2 + N(\beta, \delta) \right]\nonumber \\
		&&+ HW \left\{ -\frac{q}{2} \left(1 - \frac{q}{2}\right)^{-1} (1+dZ) \left| \nabla \ln F \right|^2\right. \nonumber\\
		&&\left.+ q \left(1 - \frac{q}{2}\right)^{-1} (1+dZ) (k + (k + l) dZ) \left\langle \nabla \ln F, \nabla \ln u \right\rangle \right.\nonumber \\
		&&+ 2k (1+dZ) \left\langle \nabla \ln F, \nabla \ln u \right\rangle  \nonumber\\
		&&\left.+\left(2(\beta - 1)+ (2\sigma - q)Z \right) (1+dZ) \left\langle \nabla \ln F, \nabla \ln u \right\rangle \right\},
	\end{eqnarray}
	where $k,l,W,Z$ are defined as the assumptions of lemma, $N(\b,\sigma)$ is defined by \eqref{N} and $S_1,S_2,S_3$ is given by Lemma \ref{cl}. Substituting \eqref{435} into \eqref{434} yields
	\begin{eqnarray}\label{436}
		\Delta G&=&W \left(\frac{H}{H+\varepsilon L}\right)^\rho \left[ \left(1 - \frac{q}{2}\right) \left( S_1 H^2 + S_2 HL + S_3 L^2 + N(\beta, \delta) \right) \right. \nonumber\\
		&&+ H \left( -\frac{q}{2} \left(1 - \frac{q}{2}\right)^{-1} (1+dZ) \left| \nabla \ln F \right|^2\right.\nonumber\\
		&&\left. + q \left(1 - \frac{q}{2}\right)^{-1} (1+dZ) (k + ( k + l) dZ) \left\langle \nabla \ln F, \nabla \ln u \right\rangle \right)\nonumber \\
		&&+ 2k (1+dZ) \left\langle \nabla \ln F, \nabla \ln u \right\rangle\nonumber\\
		&& + \left( 2(\beta - 1) + (2\sigma -q)Z \right) (1+ dZ) \left\langle \nabla \ln F, \nabla \ln u \right\rangle \nonumber\\
		&&+ 2\rho \frac{\varepsilon L}{H+\varepsilon L} (H+dL) \left\langle \nabla \ln F, \nabla \ln H - \nabla \ln L \right\rangle\nonumber \\
		&&+ \rho (\rho-1) (H+dL) \left(\frac{\varepsilon L}{H+\varepsilon L}\right)^2 \left| \nabla \ln H - \nabla \ln L \right|^2 \nonumber\\
		&&\rho (H+dL) \left( \frac{\varepsilon Z}{H+\varepsilon L} \Delta H - \frac{\varepsilon}{H+\varepsilon L} \Delta L \right)\nonumber\\
		&&+ \rho (H+dL)  \frac{2\varepsilon HL}{(H+\varepsilon L)^2} \left\langle \nabla \ln H, \nabla \ln L - \nabla \ln H \right\rangle\nonumber\\
		&&+ 2\rho (H+dL)\left(\frac{\e L}{H+\e L}\right)^2\left\langle \nabla \ln L, \nabla \ln L - \nabla \ln H \right\rangle \Big].
	\end{eqnarray}
	Now, we deal with $\Delta L$ and the relations between gradients of functions as follows.
	\begin{eqnarray}\label{437}
		\Delta L &=& \Delta \left( H^{\frac{q}{2}} u^{l} \right) \nonumber\\
		&=& u^l \Delta \left( H^{\frac{q}{2}} \right) + 2 \langle \nabla \left( H^{\frac{q}{2}} \right), \nabla (u^{l}) \rangle + H^{\frac{q}{2}} \Delta (u^{l}) \nonumber\\
		&=& u^{l} \left[ \frac{q}{2}  H^{\frac{q}{2}-1} \Delta H + \frac{q}{2} \left( \frac{q}{2} - 1 \right) H^{\frac{q}{2} - 2} |\nabla H|^2 \right]  \nonumber\\
		&&+ q lH^{\frac{q}{2}}u^l \langle \nabla \ln H, \nabla \ln u \rangle+H^{\frac{q}{2}} \left( l u^{l-1} \Delta u + l (l-1) u^{l-2} |\nabla u|^2 \right) \nonumber\\
		&=& \frac{q}{2} Z \Delta H + \frac{q}{2} \left(\frac{q}{2} - 1\right) L |\nabla \ln H|^2 \nonumber\\
		&&+ q l L \left\langle \nabla \ln H, \nabla \ln u \right\rangle + H^{\frac{q}{2}} \left( -l u^l L + \ell (l-1) u^l H \right) \nonumber\\
		&=& \frac{q}{2} Z \Delta H + \frac{q}{2} \left(\frac{q}{2} - 1\right) L |\nabla \ln H|^2 \nonumber\\
		&&+ ql L \left\langle \nabla \ln H, \nabla \ln u \right\rangle - l L^2 + l (l-1) H L.
	\end{eqnarray}
	Further, using Lemma \ref{21} and \eqref{437}, for any real functions $\delta,\tau$, we have 
	\begin{eqnarray}\label{438}
		 &&\frac{\varepsilon Z}{H+\varepsilon L} \Delta H - \frac{\varepsilon}{H+\varepsilon L} \Delta L\nonumber\\ &=&\frac{\varepsilon}{H+\varepsilon L} \left[ \left(1 - \frac{q}{2}\right) Z \left( A(\delta) H^2 + A(\tau) L^2 + B(\delta, \tau) HL + N(\delta, \tau) \right)\right. \nonumber\\
		 &&\left.+\left(1 - \frac{q}{2}\right) Z \left( 2(\delta - 1)H + (2\tau - q)L \right) \left\langle \nabla \ln H, \nabla \ln u \right\rangle\right. \nonumber\\
		 &&\left.+ \frac{q}{2} \left(1 - \frac{q}{2}\right) L \left| \nabla \ln H \right|^2 - ql L \left\langle \nabla \ln H, \nabla \ln u \right\rangle+l L^2 - l (l - 1) HL\right] .
	\end{eqnarray}
	Recall \eqref{grh} (with $\g=-1$), we see
	\begin{equation}\label{439}
		\nabla \ln H = \left(1-\frac{q}{2}\right)^{-1} (1+ dZ) \nabla \ln F - \left(1-\frac{q}{2}\right)^{-1} \left[ k + (k+l) dZ \right] \nabla \ln u.
	\end{equation}
	Therefore,
	\begin{eqnarray}
	\nabla \ln G &=& \nabla \ln F + \rho \nabla \ln \left(\frac{H}{H+\varepsilon L}\right)\nonumber \\
	&=& \nabla \ln F + \rho \left(\frac{H}{H+\varepsilon L}\right)^{-1} \nabla \left(\frac{H}{H+\varepsilon L}\right)\nonumber \\
	&=& \nabla \ln F + \rho \left(\frac{H}{H+\varepsilon L}\right)^{-1} \cdot \frac{\varepsilon HL}{(H+\varepsilon L)^2} \left(\nabla \ln H - \nabla \ln L\right)\nonumber \\
	&=& \nabla \ln F + \rho \frac{\varepsilon L}{H+\varepsilon L} \left(\nabla \ln H - \nabla \ln L\right) \nonumber\\
	&=& \nabla \ln F + \rho \frac{\varepsilon L}{H+\varepsilon L} \left(\left(1 - \frac{q}{2}\right) \nabla \ln H - l \nabla \ln u \right)\nonumber \\
	&=& \left(1 +  \frac{\rho \varepsilon L(1+dZ)}{H+\varepsilon L}\right) \nabla \ln F -  \frac{\rho \varepsilon L}{H+\varepsilon L}   (k+l) (1+dZ) \nabla \ln u,\nonumber
	\end{eqnarray}
	which is equivalent to 
	\begin{eqnarray}\label{440}
		\nabla\ln F=\frac{(1+\varepsilon Z) \nabla \ln G}{1+\varepsilon Z + \rho \varepsilon Z + \rho \varepsilon dZ^2} +  \frac{\rho \varepsilon(k+l) Z (1+dZ)}{1+\varepsilon Z + \rho \varepsilon Z + \rho \varepsilon dZ^2} \nabla \ln u.
	\end{eqnarray}
	From \eqref{439} and \eqref{440}, we have
	\begin{eqnarray}\label{441}
		&&\nabla \ln H - \nabla \ln L\nonumber\\
		 &=& \left(1 - \frac{q}{2}\right) \nabla \ln H - l \nabla \ln u\nonumber \\
		&=& (1+dZ) \nabla \ln F - (k+l)(1+dZ) \nabla \ln u\nonumber\\
		&=& \frac{(1+dZ)(1+\varepsilon Z) \nabla \ln G}{1+\varepsilon Z + \rho \varepsilon Z + \rho \varepsilon dZ^2} -  \frac{(k+l)(1+dZ)(1+\varepsilon Z)}{1+\varepsilon Z + \rho \varepsilon Z + \rho \varepsilon dZ^2} \nabla \ln u,
	\end{eqnarray}
	\begin{eqnarray}\label{442}
		\nabla\ln H&=&\left(1-\frac{q}{2}\right)^{-1}\frac{(1+dZ)(1+\varepsilon Z) \nabla \ln G}{1+\varepsilon Z + \rho \varepsilon Z + \rho \varepsilon dZ^2}\nonumber\\
		&&+\left(1 - \frac{q}{2}\right)^{-1} \left[ -k + (k+l)Z \frac{\rho \varepsilon (1 + dZ) - d(1 + \varepsilon Z)}{1 + \varepsilon Z + \rho \varepsilon Z + \rho \varepsilon dZ^2} \right]\nabla\ln u
	\end{eqnarray}
	and
	\begin{eqnarray}\label{443}
		\nabla\ln L&=&\frac{q}{2}\nabla\ln H+l\nabla\ln u\nonumber\\
		&=&\frac{q}{2-q}\frac{(1+dZ)(1+\varepsilon Z) \nabla \ln G}{1+\varepsilon Z + \rho \varepsilon Z + \rho \varepsilon dZ^2}\nonumber\\
		&&+\left[l+\frac{q}{2-q} \left(-k + (k+l)Z \frac{\rho \varepsilon (1 + dZ) - d(1 + \varepsilon Z)}{1 + \varepsilon Z + \rho \varepsilon Z + \rho \varepsilon dZ^2}\right)\right]\nabla\ln u.
	\end{eqnarray}
	If we define 
	$$
	P_1 = \frac{1 + \varepsilon Z}{1 + \varepsilon Z + \rho \varepsilon Z + \rho \varepsilon dZ^2},\qquad
	Q_1 = \frac{\rho \varepsilon (k+l) Z (1 + dZ)}{1 + \varepsilon Z + \rho \varepsilon Z + \rho \varepsilon dZ^2},
	$$ 
	
	$$
	P_2 = \frac{(1 + dZ)(1 + \varepsilon Z)}{1 + \varepsilon Z + \rho \varepsilon Z + \rho \varepsilon dZ^2},\quad
	Q_2 = (k+l)P_2,\quad P_3 = \left(1 - \frac{q}{2}\right)^{-1} P_2,
	$$ 
	
	$$
	Q_3 = \left(1 - \frac{q}{2}\right)^{-1} \left[ -k + (k+l)Z \frac{\rho \varepsilon (1 + dZ) - d(1 + \varepsilon Z)}{1 + \varepsilon Z + \rho \varepsilon Z + \rho \varepsilon dZ^2} \right],
	$$ 
	
	$$
	P_4 = \frac{q}{2} \left(1 - \frac{q}{2}\right)^{-1} P_2,
	\qquad
	Q_4 = l + \frac{q}{2} Q_3,
	$$ 
	we can rewrite \eqref{440}-\eqref{443} as
	\begin{eqnarray}
		\nabla\ln F&=&P_1\nabla\ln G+Q_1\nabla\ln u,\nonumber\\
	\nabla \ln H - \nabla \ln L	&=&P_2\nabla\ln G+Q_2\nabla\ln u,\nonumber\\
	\nabla\ln H	&=&P_3\nabla\ln G+Q_3\nabla\ln u,\nonumber\\
	\nabla\ln L	&=&P_4\nabla\ln G+Q_4\nabla\ln u.\nonumber
	\end{eqnarray}
	Substituting these equations into \eqref{436}, we have
	\begin{equation}
	\Delta G=\textbf{nonnegetive items}+\textbf{dominant terms}+\textbf{gradient items},\nonumber
	\end{equation}
	where 
	\begin{equation}
		\textbf{nonnegetive items}=\left(1-\frac{q}{2}\right)W\left( \frac{H}{H + \varepsilon L} \right)^\rho \left(N(\b,\sigma)+\frac{\rho\e(H+dL)Z}{H+\e L}N(\delta,\tau)\right),\nonumber
	\end{equation}

		\begin{eqnarray}
		\textbf{dominant terms}
		&=&W \left(\frac{H}{H+\varepsilon L}\right)^\rho \left\{ \left(1 - \frac{q}{2}\right) \left( S_1 H^2 + S_2 HL + S_3 L^2  \right) \right. \nonumber\\
		&&- \frac{q}{2-q} (k+l)^2 (\rho \e)^2 H(H+dL) \times\nonumber\\
		&& \left[ \frac{Z (1+ dZ)}{1 + \varepsilon Z + \rho \varepsilon Z + \rho \varepsilon dZ^2} \right]^2\nonumber \\
		&&+ q \left(1 - \frac{q}{2}\right)^{-1}(k+l) \rho \varepsilon H(H+dL) \times\nonumber\\
		&&\left( k + (k+l) dZ \right)  \frac{Z (1+dZ)}{1 + \varepsilon Z + \rho \varepsilon Z + \rho \varepsilon dZ^2}  \nonumber\\
		&&+ 2 k(k+l) \rho \varepsilon  \frac{Z (H+dL)^2}{1 + \varepsilon Z + \rho \varepsilon Z + \rho \varepsilon dZ^2} \nonumber\\
		&&+ \rho \varepsilon(k+l)\left( 2 (\b-1) H + (2 \sigma- q) \right)   L\times\nonumber\\
		&& \frac{(1+dZ)^2}{1 + \varepsilon Z + \rho \varepsilon Z + \rho \varepsilon dZ^2}  \nonumber\\
		&&- 2 \rho^2\left( k+l \right)^2 \frac{\varepsilon^2 L}{H+\varepsilon L} (H+dL)L \times\nonumber\\
		&&\frac{(1+dZ)^2 (1+\varepsilon Z)}{(1 + \varepsilon Z + \rho \varepsilon Z + \rho \varepsilon dZ^2)^2}  \nonumber\\
		&&+ \rho (\rho-1) (k+l)^2 H (H+dL) \times\nonumber\\
		&&\left( \frac{\varepsilon L}{H+\varepsilon L} \right)^2\left[ \frac{(1+dZ) (1+\varepsilon Z)}{1 + \varepsilon Z + \rho \varepsilon Z + \rho \varepsilon dZ^2} \right]^2  \nonumber\\
		&&+\left(1 - \frac{q}{2}\right)^{-1} \rho H (H+dL) \frac{2 \varepsilon HL}{(H+\varepsilon L)^2} \times\nonumber\\
		&& \left[ -k + (k+l) Z \frac{\rho \e(1+dZ) - d (1+\varepsilon Z)}{1 + \varepsilon Z + \rho \varepsilon Z + \rho \varepsilon dZ^2} \right] \nonumber\\
		&&\left. \times (k+l) \frac{(1+ dZ)(1+\varepsilon Z)}{1 + \varepsilon Z + \rho \varepsilon Z + \rho \varepsilon dZ^2}   \right. \nonumber\\
		&&+ \rho(k+l) H(H+dL) \frac{2 (\varepsilon L)^2}{(H+\varepsilon L)^2}  \frac{(1+dZ) (1+\varepsilon Z)}{1 + \varepsilon Z + \rho \varepsilon Z + \rho \varepsilon dZ^2} \times\nonumber \\
		&&\left[ l-\frac{qk}{2-q} + \frac{q}{2-q} (k+l) Z \frac{\rho\e (1+ dZ) - d (1+\varepsilon Z)}{1 + \varepsilon Z + \rho \varepsilon Z + \rho \varepsilon dZ^2} \right] \nonumber\\
		&&+\left(1 - \frac{q}{2}\right) \rho \varepsilon Z\frac{H+dL}{H+\varepsilon L}  \times\left[ A(\delta) H^2 + A(\tau) L^2 + B(\delta, \tau) HL \right. \nonumber\\
	&&+ \left(1 - \frac{q}{2}\right)^{-1}\left. (2(\delta-1)H + (2\tau-q)L) \times\right.\nonumber\\
	&& \left. \left(-k + (k+l)Z \frac{\rho \varepsilon (1+ dZ) - d (1+\varepsilon Z)}{1+\varepsilon Z+\rho \varepsilon Z+\rho \varepsilon dZ^2}\right) \right]\nonumber \\
		&&+ \frac{q\rho \varepsilon}{2-q}  \frac{H+dL}{H+\varepsilon L}  HL \left[ -k + (k+l)Z \frac{\rho \varepsilon (1+dZ) - d (1+\varepsilon Z)}{1+\varepsilon Z+\rho \varepsilon Z+\rho \varepsilon dZ^2} \right]^2 \nonumber\\
		&&- q \left(1 - \frac{q}{2}\right)^{-1} l \rho \varepsilon \frac{H+dL}{H+\varepsilon L}  HL\times\nonumber\\
		&& \left[ -k + (k+l)Z \frac{\rho \varepsilon (1+dZ) - d (1+\varepsilon Z)}{1+\varepsilon Z+\rho \varepsilon Z+\rho \varepsilon dZ^2} \right] \nonumber\\
		&&\left.+ \rho \varepsilon \frac{H+dL}{H+\varepsilon L} \left( l L^2 - l (l-1) HL \right)\right\}.\nonumber
	\end{eqnarray}
	
	\begin{eqnarray}
		\textbf{gradient items}	&=& W\left( \frac{H}{H + \varepsilon L} \right)^\rho\left\{-\frac{q}{2} \left(1 - \frac{q}{2}\right)^{-1} (H + dL)\times\right.\nonumber\\
		&& \left( P_1^2 |\nabla \ln G|^2 + 2 P_1 Q_1 \langle \nabla \ln G, \nabla \ln u \rangle \right) \nonumber\\
		&& + q \left(1 - \frac{q}{2}\right)^{-1} (H + dL) (k + (k+l)dZ) P_1 \langle \nabla \ln G, \nabla \ln u \rangle \nonumber\\
		&& + 2k (H + dL) P_1 \langle \nabla \ln G, \nabla \ln u \rangle \nonumber\\
		&& + (2(\beta-1)H + (2\sigma-q)L)(1 + dZ) P_1 \langle \nabla \ln G, \nabla \ln u \rangle \nonumber\\
		&& + 2\rho \frac{\varepsilon L}{H + \varepsilon L} (H + dL)\times\nonumber\\
		&& \left( P_1 P_2 |\nabla \ln G|^2 + (P_1 Q_2 + P_2 Q_1) \langle \nabla \ln G, \nabla \ln u \rangle \right)\nonumber \\
		&& + \rho (\rho - 1) (H + dL) \left( \frac{\varepsilon L}{H + \varepsilon L} \right)^2\times\nonumber\\
		&& \left( P_2^2 |\nabla \ln G|^2 + 2 P_2 Q_2 \langle \nabla \ln G, \nabla \ln u \rangle \right) \nonumber\\
		&& + \rho (H + dL) \frac{2\varepsilon HL}{(H + \varepsilon L)^2} \times\nonumber\\
		&&\left( -P_2 P_3 |\nabla \ln G|^2 - (P_2 Q_3 + P_3 Q_2) \langle \nabla \ln G, \nabla \ln u \rangle \right) \nonumber\\
		&& + 2\rho (H + dL)\left( \frac{\varepsilon L}{H + \varepsilon L} \right)^2 \times\nonumber\\
		&& \left( -P_2 P_4 |\nabla \ln G|^2 - (P_2 Q_4 + P_4 Q_2) \langle \nabla \ln G, \nabla \ln u \rangle \right) \nonumber\\
		&& + \left(1 - \frac{q}{2}\right) \rho\e \frac{H + dL}{H + \varepsilon L} Z\times\nonumber\\
		&& (2(\delta-1)H + (2\tau - q)L) P_3 \langle \nabla \ln G, \nabla \ln u \rangle \nonumber\\
		&& + \frac{q}{2} \left(1 - \frac{q}{2}\right) \rho\e \frac{H + dL}{H + \varepsilon L} L\times\nonumber\\
		&& \left( P_3^2 |\nabla \ln G|^2 + 2 P_3 Q_3 \langle \nabla \ln G, \nabla \ln u \rangle \right) \nonumber\\
		&&\left. + ql\rho\e \frac{H + dL}{H + \varepsilon L} L P_3 \langle \nabla \ln G, \nabla \ln u \rangle\right\}.\nonumber
	\end{eqnarray}

	If we combine like terms of \textbf{dominant terms}, we can represent it as 
	$$
	W\left( \frac{H}{H + \varepsilon L} \right)^\rho\frac{\left(1-\frac{q}{2}\right)\sum_{j=1}^{10}I_jH^{10-j}L^{j-1}}{H(H+\e L)^2(H^2+\e HL+\rho\e HL+\rho\e dL^2)^2},
	$$
	where $I_j$ $(1\le j\le 10)$ are functions of $\rho,\e,d,\b,\delta,\sigma,\tau$. Then we complete the proof.

\end{proof}

The factor
\begin{equation}\label{543}
\mathcal{D}:=\frac{\sum_{j=1}^{10}I_jH^{10-j}L^{j-1}}{H(H+\e L)^2(H^2+\e HL+\rho\e HL+\rho\e dL^2)^2} 
\end{equation}
in the dominant term plays a crucial role in the proof of the estimate. It should be noted that the numerator of this principal term is a 9th-degree polynomial, while the denominator is a 7th-degree polynomial with a common factor \( H \). Therefore, when the numerator is ``positive", the entire fraction is not merely a  ``positive" quadratic polynomial but an enhanced cubic rational fraction (with the denominator being \( H \)). This is completely different from the case when \( q = 0 \).
 It is very important that only this enhanced principal term can control all the gradient terms appearing in the equation \eqref{529}. The following two lemmas provide the lower bound estimates for the dominant term in its enhanced form. %The inequality conditions appearing in the Definition \ref{set} will be used in the proof of the following lemmas.

%The following two lemmas gives the positivity of leading coefficients of $\Delta F$ $S_1,\cdots,S_6$ under the condition of Theorem \ref{t27} and \ref{t28}  respectively.

\begin{lem}\label{dk1}
	Let $u$ be a positive solution of \eqref{pqle} with $q\in[0,2)$ in the domain $\O\subset\M$ ($n\ge 3$). Define $H,L$ as \eqref{21} and the function $Z$ in the domain $\O_r(u)$ as \eqref{22}. For the factor $\mathcal{D}$ given by \eqref{543}, if $(p,q)\in \mathbb{BL}(n)$ (whose definition is given in Definition \ref{set}) and $p+q\ge\frac{n}{n-2}$, there exist constants \(\e= d>0 \), \(\delta=\beta\in[0,\frac{2}{n-2}) \), \( \rho\in(0,1) \), $k=-\left(1-\frac{q}{2}\right)\b$, $\gamma=-1$ and $\kappa>0$, all of which depend only on \( n \), \( p \), and \( q \), as well as simple functions $\sigma,\tau$ on $\O$ such that
	$$
	\mathcal{D}\ge \kappa\left(1+Z\right)(H+L)^2\quad on \,\,\,\O_r(u).
	$$   
	
\end{lem}

\begin{lem}\label{dk2}
	Let $u$ be a positive solution of \eqref{pqle} with $q\in(1,2)$ in the domain $\O\subset\M$ ($n\ge 3$). Define $H,L$ as \eqref{21} and the function $Z$ in the domain $\O_r(u)$ as \eqref{22}. For the factor $\mathcal{D}$ given by \eqref{543}, if $(p,q)\in \mathbb{L}(n)$ (whose definition is given in Definition \ref{set}) and $p+q\ge\frac{n}{n-2}$, there exist constants \(\e>0, d>0 \), \(\delta=\beta\in[0,\frac{2}{n-2}) \), \( \rho=1 \), $k=-\left(1-\frac{q}{2}\right)\b$, $\gamma=-1$ and $\kappa>0$, all of which depend only on \( n \), \( p \), and \( q \), as well as simple functions $\sigma,\tau$ on $\O$ such that
$$
\mathcal{D}\ge \kappa\left(1+Z\right)(H+L)^2 \quad on \,\,\,\O_r(u).
$$   
\end{lem}
\vspace{3mm}
We give their proofs as follows.

\begin{proof}[Proof of Lemma \ref{dk1}]
	
	We will proceed with the long argument by dividing it into three cases. Regardless of the case, we always fix $\e=d$, $\g=-1$ and $k=-(1-\frac{q}{2})\b$ with an undetermined $\b$.

	Case 1. \( q \in \left[ 0, 1 - \frac{1}{\sqrt{n - 1}} \right] \). In this case,
	\[l=p+q-1\in\left[\frac{2}{n-2},\frac{(2-q)^2}{(1-q)(n-2)}\right).\] 
		First, by Lemma \ref{p1}, we can fix \( d > 0 \), \( \beta \in \left[ 0, \frac{2}{n - 1} \right) \) such that 
		\[\theta= \frac{d \left( l - \left( 1 - \frac{q}{2} \right) \beta \right)}{1 - \frac{q}{2}}\in\left(0,\frac{2}{n - (n-1)q}\right) ,\]
		 and choose  \( \sigma = \tau \equiv \frac{2 - n \theta}{2(n - 1)} \) such that (the definitions of $S_i$  are given by Lemma \ref{cl})
	\[
	S_i > 0, \quad i = 1, 2, 3.
	\]
Then by the formulas of \( I_1, \cdots, I_{10} \) in Lemma \ref{a3}, we easily see that when we set \( \varepsilon = d \),
	\[
	I_j > 0, \quad j = 1, 2, \cdots, 7 \quad \text{if } \rho \to 0^+.
	\]
		Now, we focus on the estimate of \( I_8, I_9,I_{10} \). By Lemma \ref{a3}, we see (the function $A$ is given by \eqref{A})
		\[
		I_8 = \rho  \varepsilon^4 d\left( 1 - \frac{q}{2} \right) \left[ A(\tau) + (q-1)\theta^2 + 2 \theta (\sigma - \tau) + 2 S_3 \right] + o(\rho),
		\]
		
		\[
		I_9 = \rho^2 \varepsilon^4 d^2 \left( 1 - \frac{q}{2} \right) \left[ 2 A(\tau) + q \theta^2 + 2 \theta (\sigma - \tau) + S_3 \right] + o(\rho^2),
		\]
		
		\[
		I_{10} = \rho^3 \varepsilon^4 d^3 \left( 1 - \frac{q}{2} \right) A(\tau).
		\]
		Notice that \( I_8 \) is positive when \( \rho \to 0^+ \) if
	\[
	A(\tau) + (q - 1) \theta^2 > 0.
	\]
	At this case, \( I_9 \) and \( I_{10} \) are obviously positive when \( \rho \to 0^+ \).
	
	Direct computation gives
	\[
	\begin{split}
		A(\tau) + (q - 1) \theta^2 &= \frac{2}{n} \left( 1 + \frac{2 - n \theta}{2(n - 1)} \right)^2 - 2 \left( \frac{2 - n \theta}{2(n - 1)} \right)^2 + (q - 1) \theta^2 \\
		&= \frac{2}{n - 1} + \left( q - 1 - \frac{n}{2(n - 1)} \right) \theta^2 \\
		&> \frac{2}{n - 1} + \left( q - 1 - \frac{n}{2(n - 1)} \right) \frac{4}{(n - (n - 1)q)^2} \\
		&= \frac{2}{(n - (n - 1)q)^2} \left[ (n - 1)(q - 1)^2 - 1 \right] \geq 0. \,\,\, \left( \text{by } q \leq 1 - \frac{1}{\sqrt{n - 1}} \right)
	\end{split}
	\]
	This is the only place where we use the condition $q \leq 1 - \frac{1}{\sqrt{n - 1}}$. Therefore, under these fixed parameters and a sufficiently small and positive $\rho$, we conclude that $I_j>0$, $j=1,2,\cdots,10$. So there exists a $\kappa=\kappa(n,p,q)>0$ such that
		$$
	\mathcal{D}\ge \kappa\left(1+Z\right)(H+L)^2\quad on \,\,\,\O_r(u).
	$$

	Case 2. $q\in(1-\frac{1}{\sqrt{n-1}},\frac{3}{2})$. In this case, 
	\[l=p+q-1<\mathcal{L}(n, q).\]
	From the definition of \( \mathcal{L}(n, q) \) (recall Definition \ref{set}) and the continuity, there exist \( \theta \in (0, 2) \), \( \tau_0 \in \left( -\frac{1}{\sqrt{n}+1}, 0 \right) \), and \( m \in (0, 2) \) which depend only on $n,p,q$ such that
	\[
	\sigma_0 := \frac{2 - n \theta}{2(n - 1)},
	\]
\begin{equation}\label{544}
	2 A(\tau_0) + 2 \theta (\sigma_0 - \tau_0) + q \theta^2 > 0,
\end{equation}

\begin{equation}\label{545}
	A(\sigma_0) + (q - 1 - 2 \sigma_0) \theta - \frac{q}{2} \theta^2 > 0,
\end{equation}

\begin{equation}\label{546}
	(q - 1) \theta^2 + 2 \theta (\sigma_0 - \tau_0) + (m + 2) \left( A(\sigma_0) + (q - 1 - 2 \sigma_0) \theta - \frac{q}{2} \theta^2 \right) + A(\tau_0) > 0,
\end{equation}
	\[
	\begin{split}
		l &= p + q - 1 < \frac{2 - q}{n - 2} + \frac{\theta (n - 2) + 2}{(n - 2)(2 - \theta)} \\
		&= f\left( \frac{2}{n - 2} \right),
	\end{split}
	\]
	where
	\[
	f(\beta) := \frac{1}{2 - \theta} \left[ (3 - q) \beta + \theta \left( 1 - \left( 1 - \frac{q}{2} \right) \beta \right) \right].
	\]
	
	%So, there exists  \( \beta_0=\b_0(n,p,q) \in \left( 0, \frac{2}{n - 2} \right) \) such that when $\b\in[\b_0,\frac{2}{n-2})$, \( l = p + q - 1 < f(\beta) \).
	Then for an undetermined $\b\in(0,\frac{2}{n-2})$, we define \( d \) as follows:
	\[
	\theta = \frac{d \left( l - \left( 1 - \frac{q}{2} \right) \beta \right)}{1 - \frac{q}{2}},
	\]
	and we claim that for these fixed constant \( \theta \) and functions
	\(
	\sigma \equiv \sigma_0, \quad \tau \equiv\tau_0,
	\) there exist $\delta=\b=\b(n,p,q)\in[0,\frac{2}{n-2})$, $d=d(n,p,q)>0$ and $\rho=\rho(n,p,q)>0$ such that
	\( S_i > 0 \), \( i = 1, 2, 3 \), and furthermore, \[ \sum_{j=1}^{10} I_j H^{10 - j} L^{j - 1} \geq C(n, p, q) (H^9 + L^9) .\]
	
	First, from the definition of $S_1,S_2$, we immediately have
	\[
	0 < \beta < \frac{2}{n - 2} \implies S_1 > 0,
	\]
	\[
	S_2 = \frac{4}{n }(1 + \b)(1 + \sigma) - 2 \beta \sigma - 2 l + (1 - q) \beta + \theta l + \theta \left( 1 - \left( 1 - \frac{q}{2} \right) \beta \right),
	\]
	when \( \beta = \frac{2}{n - 2} \),
	\[
	\begin{split}
		S_2 &= (2 - \theta) \left[ \frac{2 - q}{n - 2} + \frac{\theta (n - 2) + 2}{(n - 2)(2 - \theta)} - l \right] \\
		&= (2 - \theta) \left( f\left( \frac{2}{n - 2} \right) - l \right) > 0.
	\end{split}
	\]
	By the continuity again, there exists \( \beta = \beta(n, p, q) \in \left(0, \frac{2}{n - 2} \right) \) such that \(S_1>0, S_2 > 0 \).
	By \eqref{545} and the definition of $S_3$, we directly have
	\[
	S_3 = A(\sigma_0) + (q - 1 - 2 \sigma_0) \theta - \frac{q}{2} \theta^2 > 0.
	\]
	Up to now, for these fixed $\delta,\b,d,\sigma,\tau$, we have $S_i>0$ for $i=1,2,3$ and hence there exists a $\rho_0=\rho_0(n,p,q)>0$ such that when $\rho\in(0,\rho_0)$, $I_j>0$ for $j=1,2,\cdots,6$. We then estimate the remaining terms.
	
	\[
	\begin{split}
		&I_7 H^3 L^6 + I_8 H^2 L^7 + I_9 H L^8 \\
		&= \left[ \left( 1 -\frac{q}{2} \right) S_3 \varepsilon^4 + O(\rho) \right] H^3 L^6\\
		& \quad+ \left[ \rho d \varepsilon^4 \left( 1 - \frac{q}{2} \right) \left[ A(\tau_0) + (q - 1) \theta^2 + 2 \theta (\sigma_0 - \tau_0) + 2 S_3 \right] + O(\rho^2) \right] H^2 L^7 \\
		&\quad + \left[ \rho^2 d^2 \varepsilon^4 \left( 1 - \frac{q}{2} \right) \left[ 2 A(\tau_0) + q \theta^2 + 2 \theta (\sigma_0 - \tau_0) + S_3 \right] + O(\rho^3) \right] H L^8 \\
		&\geq \left[ \left( 1 - \frac{q}{2} \right) \left( 1 - \left( \frac{m}{2} \right)^2 \right) S_3 \varepsilon^4 + O(\rho) \right] H^3 L^6 \\
		&\quad + \left[ \rho d \varepsilon^4 \left( 1 - \frac{q}{2} \right) \left[ A(\tau_0) + (q - 1) \theta^2 + 2 \theta (\sigma_0 - \tau_0) + (m+2) S_3 \right] + O(\rho^2) \right] H^2 L^7 \\
		&\quad + \left[ \rho^2 d^2 \varepsilon^4 \left( 1 - \frac{q}{2} \right) \left[ 2 A(\tau_0) + q \theta^2 + 2 \theta (\sigma_0 - \tau_0) \right] + O(\rho^3) \right] H L^8 \\
		&\geq 0, \quad \text{when } \rho \text{ is sufficiently small.}
	\end{split}
	\]
	The first inequality above makes use of the Cauchy-Schwarz inequality, while the second one employs inequalities \eqref{544} and \eqref{546}. Therefore, there exists $\rho=\rho(n,p,q)>0$ such that 
	\[ \sum_{j=1}^{10} I_j H^{10 - j} L^{j - 1} \geq C(n, p, q) (H^9 + L^9) ,\]
and	 same estimate about $	\mathcal{D}$ is also obtained.

	Case 3. $q\in[\frac{3}{2},2)$. In this case, we choose \( \tau \) to be a function (rather than a constant) for the first time. The key role of this simple function is to take different values in the regions where \( H \) is greater than \( L \) and vice versa, thereby allowing them to control each other’s negative parts. Furthermore, this enables us to derive a uniform lower bound estimate for the numerator of \( \mathcal{D} \).

	First, we fix \( \beta = 0 \), \( \theta = 2 \), \(  d = \frac{\theta \left( 1 - \frac{q}{2} \right)}{l - \left( 1 - \frac{q}{2} \right) \beta} =\frac{2-q}{l}> 0 \) and \( \sigma \equiv -1 \) such that
	\[
	S_1 = \frac{2}{n}, \quad S_2 = 2, \quad S_3 = 0.
	\]
	 When \( \rho \) is positive and sufficiently small,
	\[
	I_j > 0, \quad j = 1, 2, \cdots, 6.
	\]
	For any undetermined \( \rho > 0 \), we have (by Lemma \ref{a3})
	\[
	I_8 = \rho d \varepsilon^4 \left( 1 - \frac{q}{2} \right) \left[ A(\tau) + 4(-\tau-1) + 4(q-1) \right] + O(\rho^2),
	\]
	\[
	I_9 = \rho^2 d^2 \varepsilon^4 \left( 1 - \frac{q}{2} \right) \left[ 2 A(\tau) + 4 q + 4(-\tau-1) \right] + O(\rho^3),
	\]
	\[
	I_{10} = \rho^3 d^3 \varepsilon^4 \left( 1 - \frac{q}{2} \right) A(\tau).
	\]

	Define
	\[\Large	\tau = - \mathbbm{1}_{\{ x \in \Omega: H > \chi \rho d L \}} + \tau_0 \mathbbm{1}_{\{ x \in \Omega: H \leq \chi \rho d L \}},
	\]
	where $\chi = \chi(n) = 1 + \frac{2}{\sqrt{n}+1}$, \( \tau_0 \in \left( -\frac{1}{\sqrt{n}+1}, 0 \right) \) satisfies \( A(\tau_0) > 0 \) and

	\[
	\frac{\frac{1}{2} - \tau_0}{\chi}  -\tau_0 - \frac{1}{2} > 0.
	\]
	At this time, 	 $I_j$ ($1\le j\le 10$) are simple functions with a range of at most two points.
We then carefully analysis the lower bound of $\sum_{j=1}^{10} I_j H^{10 - j} L^{j - 1}$.
	
	On the subset \( \{ x \in \Omega: H > \chi \rho d L \} \),
	\[
	\begin{split}
		&I_9 H L^8 + I_{10} L^9 \\
		&= \rho^2 \varepsilon^4 d^2 \left( 1 - \frac{q}{2} \right) \left[ \frac{4(q - 1)}{\sqrt{\chi}} + 4(q - 1) \left( 1 - \frac{1}{\sqrt{\chi}} \right) + O(\rho) \right] H L^8\\
		&\quad - 2 \rho^3 \varepsilon^4 d^3 \left( 1 - \frac{q}{2} \right) L^9 \\
		&\geq \rho^2 \varepsilon^4 d^2 \left( 1 - \frac{q}{2} \right) \left[ 4(q - 1) \left( 1 - \frac{1}{\sqrt{\chi}} \right)+ O(\rho) \right] H L^8 \\
		&\quad + \left[ 4(q - 1) \sqrt{\chi} - 2 \right] \rho^3 \varepsilon^4 d^3 \left( 1 - \frac{q}{2} \right) L^9.
	\end{split}
	\]
	
	%Notice that \( q \geq \frac{3}{2} \) and \( \chi \to 1 \).
	
	At this time, by Young's inequality,
	\[
	I_7 H^3 L^6 = O(\rho) H^3 L^6 \geq - \frac{1}{4} \left( 1 - \frac{q}{2} \right) S_2 \varepsilon^4 H^4 L^5 - O(\rho^3) H L^8,
	\]
	\[
	\begin{split}
		I_8 H^2 L^7 &= \rho d \varepsilon^4 \left( 1 - \frac{q}{2} \right) \left[ 4 \left( q - \frac{3}{2} \right) + O(\rho) \right] H^2 L^7 \\
		&\geq O(\rho^2) \varepsilon^4 d \left( 1 - \frac{q}{2} \right) H^2 L^7 \\
		&\geq - \frac{1}{4} \left( 1 - \frac{q}{2} \right) S_2 \varepsilon^4 H^4 L^5 - O(\rho^3) H L^8.
	\end{split}
	\]
	So (recall the $I_6$ in Lemma \ref{a3} and $q\ge\frac{3}{2}$)
	\[
	\begin{split}
		\sum_{j=6}^{10} I_j H^{10 - j} L^{j - 1} &\geq \left(\frac{1}{2} \left( 1 - \frac{q}{2} \right) S_2 \varepsilon^4+O(\rho)\right) H^4 L^5 \\
		&\quad + \rho^2 \varepsilon^4 d^2 \left( 1 - \frac{q}{2} \right) \left[ 4(q - 1) \left( 1 - \frac{1}{\sqrt{\chi}} \right) + O(\rho) \right] H L^8 \\
		&\quad + \left[ 4(q - 1) \sqrt{\chi} - 2 \right] \rho^3 \varepsilon^4 d^3 \left( 1 - \frac{q}{2} \right) L^9\\
		&\ge \left[ 2 \sqrt{\chi} - 2 \right] \rho^3 \varepsilon^4 d^3 \left( 1 - \frac{q}{2} \right) L^9 \quad \text{when } \rho \text{ is sufficiently small.}
	\end{split}
	\]
	Therefore, there exist \( \rho_1 = \rho_1(n, p, q) \) and \( \eta_1 = \eta_1(n,p,q, \rho)>0 \), such that when \( \rho \in (0, \rho_1] \), 
	\[
	\sum_{j=1}^{10} I_j H^{10 - j} L^{j - 1} \geq \eta_1 \left( H^9 + L^9 \right)
	\]
	on the subset \( \{ x \in \Omega: H > \chi \rho d L \} \).

	On the other hand, on the subset \( \{ x \in \Omega: H \leq \chi \rho d L \} \),
	\[
	\begin{split}
		&I_7 H^3 L^6 + I_8 H^2 L^7 + I_9 H L^8 \\
		&= O(\rho) H^3 L^6 + \rho d \varepsilon^4 \left( 1 - \frac{q}{2} \right) HL^7 \times \\
		&\quad \left[ \left( A(\tau_0) + 4(q-2 - \tau_0) + O(\rho) \right) H + \left( 2 A(\tau_0) + 4(q-1- \tau_0) + O(\rho) \right) \rho d L \right] \\
		&\geq \rho d \varepsilon^4 \left( 1 - \frac{q}{2} \right) H L^7 \left[ \left( 4 \left( \frac{1}{2} - \tau_0 \right) + O(\rho) \right) \rho d L - \left( 4 \left( \frac{1}{2} + \tau_0 \right) + O(\rho) \right) H \right] \\
		&\geq \rho d \varepsilon^4 \left( 1 - \frac{q}{2} \right) H L^7 \left[ 4 \left( \frac{\frac{1}{2} - \tau_0}{\chi} - \frac{1}{2} - \tau_0 +O(\rho)\right) \rho d L \right]\\
		&\ge 0, \quad \text{when } \rho \text{ is sufficiently small.}
	\end{split}
	\]
	Here, we use $$A(\tau_0)>0,\quad  \frac{\frac{1}{2} - \tau_0}{\chi} - \frac{1}{2} - \tau_0>0,$$ $q\ge \frac{3}{2}$ and $O(\rho) H^3 L^6\ge O(\rho^2)H^2L^7$ on \( \{ x \in \Omega: H \leq \chi \rho d L \} \).
	Therefore, there exist \( \rho_2 = \rho_2(n, p, q) \) and \( \eta_2 = \eta_2(n,p,q, \rho)>0 \), such that when \( \rho \in (0, \rho_2] \), 
	\[
	\sum_{j=1}^{10} I_j H^{10 - j} L^{j - 1} \geq \eta_2 \left( H^9 + L^9 \right)
	\]
	on the subset \( \{ x \in \Omega: H > \chi \rho d L \} \).
	
	Finally, we choose \( \rho = \min(\rho_1, \rho_2) \) and \( \eta = \min(\eta_1, \eta_2) \) such that 
	$	\sum_{j=1}^{10} I_j H^{10 - j} L^{j - 1} \geq \eta \left( H^9 + L^9 \right)$ on $\O$. Furthermore, we have $$
	\mathcal{D}\ge \kappa\left(1+Z\right)(H+L)^2 \quad on \,\,\,\O_r(u)
	$$   for some positive $\kappa=\kappa(n,p,q)$
	and finish the proof of Lemma \ref{dk1}.

\end{proof}

\begin{proof}[Proof of Lemma \ref{dk2}]
%The line of reasoning here is highly analogous to the proof of Lemma \ref{dk1}.  

		We  proceed with the argument by dividing it into two cases. Regardless of the case, we always fix $\rho=1$, $\g=-1$ and $k=-(1-\frac{q}{2})\b$ with an undetermined $\b$.

	Case 1. $q\in(1,\frac{5}{3}).$ The proof for this case is very similar to Case 2 in the proof of Lemma \ref{dk1}, except for differences in some algebraic computations.
	In this case, 
	\[l=p+q-1<\mathcal{H}(n, q).\]
	From the definition of \( \mathcal{H}(n, q) \) (recall Definition \ref{set}) and the continuity, there exist \( \theta \in (0, 2) \), \( \tau_0 \in \left( -\frac{1}{\sqrt{n}+1}, 0 \right) \), and \( m \in (0, 2) \) which depend only on $n,p,q$ such that
	\[
	\sigma_0 := \frac{2 - n \theta}{2(n - 1)},
	\]
	\begin{equation}\label{547}
		2 A(\tau_0) + 2 \theta (\sigma_0 - \tau_0) + q \theta^2 > 0,
	\end{equation}
	
	\begin{equation}\label{548}
		A(\sigma_0) + (q - 1 - 2 \sigma_0) \theta - \frac{q}{2} \theta^2 > 0,
	\end{equation}
	
	\begin{equation}\label{549}
		\left(\frac{3}{2}q-2\right) \theta^2 + 2 \theta (\sigma_0 - \tau_0) + (m + 2) \left( A(\sigma_0) + (q - 1 - 2 \sigma_0) \theta - \frac{q}{2} \theta^2 \right) + A(\tau_0) > 0,
	\end{equation}
	\[
		l = p + q - 1 < \frac{2 - q}{n - 2} + \frac{\theta (n - 2) + 2}{(n - 2)(2 - \theta)} 
		= f\left( \frac{2}{n - 2} \right),
	\]
	where
	\[
	f(\beta) := \frac{1}{2 - \theta} \left[ (3 - q) \beta + \theta \left( 1 - \left( 1 - \frac{q}{2} \right) \beta \right) \right].
	\]
		Then for undetermined $\b\in(0,\frac{2}{n-2})$, we define \( d \) by
	\[
	\theta = \frac{d \left( l - \left( 1 - \frac{q}{2} \right) \beta \right)}{1 - \frac{q}{2}}.
	\]
	By completely same argument as the Case 2 in the proof of Lemma \ref{dk1} we know that for \(
	\sigma \equiv \sigma_0
	\),  there exist
	 $\delta=\b=\b(n,p,q)\in[0,\frac{2}{n-2})$ and $d=d(n,p,q)>0$ such that
	\( S_i > 0 \), \( i = 1, 2, 3 \).
	Now we fix $\b,\delta,d,\sigma$, and the function  $\tau\equiv\tau_0$ and have (see the asymptotic behavior in Lemma \ref{a4})
	\[
	I_5 = \e d \left( 1 - \frac{q}{2} \right) \left[  A(\tau_0) + (\frac{3q}{2}-2) \theta^2 + 2 \theta (\sigma_0 - \tau_0) + 2S_3 + O(\e) \right],
	\]
	
	\[
	I_7 = \e^2 d^2 \left( 1 - \frac{q}{2} \right) \left[ 2A(\tau_0) + q \theta^2 + 2 \theta (\sigma_0 - \tau_0) +  S_3 + O(\e) \right].
	\]
	So when $\e$ is sufficiently small and positive, by Young's inequality, we have
	
		\[
	\begin{split}
		&I_3 H^7 L^2 + I_5 H^5 L^4 + I_7 H^3 L^6 \\
		&= \left[ \left( 1 -\frac{q}{2} \right) S_3 + O(\e) \right] H^7 L^2\\
		& \quad+  \e d \left( 1 - \frac{q}{2} \right) \left[ A(\tau_0) + (\frac{3q}{2} - 2) \theta^2 + 2 \theta (\sigma_0 - \tau_0) + 2 S_3+ O(\e) \right]  H^5 L^4 \\
		&\quad + \e^2 d^2  \left( 1 - \frac{q}{2} \right) \left[ 2 A(\tau_0) + q \theta^2 + 2 \theta (\sigma_0 - \tau_0) + S_3 + O(\e)\right]  H^3 L^6 \\
		&\geq \left[ \left( 1 - \frac{q}{2} \right) \left( 1 - \left( \frac{m}{2} \right)^2 \right) S_3 + O(\e) \right] H^7 L^2 \\
		& \quad+  \e d \left( 1 - \frac{q}{2} \right) \left[ A(\tau_0) + \left(\frac{3q}{2} - 2\right) \theta^2 + 2 \theta (\sigma_0 - \tau_0) + (2+m) S_3+ O(\e) \right]  H^5 L^4 \\
		&\quad + \e^2 d^2  \left( 1 - \frac{q}{2} \right) \left[ 2 A(\tau_0) + q \theta^2 + 2 \theta (\sigma_0 - \tau_0) +  O(\e)\right]  H^3 L^6 ,
	\end{split}
	\]
	
\begin{align*}
	I_4 H^6 L^3 &= O(\varepsilon) H^6 L^3 \geq -\frac{1}{4} \left(1 - \frac{q}{2}\right)\left(1 - \left(\frac{m}{2}\right)^2\right) S_3 H^7 L^2 + O(\varepsilon^2) H^5 L^4, \\
	I_6 H^4 L^5 &= O(\varepsilon^2) H^4 L^5 \geq -\frac{1}{4} \left(1 - \frac{q}{2}\right)\left(1 - \left(\frac{m}{2}\right)^2\right) S_3 H^7 L^2 + O(\varepsilon^4) H L^8, \\
	I_8 H^2 L^7 &= O(\varepsilon^3) H^2 L^7 \geq -\frac{1}{4} \left(1 - \frac{q}{2}\right)\left(1 - \left(\frac{m}{2}\right)^2\right) S_3 H^7 L^2 + O(\varepsilon^{\frac{18}{5}}) H L^8,
\end{align*}
	then combining these inequalities and using inequalities \eqref{547}--\eqref{549} will yield
	\begin{align*}
		\sum_{j=3}^{10} I_j H^{10 - j} L^{j - 1} &\geq \left[ \frac{1}{4} \left(1 - \frac{q}{2}\right) \left(1 - \left(\frac{m}{2}\right)^2\right) S_3 + O(\varepsilon) \right] H^7 L^2 \\
		&\quad + \varepsilon^3 d^3 \left(1 - \frac{q}{2}\right) \left[ A(\tau_0) + O\left(\varepsilon^{\frac{3}{5}}\right) \right] H L^8 \\
		&\quad + \varepsilon^4 d^3 \left(1 - \frac{q}{2}\right) A(\tau_0) L^9\\
		&\ge \varepsilon^4 d^3 \left(1 - \frac{q}{2}\right) A(\tau_0) L^9\quad \text{when } \e \text{ is sufficiently small.}
	\end{align*}
	Therefore, there exists $\e=\e(n,p,q)>0$ such that 
	\[ \sum_{j=1}^{10} I_j H^{10 - j} L^{j - 1} \geq C(n, p, q) (H^9 + L^9) ,\]
	and	 same estimate about $	\mathcal{D}$ is also obtained.

		Case 2. $q\in[\frac{5}{3},2)$.

	First, we fix \( \beta = 0 \), \( \theta = 2 \), \(  d = \frac{\theta \left( 1 - \frac{q}{2} \right)}{l - \left( 1 - \frac{q}{2} \right) \beta} =\frac{2-q}{l}> 0 \) and \( \sigma \equiv -1 \) such that
	\[
	S_1 = \frac{2}{n}, \quad S_2 = 2, \quad S_3 = 0.
	\]
	When \( \e \) is positive and sufficiently small,
	\[
	I_j > 0, \quad j = 1, 2.
	\]
	For any undetermined \( \e > 0 \), we have (by Lemma \ref{a4})
\[
I_5 = \varepsilon d \left( 1 - \frac{q}{2} \right) \left[ A(\tau) + 6\left(q-2-\frac{2}{3}\tau\right) +  O(\varepsilon)\right] ,
\]

\[
I_7 =  \varepsilon^2 d^2 \left( 1 - \frac{q}{2} \right) \left[ 2 A(\tau) + 4(q-1-\tau) + O(\varepsilon)\right],
\]

\[
I_9 = \varepsilon^3 d^3 \left( 1 - \frac{q}{2} \right) \left[ A(\tau) + O(\varepsilon) \right],
\]

\[
I_{10} = \varepsilon^4 d^3 \left( 1 - \frac{q}{2} \right) A(\tau) .
\]
	Define the function
	\[\Large	\tau = - \mathbbm{1}_{\{ x \in \Omega: H > \chi \sqrt{\e d} L \}} + \tau_0 \mathbbm{1}_{\{ x \in \Omega: H \leq \chi   \sqrt{\e d}L\}},
	\]
	where we set $\chi = \left(\frac{8}{7}\right)^{\frac{1}{3}}$, \( \tau_0 \in \left( -\frac{1}{\sqrt{n}+1}, 0 \right) \) such that \( A(\tau_0) > 0 \).
	
	At this time, 	 $I_j$ ($1\le j\le 10$) are simple functions with a range of at most two points.
	We then estimate the lower bound of $\sum_{j=1}^{10} I_j H^{10 - j} L^{j - 1}$.
	
		On the subset \( \{ x \in \Omega: H > \chi \sqrt{\e d} L \} \),
by Young's inequality, we have
	\[
	I_3 H^7 L^2 = O(\varepsilon) H^7 L^2 \geq O(\sqrt{\varepsilon}) H^8 L,\footnote{The inequality represnents that there exists $\e_0=\e_0(n,p,q)>0$ such that when $\e\in(0,\e_0)$, $I_3 H^7 L^2=f(\e)H^7 L^2, |f(\e)|\le C(n,p,q)\e,$ and so $f(\e) H^7 L^2\geq -C(n,p,q)\sqrt{\varepsilon} H^8 L.$ The other similar inequalities below are to be interpreted in the same manner.}
	\]
	
	\[
	I_4 H^6 L^3 = O(\varepsilon) H^6 L^3 \geq -\frac{1}{6} \left( 1 - \frac{q}{2} \right) S_2 H^8 L + O(\varepsilon^4) L^9,
	\]
	
	\[
	I_6 H^4 L^5 = O(\varepsilon^2) H^4 L^5 \geq -\frac{1}{6} \left( 1 - \frac{q}{2} \right) S_2 H^8 L + O(\varepsilon^4) L^9,
	\]
	
	\[
	I_8 H^2 L^7 = O(\varepsilon^3) H^2 L^7 \geq -\frac{1}{6} \left( 1 - \frac{q}{2} \right) S_2 H^8 L + O(\varepsilon^4) L^9,
	\]

	\begin{eqnarray}
		I_5 H^5 L^4= \varepsilon d \left( 1 - \frac{q}{2} \right) \left[  6\left(q-\frac{5}{3}\right) + O(\varepsilon)\right] \ge O(\e^2)H^5 L^4\ge O(\sqrt{\e})H^8 L,\nonumber
	\end{eqnarray}
	
	\[
	I_7 =  \varepsilon^2 d^2 \left( 1 - \frac{q}{2} \right) \left[  4(q-1) + O(\varepsilon)\right],
	\]
	
	\[
	I_9 = \varepsilon^3 d^3 \left( 1 - \frac{q}{2} \right) \left[ -2 + O(\varepsilon) \right],
	\]
	
		\[
	O(\varepsilon^4) L^9 \geq O\left( \varepsilon^{3} \right) H^2 L^7.
	\]
	So, when  $\e$ is sufficiently small, we have (recall $q\ge\frac{5}{3}$)
	\[
	\begin{split}
		&\sum_{j=2}^{10} I_j H^{10 - j} L^{j - 1} \\
		&\geq \left( \frac{1}{2} \left( 1 - \frac{q}{2} \right) S_2 + O(\sqrt{\varepsilon}) \right) H^8 L \\
		&\quad + \varepsilon^2 d^2 \left( 1 - \frac{q}{2} \right) \left[ 4(q - 1) + O(\varepsilon) \right] H^3 L^6 \\
		&\quad + \varepsilon^3 d^3 \left( 1 - \frac{q}{2} \right) \left[ -2 + O\left( \sqrt{\e} \right) \right] H L^8 \\
		&\geq \left( \frac{1}{2} \left( 1 - \frac{q}{2} \right) S_2 + O(\sqrt{\varepsilon}) \right) H^8 L \\
		&\quad + \varepsilon^2 d^2 \left( 1 - \frac{q}{2} \right) \left[ 4(q - 1) \left( 1 - \frac{1}{\chi} \right) + O(\varepsilon) \right] H^3 L^6 \\
		&\quad + \varepsilon^3 d^3 \left( 1 - \frac{q}{2} \right) \left[ 4(q - 1)\chi - 2 + O\left( \sqrt{\e}\right) \right] H L^8\\
		&\ge \frac{2}{3}\varepsilon^3 d^3 \left( 1 - \frac{q}{2} \right)H L^8\ge \frac{2\chi}{3}(\e d)^{\frac{7}{2}}\left( 1 - \frac{q}{2} \right)L^9.
	\end{split}
	\]
		Therefore, there exist \( \e_1 = \e_1(n, p, q) \) and \( \eta_1 = \eta_1(n,p,q, \e)>0 \), such that when \( \e \in (0, \e_1] \), 
	\[
	\sum_{j=1}^{10} I_j H^{10 - j} L^{j - 1} \geq \eta_1 \left( H^9 + L^9 \right)
	\]
	on the subset \( \{ x \in \Omega: H > \chi \sqrt{\e d} L \} \).
	
		On the other hand, on the subset  \( \{ x \in \Omega: H \le \chi \sqrt{\e d} L \} \), when $\e$ is positive and small,
	\[
	I_3 H^7 L^2 = O(\varepsilon) H^7 L^2 \geq O(\varepsilon^4) H L^8,
	\]
	
	\[
	I_4 H^6 L^3 = O(\varepsilon) H^6 L^3 \geq O\left( \varepsilon^{\frac{7}{2}} \right) H L^8,
	\]
	
	\[
	I_6 H^4 L^5 = O(\varepsilon^2) H^4 L^5 \geq O\left( \varepsilon^{\frac{7}{2}} \right) H L^8,
	\]
	
	\[
	I_8 H^2 L^7 = O(\varepsilon^3) H^2 L^7 \geq O\left( \varepsilon^{\frac{7}{2}} \right) H L^8,
	\]
	\[
	I_9 = \varepsilon^3 d^3 \left( 1 - \frac{q}{2} \right) \left[ A(\tau_0) + O(\varepsilon) \right]>0,
	\]
	
	\[
	I_{10} = \varepsilon^4 d^3 \left( 1 - \frac{q}{2} \right) A(\tau_0) >0,
	\]

	\[
	\begin{split}
		&I_5 H^5 L^4 + I_7 H^3 L^6 \\
		&\geq  \varepsilon d \left( 1 - \frac{q}{2} \right) \left[ 6(q - 2) + O(\varepsilon) \right] H^5 L^4 \\
		&\quad +  \varepsilon^2 d^2 \left( 1 - \frac{q}{2} \right) \left[ 4(q - 1) + O(\varepsilon) \right] H^3 L^6 \\
		&\geq  \varepsilon d \left( 1 - \frac{q}{2} \right) \left[ 6(q - 2) +  \frac{4(q - 1)}{\chi^3} + O(\varepsilon) \right] H^5 L^4 \\
		&\quad +  \varepsilon^2 d^2 \left( 1 - \frac{q}{2} \right) \left[ 4(q - 1)\left(1-\frac{1}{\chi}\right) + O(\varepsilon) \right] H^3 L^6 \\
		&\geq  \varepsilon d \left( 1 - \frac{q}{2} \right) \left[ \frac{1}{3} + O(\varepsilon) \right] H^5 L^4 \ge 0,	\end{split}
	\]
	and so 
		 there exist \( \e_2 = \e_2(n, p, q) \) and \( \eta_2 = \eta_2(n,p,q, \e)>0 \), such that when \( \e \in (0, \e_2] \), 
	\[
	\sum_{j=1}^{10} I_j H^{10 - j} L^{j - 1} \geq \eta_2 \left( H^9 + L^9 \right)
	\]
	on the subset \( \{ x \in \Omega: H > \chi \sqrt{\e d} L \} \).
	
		Finally, we choose \( \rho = \min(\rho_1, \rho_2) \) and \( \eta = \min(\eta_1, \eta_2) \) such that 
	$	\sum_{j=1}^{10} I_j H^{10 - j} L^{j - 1} \geq \eta \left( H^9 + L^9 \right)$ on $\O$. Furthermore, we have the desired estimate
	and finish the proof.

\end{proof}

\section{Proofs of Theorem \ref{2d}, \ref{t27} and \ref{t28}}\label{S5}
In this section, we give the proofs of Theorems \ref{2d}, \ref{t27} and \ref{t28}. Thus, Theorem \ref{lrm} is proven.  Due to Theorem \ref{t10}, we always assume $p+q-1\ge\frac{n+3}{n-1}\ge \frac{n}{n-2}$ ($n\ge 3$) in the whole section.

First, we need the following lower bounds for positive superharmonic functions on a annular domain in the complete Riemannian manifold with non-negative Ricci curvature, see also \cite[Lemma 2.8]{CM} and \cite[Lemma 2.3]{SZ} for a similar estimate on the exterior domain.
	\begin{lem}\cite[Lemma 2.5]{LU3}\label{lb}
	Let $\left(\M, g\right)$ be an n-dimensional ($n\ge 2$) complete Riemannian manifold with $Ric\ge 0$. Let $u \in C(\overline{B(x_0,R) \backslash B(x_0,r)})\cap C^2(B(x_0,R) \backslash B(x_0,r))$ be a positive function satisfying $\d u\le 0$  on $B(x_0,R) \backslash B(x_0,r)$ with $r<R$. Then 
	$$
	u(x) \geq \min_{\partial B(x_0,r)}u\left(\left(\frac{r}{r(x)}\right)^{\a(n)}-\left(\frac{r}{R}\right)^{\a(n)}\right) \quad \text { on }  B(x_0,R) \backslash B(x_0,r),
	$$
	where $r(x)$ is the distance from $x_0$ to $x$, $\a(n)=\a>0$ for $n=2$ and $\a(n)=n-2$ for $n\ge 3$. 
\end{lem}

First, we give the proof of Theorem \ref{2d}, which yields the Liouville theorem in two dimension case (though the Liouville theorem can be directly derived from the parabolicity of the manifold).

\begin{proof}[Proof of Theorem \ref{2d}]
	Let $u$ be a positive solution of \eqref{pqle} on $B(x_0,2R)\subset\mathcal{M}^2$. We define the function $F$ by \eqref{F} with $\g=-1$, $d=0$ and $k=-\b$ for some undetermined $\b>0$. By Lemma \ref{kl} and Lemma \ref{a1}, we immediately have (choose $\sigma\equiv 0$)
	\begin{eqnarray}
		\d F&\ge& u^{-\b}\left((1+\b)H^2+((3-q)\b-2l)HL+L^2\right)\nonumber\\
		&&-u^{-\b}[2H+qL]\left\langle\nabla\ln F,\nabla\ln u\right\rangle \quad\text{on}\quad \{x\in B(x_0,2R):|\nabla u|>0\},\nonumber
	\end{eqnarray}
	where $H,L,l$ are give in \eqref{21} and we use the condition $Ric\ge 0$. We fix $\b=\max(1,\frac{2|l|}{3-q})$ (notice that $q\in[0,2)$) and we have 
	\begin{equation}\label{6200}
			\d F\ge u^{-\b}\left(H^2+L^2\right)-u^{-\b}(2H+qL)\left\langle\nabla\ln F,\nabla\ln u\right\rangle 
	\end{equation}
	on $\{x\in B(x_0,2R):|\nabla u|>0\}$. Now we define 
	\begin{equation}
		A(x)=\P(x)F(x),\qquad\text{on $B(x_0,2R)$},\nonumber
	\end{equation}
	where $\P(x)$ is an undetermined cut-off function as in Lemma \ref{cut}. Then by Leibniz rule, we see
	\begin{equation}
		\d A=\frac{\d \P}{\P}A+2\left\langle\nabla \P,\nabla H\right\rangle+\P\d H,\nonumber
	\end{equation}
	on $\{x\in B(x_0,2R):A(x)>0\,\, \text{and}\,\, x\notin cut(x_0)\}$, where
	$cut(x_0)$ is the cut locus of $x_0$.
	Without loss of generality, we assume  the maximum value point $\x$ of $A$ is positive (or the proof is finished) and is outside of cut locus of $x_0$.
	At the point $\x$, by the chain rule $\nabla A(\x)=\left(\nabla\P H+\P\nabla H\right)(\x)=0$ and using \eqref{6200}, we obtain
	\begin{eqnarray}\label{63}
		0&\ge&\d A=\left(\frac{\d \P}{\P}-2\frac{|\nabla\P|^2}{\P^2}\right)A+\P\Delta F\nonumber\\
		&\ge &\left(\frac{\d \P}{\P}-2\frac{|\nabla\P|^2}{\P^2}\right)A+u^{-\b}\P\left(H^2+L^2\right)\nonumber\\
		&&+u^{-\b}(2H+qL)\left\langle\nabla\P,\nabla\ln u\right\rangle \nonumber\\
		&\ge& \left(\d \P-2\frac{|\nabla\P|^2}{\P}\right)u^{-\b}H+u^{-\b}\P\left(H^2+L^2\right)\nonumber\\
		&&-u^{-\b}(2H+qL)\left(\frac{1}{8}\P H+2\frac{|\nabla\P|^2}{\P}\right),
	\end{eqnarray}
	which implies (by the property of cut-off function) \[\P H(\x)\le\frac{C(p,q)}{R^2}.\] Notice that $\x\in B(x_0,\frac{3}{2}R)\setminus B(x_0,R)$ or $\d A(\x)=\d F(\x)\ge u^{-\b}\left(H^2+L^2\right)(\x)>0$ (by \eqref{6200}), which yields a contradiction. Now,  choose $\a=\frac{1}{\b}>0$ and use Lemma \ref{lb}, we have for any $0<r\le R$, 
	\[
	u(\x)\ge  \min_{\partial B(x_0,r)}u\left(\left(\frac{r}{{\rm d}(x_0,\x)}\right)^{\a}-\left(\frac{r}{2R}\right)^{\a}\right)\ge C(\a)\left(\frac{r}{R}\right)^{\a}\min_{\partial B(x_0,r)}u.
	\]
	Therefore, \[ 
	\sup\limits_{B(x_0,R)}u^{-\b}|\nabla\ln u|^{2}\le A(\x)\le u^{-\b}(\x)\times \frac{C(p,q)}{R^2}\le \frac{C(p,q)\left(\min\limits_{\partial B(x_0,r)}u\right)^{-\b}}{Rr}.
	\]

\end{proof}

Then, we give the proof of Theorem \ref{t27}. Comparing with $q=0$ case, there are many new arguments emerging here. In particular, at maximal value points of auxiliary functions, the upper bound of maximal value of the solutions is automatically obtained when $q\le 1$.

\begin{proof}[Proof of Theorem \ref{t27}]
We just prove $q>0$ case because $q=0$ case is similar and simpler. To clarify the proof, we divide it into there steps.

Step 1: Setting of the auxiliary functions.

Let $u$ be a positive solution of \eqref{pqle} on $B(x_0,2R)\subset\mathcal{M}^n$ and $H,l,L,Z$ be given by \eqref{21} and \eqref{22}.
	Define the auxiliary function $G$ as \ref{g} with parameters $\b,d,\e,\rho$.
	Without of loss generality, we assume that 
	$\sup\limits_{B(x_0,R)}G>0$  or our proof is complete. 	Define
	\begin{equation}
		A(x)=\P(x)G(x)\qquad\text{on $B(x_0,2R)$},\nonumber
	\end{equation}
	where $\P(x)$ is a cut-off function as in Lemma \ref{cut}. Then as routine, we have
	\begin{equation}\label{53}
		\d A=\frac{\d \P}{\P}A+2\left\langle\nabla \P,\nabla G\right\rangle+\P\d G,
	\end{equation}
	on $\{x\in B(x_0,2R):A(x)>0\,\, \text{and}\,\, x\notin cut(x_0)\}$, where
	$cut(x_0)$ is the cut locus of $x_0$.
	Without loss of generality, we assume  the maximum value point $\x$ of $A$ is outside of cut locus of $x_0$.
	At the point $\x$, we have
	\begin{eqnarray}\label{640}
		0\ge\d A=\left(\frac{\d \P}{\P}-2\frac{|\nabla\P|^2}{\P^2}\right)A+\P\Delta G,
	\end{eqnarray}
	where we use $\nabla A(\x)=0$.
	\vspace{3mm}
	
	We next consider two cases.\vspace{4mm}
	
	Case 1. $q\in[0,1]$ and $(p,q)\in\mathbb{L}(n)$. 
	Notice that in this case, $(p,q)\in \mathbb{BL}(n)$.

	By Lemma \ref{dk1}, 
	we can choose $k=-\left(1-\frac{q}{2}\right)\b, \gamma=-1$,  there exist constants \(\e= d>0 \), \(\delta=\beta\in[0,\frac{2}{n-2}) \), \( \rho\in(0,1) \),   and $\kappa>0$, all of which depend only on \( n \), \( p \), and \( q \), as well as simple functions $\sigma,\tau$ on $\O$ such that
	$$
	\mathcal{D}\ge \kappa\left(1+Z\right)(H+L)^2\quad on\quad \{x\in B(x_0,2R):|\nabla u|>0\},.
	$$   
	where $\mathcal{D}$ is given by \eqref{543}. From Lemma \ref{eg}, using the condition $Ric\ge 0$, we have 
	\begin{eqnarray}\label{650}
		\d G \ge \kappa_0 W\left(\frac{H}{H+\e L}\right)^{\rho}(1+Z)(H+L)^2+\textbf{gradient items}
	\end{eqnarray}
	on $\{x\in B(x_0,2R):|\nabla u|>0\}$, where $W=u^{-\left(1-\frac{q}{2}\right)\b}H^{-\frac{q}{2}}$, $\kappa_0=(1-\frac{q}{2})\kappa$ and  \textbf{gradient items} is given in the proof of Lemma \ref{eg}.
	
	\vspace{4mm}
	
		Case 2. $q\in(1,2)$ and $(p,q)\in\mathbb{L}(n)$.

	By Lemma \ref{dk2}, 
	we can choose $k=-\left(1-\frac{q}{2}\right)\b, \gamma=-1$,  there exist constants \(d>0 \), \(\delta=\beta\in[0,\frac{2}{n-2}) \), \( \rho=1, \e>0 \),   and $\kappa>0$, all of which depend only on \( n \), \( p \), and \( q \), as well as simple functions $\sigma,\tau$ on $\O$ such that
	$$
	\mathcal{D}\ge \kappa\left(1+Z\right)(H+L)^2\quad on\quad \{x\in B(x_0,2R):|\nabla u|>0\},
	$$   
	where $\mathcal{D}$ is given by \eqref{543}. By same reason as Case 1, inequality \eqref{650} is also valid.

	\vspace{3mm}

Step 2: Deal with gradient terms.
	
	For estimating the item $\P\d G$, we need to deal with the item $\P\times\textbf{gradient items}$, where  \textbf{gradient items} is given in the proof of Lemma \ref{eg}. First, from the definitions of $P_i,Q_i$ ($1\le i\le 4$), we have (the constant $C(n,p,q)$ may be different)
	\begin{align*}
		|P_1| &\leq \frac{C(n,p,q)}{1 + Z}, \\
		|P_i| &\leq C(n,p,q), \quad i = 2,3,4, \\
		|Q_i| &\leq C(n,p,q), \quad i = 1,2,3,4.
	\end{align*}
	
	We then estimate each term of $\P\times\textbf{gradient items}$ at the point $\x$ in detail as follows (here we factor out the common factor $W\left(\frac{H}{H+\e L}\right)^{\rho}$). 
	
	\begin{align*}
		& -\frac{q}{2} \left(1 - \frac{q}{2}\right)^{-1} (H + dL) \left( P_1^2 \P|\nabla \ln G|^2 + 2 P_1 Q_1\P \langle \nabla \ln G, \nabla \ln u \rangle \right)\\
		&=-\frac{q}{2} \left(1 - \frac{q}{2}\right)^{-1} (H + dL) \left( P_1^2 \frac{|\nabla \P|^2}{\P} - 2 P_1 Q_1 \langle \nabla \P, \nabla \ln u \rangle \right)\\
		&\ge-C(n,p,q)H\left(\frac{|\nabla \P|^2}{\P}+|\nabla\P|H^{\frac{1}{2}}\right)\\
		&\ge -\frac{\kappa_0}{32} \Phi H^2 - C(n,p,q) H \frac{|\nabla \Phi|^2}{\Phi},\\
		&\\
		& q\left(1 - \frac{q}{2}\right)^{-1}(H + dL)(k + (k + l)dZ) P_1 \Phi \langle \nabla \ln G, \nabla \ln u \rangle \\
		&= -q\left(1 - \frac{q}{2}\right)^{-1}(H + dL)(k + (k + l)dZ) P_1 \langle \nabla \Phi, \nabla \ln u \rangle \\
		& \geq -C(n,p,q) (H + L) H^{\frac{1}{2}} |\nabla \Phi| \\
		& \geq -\frac{\kappa_0}{32} \P(H + L)^2  - C(n,p,q) H \frac{|\nabla \Phi|^2}{\Phi}, \\
		& \\
		& 2k(H + dL) P_1 \Phi \langle \nabla \ln G, \nabla \ln u \rangle \\
		&= -2k(H + dL) P_1 \langle \nabla \Phi, \nabla \ln u \rangle \\
		& \geq -C(n,p,q) H^{\frac{3}{2}} |\nabla \Phi| \geq -\frac{\kappa_0}{32} \Phi H^2 - C(n,p,q) H \frac{|\nabla \Phi|^2}{\Phi},
	\end{align*}

	\begin{align*}
		&(2(\b - 1)H + (2\delta - q)L)(1 + dZ)P_1\Phi\langle\nabla\ln G, \nabla\ln u\rangle \\
		&=-(2(\b - 1)H + (2\delta - q)L)(1 + dZ)P_1\langle\nabla\Phi, \nabla\ln u\rangle \\
		&\geq -C(n,p,q)(H + L)H^{\frac{1}{2}}|\nabla\Phi| \\
		&\geq -\frac{\kappa_0}{32}\Phi(H + L)^2 - C(n,p,q)H\frac{|\nabla\Phi|^2}{\Phi}, \\
		& \\
		&2\rho \frac{\varepsilon L}{H + \e L}(H + dL)\left(P_1P_2\Phi|\nabla\ln G|^2 + (P_1Q_2 + Q_1P_2)\Phi\langle\nabla\ln G, \nabla\ln u\rangle\right) \\
		&= 2\rho \frac{\varepsilon L}{H + \e L}(H + dL)\left(P_1P_2\frac{|\nabla\Phi|^2}{\Phi} - (P_1Q_2 + Q_1P_2)\langle\nabla\Phi, \nabla\ln u\rangle\right) \\
		&\geq -C(n,p,q)(H + L)\left(\frac{|\nabla\Phi|^2}{\Phi} + H^{\frac{1}{2}}|\nabla\Phi|\right) \\
		&\geq -\frac{\kappa_0}{32}\P(H + L)^2 - C(n,p,q)(H + L)\frac{|\nabla\Phi|^2}{\Phi}, \\
		& \\
		&\rho(\rho  - 1)(H + dL)\left(\frac{\varepsilon L}{H + \e L}\right)^2\left(P_2^2\Phi|\nabla\ln G|^2 + 2P_2Q_2\Phi\langle\nabla\ln G, \nabla\ln u\rangle\right) \\
		&=\rho(\rho  - 1)(H + dL)\left(\frac{\varepsilon L}{H + \e L}\right)^2 \left(P_2^2\frac{|\nabla\Phi|^2}{\Phi} - 2P_2Q_2\langle\nabla\Phi, \nabla\ln u\rangle\right) \\
		&\geq -C(n,p,q)(H + L)\left(\frac{|\nabla\Phi|^2}{\Phi} + |\nabla\Phi|H^{\frac{1}{2}}\right) \\
		&\geq -\frac{\kappa_0}{32}\P(H + L)^2 - C(n,p,q)(H + L)\frac{|\nabla\Phi|^2}{\Phi}, \\
		& \\
		&\rho (H + dL)\frac{2\varepsilon HL}{(H + \varepsilon L)^2}\left(-P_2P_3\Phi|\nabla\ln G|^2 - (P_2Q_3 + P_3Q_2)\Phi\langle\nabla\ln G, \nabla\ln u\rangle\right) \\
		&= \rho (H + dL)\frac{2\varepsilon HL}{(H + \varepsilon L)^2}\left(-P_2P_3\frac{|\nabla\Phi|^2}{\Phi} + (P_2Q_3 + P_3Q_2)\langle\nabla\Phi, \nabla\ln u\rangle\right)\\
		&\geq -C(n,p,q)(H + L)\left(\frac{|\nabla\Phi|^2}{\Phi} + |\nabla\Phi|H^{\frac{1}{2}}\right) \\
		&\geq -\frac{\kappa_0}{32}\P(H + L)^2 - C(n,p,q)(H + L)\frac{|\nabla\Phi|^2}{\Phi},
	\end{align*}

	\begin{align*}
		&2\rho (H + dL)\left(\frac{\varepsilon L}{H + \varepsilon L}\right)^2\left(-P_2P_4\Phi|\nabla\ln G|^2 - (P_2Q_4 + P_4Q_2)\Phi\langle\nabla\ln G, \nabla\ln u\rangle\right) \\
		&= 2\rho(H + dL)\left(\frac{\varepsilon L}{H + \varepsilon L}\right)^2\left(-P_2P_4\frac{|\nabla\Phi|^2}{\Phi} + (P_2Q_4 + P_4Q_2)\langle\nabla\Phi, \nabla\ln u\rangle\right) \\
		&\geq -C(n,p,q)(H + L)\left(\frac{|\nabla\Phi|^2}{\Phi} + |\nabla\Phi|H^{\frac{1}{2}}\right) \\
		&\geq -\frac{\kappa_0}{32}\Phi(H + L)^2 - C(n,p,q)(H + L)\frac{|\nabla\Phi|^2}{\Phi}, \\
		& \\
		&\left(1 - \frac{q}{2}\right)\rho \varepsilon\frac{H + dL}{H + \varepsilon L}Z\left(2(\delta - 1)H + (2\tau - q)L\right)P_3\Phi\langle\nabla\ln G, \nabla\ln u\rangle \\
		&= -\left(1 - \frac{q}{2}\right)\rho \varepsilon\frac{H + dL}{H + \varepsilon L}Z\left(2(\delta - 1)H + (2\tau - q)L\right)P_3\langle\nabla\Phi, \nabla\ln u\rangle \\
		&\geq -C(n,p,q)Z(H + L)H^{\frac{1}{2}}|\nabla\Phi| \\
		&\geq -\frac{\kappa_0}{32}\Phi Z(H + L)^2 - C(n,p,q)L\frac{|\nabla\Phi|^2}{\Phi}, \\
		& \\
		&\frac{q}{2}\left(1 - \frac{q}{2}\right)\rho \varepsilon\frac{H + dL}{H + \varepsilon L}L\left(P_3^2\Phi|\nabla\ln G|^2 + 2P_3Q_3\Phi\langle\nabla\ln G, \nabla\ln u\rangle\right) \\
		&= \frac{q}{2}\left(1 - \frac{q}{2}\right)\rho \varepsilon\frac{H + dL}{H + \varepsilon L}L\left(P_3^2\frac{|\nabla\Phi|^2}{\Phi} - 2P_3Q_3\langle\nabla\Phi, \nabla\ln u\rangle\right) \\
		&\geq -C(n,p,q)L\left(\frac{|\nabla\Phi|^2}{\Phi} + |\nabla\Phi|H^{\frac{1}{2}}\right) \\
		&\geq -\frac{\kappa_0}{32}\Phi L^2 - C(n,p,q)(H + L)\frac{|\nabla\Phi|^2}{\Phi}, \\
		& \\
		& l q\rho \varepsilon\frac{H + dL}{H + \varepsilon L}LP_3\Phi\langle\nabla\ln G, \nabla\ln u\rangle \\
		&= -l q\rho \varepsilon\frac{H + dL}{H + \varepsilon L}LP_3\langle\nabla\Phi, \nabla\ln u\rangle \\
		&\geq -C(n,p,q)L|\nabla\Phi|H^{\frac{1}{2}}\geq -\frac{\kappa_0}{32}\Phi L^2 - C(n,p,q)H\frac{|\nabla\Phi|^2}{\Phi}.
	\end{align*}
Then combining above inequalities yields
	\[
	\P\times\textbf{gradient items}\ge -\frac{\kappa_0}{2}\Phi(H + L)^2 - C(n,p,q)(H + L)\frac{|\nabla\Phi|^2}{\Phi},
	\]
	and so
	\begin{equation}\label{660}
		\P\d G\ge W\left(\frac{H}{H+\e L}\right)^{\rho}\left[\frac{\kappa_0}{2}\P (1+Z)(H+L)^2- C(n,p,q)(H + L)\frac{|\nabla\Phi|^2}{\Phi}\right].
	\end{equation}

	Step 3: End of the proof.

	Combining \eqref{640} and \eqref{660} and using the property of cut-off function, we obtain
	\[
	0\ge  W\left(\frac{H}{H+\e L}\right)^{\rho}\left[\frac{\kappa_0}{2} \P(1+Z)(H+L)^2- C(n,p,q)(H + L)\frac{1}{R^2}\right],
	\]
	which implies
	\begin{equation}\label{670}
	\P(1+Z)(H+L)(\x)\le \frac{C(n,p,q)}{R^2}.
	\end{equation}
	
	Then we consider two cases in Step 1.
	\vspace{3mm}
	
	Case 1. $q\in[0,1]$.
	From \eqref{670}, we have (at the point $\x$)
	\begin{equation}\label{680}
		\P	L^2\le C(n,p,q)R^{-2}H\Rightarrow\P^{2-q} u^{2l}\le C(n,p,q)R^{-2}(\P H)^{1-q}\le C(n,p,q)R^{-2(2-q)}.
	\end{equation}
	Notice that $\x\in B(x_0,\frac{3}{2}R)\setminus B(x_0,R)$ or
	by \eqref{640} and \eqref{650},
	\[
	0\ge \d A(\x)=\d G(\x)\ge\kappa_0 W\left(\frac{H}{H+\e L}\right)^{\rho}(1+Z)(H+L)^2(\x)>0
	\]
	 which yields a contradiction.
	So, by Lemma \ref{lb}
	we have for any $0<r\le R$, 
	\begin{equation}\label{690}
			u(\x)\ge  \min_{\partial B(x_0,r)}u\left(\left(\frac{r}{{\rm d}(x_0,\x)}\right)^{n-2}-\left(\frac{r}{2R}\right)^{n-2}\right)\ge C(n)\left(\frac{r}{R}\right)^{n-2}\min_{\partial B(x_0,r)}u.
	\end{equation}
	Therefore, by \eqref{680} and \eqref{690}, we have
	\begin{align*}
		A(x^*) &= \Phi u^{-\left(1 - \frac{q}{2}\right)\beta} H^{-\frac{q}{2}} (H + dL) \left( \frac{H}{H + \varepsilon L} \right)^{\rho} (x^*) \\
		&\leq \Phi \left[ \left( u^{-\beta} H \right)^{\left(1 - \frac{q}{2}\right)} + d u^{l - \left(1 - \frac{q}{2}\right)\beta} \right] (x^*) \\
		&\leq \left[ \left( \Phi u^{-\beta} H \right)^{1-\frac{q}{2}} + d \left( \Phi^{\frac{2-q}{2l}} u \right)^{l - \left(1 - \frac{q}{2}\right)\beta} \right] (x^*)\quad \left(by\,\, \P\le 1,\,\, \frac{2-q}{2l}\times \left(l - \left(1 - \frac{q}{2}\right)\right)< 1\right) \\
		&\leq C(n,p,q) \left(R^{-(2-q)} \left(  \left( \frac{r}{R} \right)^{n - 2} \min_{\partial B(x_0,r)} u \right)^{-\left(1 - \frac{q}{2}\right)\beta} +  R^{-\frac{2 - q}{l} \times \left( l - \left(1 - \frac{q}{2}\right)\beta \right)} \right).
	\end{align*}
	In this case, $\e=d$, $\rho=\rho(n,p,q)\in(0,1)$ and so in $B(x_0,R)$,
	\[
	A=\P u^{-\left(1-\frac{q}{2}\right)\b}H^{-\frac{q}{2}}(H+dL)^{1-\rho}H^{\rho}\ge \P(u^{-\b}H)^{1-\frac{q}{2}}  .
	\]
	Finally,
		\begin{eqnarray}
		&&\sup\limits_{B(x_0,R)}u^{-(1-\frac{q}{2})\b}|\nabla\ln u|^{2-q}\nonumber\\
		&\le& \sup\limits_{B(x_0,2R)}\P u^{-(1-\frac{q}{2})\b}|\nabla\ln u|^{2-q}\le\sup\limits_{B(x_0,2R)}A= A(\x)\nonumber\\
		&\le& C(n,p,q)\left(R^{-(2-(n-2)\b)(1-\frac{q}{2})}\left(r^{n-2}\min_{\partial B(x_0,r)}u\right)^{-\left(1-\frac{q}{2}\right)\b}+R^{-\eta}\right),\nonumber
	\end{eqnarray}
	where $\eta=\frac{2 - q}{l} \times \left( l - \left(1 - \frac{q}{2}\right)\beta \right)$.
	
	\vspace{3mm}
	
	Case 2. $q\in(1,2)$. Currently, we cannot directly derive \eqref{680} from \eqref{670}; however, we have  $\rho = 1$. Then we directly have 
	\vspace{3mm}
	
		\begin{eqnarray}
			A(x^*) &=& \Phi u^{-\left(1 - \frac{q}{2}\right)\beta} H^{-\frac{q}{2}} (H + dL)  \frac{H}{H + \varepsilon L}  (x^*)\nonumber\\
			&\leq& C(n,p,q) \left( \Phi u^{-\beta} H \right)^{1-\frac{q}{2}}(x^*) \nonumber\\
			&\leq& C(n,p,q) R^{-(2-q)} \left(  \left( \frac{r}{R} \right)^{n - 2} \min_{\partial B(x_0,r)} u \right)^{-\left(1 - \frac{q}{2}\right)\beta} \nonumber
		\end{eqnarray}
	and 
	\begin{eqnarray}
		&&\sup\limits_{B(x_0,R)}u^{-(1-\frac{q}{2})\b}|\nabla\ln u|^{2-q}\nonumber\\
		&\le& \sup\limits_{B(x_0,2R)}\P \left(u^{-(1-\frac{q}{2})}H\right)^{-\b}\le C(n,p,q)\sup\limits_{B(x_0,2R)}A= C(n,p,q)A(\x)\nonumber\\
		&\le& C(n,p,q)R^{-(2-(n-2)\b)(1-\frac{q}{2})}\left(r^{n-2}\min_{\partial B(x_0,r)}u\right)^{-\left(1-\frac{q}{2}\right)\b}.
	\end{eqnarray}
	Whether in Case 1 or Case 2, we complete the proof of the theorem.

\end{proof}

Then we give the proof of Theorem \ref{t28}, which is nearly same as the proof of Theorem \ref{t27} except we do not have the upper bound \eqref{680} when $q\in(1,2)$.

\begin{proof}[Proof of Theorem \ref{t28}]
As the proof of Theorem \ref{t27}, let $u$ be a positive solution of \eqref{pqle} on $B(x_0,2R)\subset\mathcal{M}^n$ and $H,l,L,Z$ be given by \eqref{21} and \eqref{22}.
Define the auxiliary function $G$ as \ref{g} with parameters $\b,d,\e,\rho$.
Without of loss generality, we assume that 
$\sup\limits_{B(x_0,R)}G>0$  or our proof is complete. 	Define
\begin{equation}
	A(x)=\P(x)G(x)\qquad\text{on $B(x_0,2R)$},\nonumber
\end{equation}
where $\P(x)$ is a cut-off function as in Lemma \ref{cut}. Then as routine,  we assume  the maximum value point $\x$ of $A$ is outside of cut locus of $x_0$.
At the point $\x$, we have the estimate \eqref{640}.

Now, $q\in(1,2)$ and $(p,q)\in\mathbb{BL}(n)$. 
By Lemma \ref{dk1} and completely same argument as the proof of Theorem \ref{t27}, 
we can choose $k=-\left(1-\frac{q}{2}\right)\b, \gamma=-1$,  there exist constants \(\e= d>0 \), \(\delta=\beta\in[0,\frac{2}{n-2}) \), \( \rho\in(0,1) \),   and $\kappa>0$, all of which depend only on \( n \), \( p \), and \( q \), as well as simple functions $\sigma,\tau$ on $\O$ such that
\[
\left(1-\rho\right)l-\left(1-\frac{q}{2}\right)\b>0\footnote{Because we assume $l\ge\frac{2}{n-2}$ and Lemma \ref{dk2} is always valid with sufficiently small and positive $\rho$. }
\]
and the estimate \eqref{650} is valid. Same reason as Case 2 in the proof of Theorem \ref{t27} yields \eqref{670}. However, we cannot yields \eqref{680} from \eqref{670} in this case. Anyway, we have
\begin{eqnarray}
	A(\x)&=&\P u^{-\left(1-\frac{q}{2}\right)\b}H^{-\frac{q}{2}}(H+dL)^{1-\rho}H^{\rho}\nonumber\\
	&\le&\P u^{-\left(1-\frac{q}{2}\right)\b}H^{-\frac{q}{2}}\left(H+d^{1-\rho}H^{\rho}L^{1-\rho}\right)\nonumber\\
	&\le&C(n,p,q)\P\left[\left(u^{-\b}H\right)^{1-\frac{q}{2}}+H^{\left(1-\frac{q}{2}\right)\rho}u^{l(1-\rho)-\left(1-\frac{q}{2}\right)\b}\right]\nonumber\\
	&\le&C(n,p,q)\left[\left(\P u^{-\b}H\right)^{1-\frac{q}{2}}+(\P H)^{\left(1-\frac{q}{2}\right)\rho}\Vert u\Vert_{L^{\infty}(B(x_0,2R))}^{l(1-\rho)-\left(1-\frac{q}{2}\right)\b}\right]\nonumber\\
		&\le& C(n,p,q)R^{-(2-(n-2)\b)(1-\frac{q}{2})}\left(r^{n-2}\min_{\partial B(x_0,r)}u\right)^{-\left(1-\frac{q}{2}\right)\b}\nonumber\\
		&&+C(n,p,q)R^{-(2-q)\rho}\Vert u\Vert^{(1-\rho)(p+q-1)-(1-\frac{q}{2})\b}_{L^{\infty}(B(0,2R))},
\end{eqnarray}
where we use the estimate \eqref{690} at last inequality.
Finally, 
\begin{eqnarray}
	&&\sup\limits_{B(x_0,R)}\left(u^{-(1-\frac{q}{2})\b}|\nabla\ln u|^{2-q}+u^{(1-\rho)(p+q-1)-(1-\frac{q}{2})\b}|\nabla\ln u|^{(2-q)\rho}\right)\nonumber\\
	&\le& C(n,p,q)\sup\limits_{B(x_0,2R)}A\le C(n,p,q)A(\x)\nonumber\\
		&\le&C(n,p,q)R^{-(2-(n-2)\b)(1-\frac{q}{2})}\left(r^{n-2}\min_{\partial B(x_0,r)}u\right)^{-\left(1-\frac{q}{2}\right)\b}\nonumber\\
	&&+C(n,p,q)R^{-(2-q)\rho}\Vert u\Vert^{(1-\rho)(p+q-1)-(1-\frac{q}{2})\b}_{L^{\infty}(B(0,2R))}.\nonumber
\end{eqnarray}
We complete the proof.

\end{proof}

We have already provided the proof of Theorem \ref{lrm} for the case when \( q \ge 2 \) in Section \ref{S3}; below, we present the proof for the case when \( 0\le q < 2 \).

\begin{proof}[Proof of Theorem \ref{lrm}]
	
	When $n=2$, Theorem \ref{lrm} is implied by Lemma \ref{2d} by letting $r=1$ and $R\to\infty$.
	
	When $n\ge 3$, we have:
	
	 \begin{itemize}
	 	\item if $p+q<\frac{n}{n-2}\le\frac{n+3}{n-1}$, Theorem \ref{lrm} is implied by Theorem \ref{t10};
	 	
	 	\item if $p+q\ge\frac{n}{n-2}$ and $(p,q)\in\mathbb{L}(n)$,  Theorem \ref{lrm} is implied by Theorem \ref{t27} by letting $r=1$ and $R\to\infty$; 
	 	
	 	\item if $p+q\ge\frac{n}{n-2}$ and $(p,q)\in\mathbb{BL}(n)$, the part $(2)$ of Theorem \ref{lrm} is implied by Theorem \ref{t27} by letting $r=1$ and $R\to\infty$.
	 \end{itemize}

\end{proof}

\section{Mutual control lemmas}\label{S6}
In this independent section, we provide the relations between items $H$ and $L$ (see the definitions of  them in \eqref{21}), which is important to the proof of Theorem \ref{LFG} in Section \ref{S7}. Concretely, their $L^{\infty}$ norm are controlled by each other.  Besides Euclidean case, we also consider  manifold case because they themselves are also interesting facts. Same idea already appeared in \cite[Section 5]{LU1} for $q=0$ case. Hence the content of this section can be seen as the partial generalizations of \cite[Section 5]{LU1}. We call the following two lemmas as mutual control lemmas.

\begin{lem}\label{control1}
	Let $(\M,g)$ be an $n$-dimensional complete Riemannian manifold with $Ric\ge-Kg$ and $K\ge 0$. Assume that $q\ge0$ and $p\in\mathbb{R}$. Let $u$ be a positive solution of equation \eqref{pqle} on $B(x_0,2R)\subset\M$ , then
	\begin{equation}\label{61}
		\sup\limits_{B(x_0,R)}\frac{|\nabla u|^2}{u^2}\le C(n,p,q)\left(\sup\limits_{B(x_0,2R)}u^{p-1}|\nabla u|^q+\frac{1}{R^2}+K\right).
	\end{equation}
\end{lem}

\begin{lem}\label{control2}
	Let $(\M,g)$ be an $n$-dimensional complete Riemannian manifold with $Ric\ge-Kg$ and $K\ge 0$. Assume that $q\ge0$ and $p\in\mathbb{R}$. Let $u$ be a positive solution of equation \eqref{pqle} on $B(x_0,2R)\subset\M$, then
	\begin{equation}\label{62}
		\sup\limits_{B(x_0,R)}u^{p-1}|\nabla u|^q\le C(n,p,q)\left(\sup\limits_{B(x_0,2R)}\frac{|\nabla u|^2}{u^2} +\frac{1}{R^2}+K\right).
	\end{equation}
\end{lem}

\begin{proof}[Proof of Lemma \ref{control1}]
	
	Let $u$ be a positive solution of equation \eqref{pqle} on $B(x_0,2R)\subset\M$ and $H,l,L$ are defined by \eqref{21}. By choosing $\b=\sigma=0$ in Lemma \ref{l2} and using condition $Ric\ge -Kg$, then we have
	\begin{equation}\label{63}
	\Delta H\ge \frac{2}{n}\left(H^2+L^2\right)+\left(\frac{4}{n}-2l\right)HL-\left(2H+qL\right)\left\langle\nabla\ln u,\nabla \ln H\right\rangle-KH
	\end{equation}
	on $\{x\in B(x_0,2R):|\nabla u|(x)>0\}$.
	Define
	\begin{equation}
		A(x)=\P(x)H(x),\qquad\text{on $B(x_0,2R)$},
	\end{equation}
	where $\P(x)$ is an undetermined cut-off function as in Lemma \ref{cut}. Then
	\begin{equation}\label{65}
		\d A=\frac{\d \P}{\P}A+2\left\langle\nabla \P,\nabla H\right\rangle+\P\d H,
	\end{equation}
	on $N:=\{x\in B(x_0,2R):A(x)>0\,\, \text{and}\,\, x\notin cut(x_0)\}$, where
	$cut(x_0)$ is the cut locus of $x_0$.

	Without loss of generality, we assume $A(x^{\star})=\max\limits_{B(x_0,2R)} A>0$ and the maximum value point $\x$ is outside of cut locus of $x_0$.
	At the point $\x$, by \eqref{63} and \eqref{65}, we see
		\begin{eqnarray}\label{66}
		0\ge\d A&\ge&\left(\frac{\d \P}{\P}-2\frac{|\nabla\P|^2}{\P^2}\right)A+\frac{2}{n}\P\left(H^2+L^2\right)-2KA\nonumber\\
		&&+\left(\frac{4}{n}-2l\right)\P HL -\P\left(2 H+q L\right)\langle\nabla\ln u, \nabla \ln H\rangle\nonumber\\
	&\ge&\left(\frac{\d \P}{\P}-2\frac{|\nabla\P|^2}{\P^2}\right)A+\frac{2}{n}\P\left(H^2+L^2\right)-2KA\nonumber\\
	&&+\left(\frac{4}{n}-2l\right)\P HL +\left(2 H+q L\right)\langle\nabla\ln u, \nabla \P\rangle,
	\end{eqnarray}
	where we use $\nabla A(\x)=\left(\nabla\P H+\P\nabla H\right)(\x)=0$. By Cauchy-Schwarz inequality and basic inequality, we have
	$$
	2H\langle\nabla\ln u, \nabla \P\rangle\ge- \left(\frac{1}{n}\P H^2+n H\frac{|\nabla \P|^2}{\P}\right)
	$$
	and 
	$$
	qL\langle\nabla\ln u, \nabla \P\rangle\ge -\frac{q}{2}\P HL-\frac{q}{2}\frac{|\nabla\P|^2}{\P}L\ge -\frac{q}{2}\P HL-\frac{2}{n}L^2-\frac{nq^2}{2}\frac{|\nabla\P|^4}{\P^2}.
	$$
	Substituting these inequalities into \eqref{66} yields
	\begin{equation}\label{67}
	0\ge \left(\frac{\d \P}{\P}-(2+n)\frac{|\nabla\P|^2}{\P^2}\right)A+\frac{1}{n}\P H^2-2KA+\left(\frac{4}{n}-2l-\frac{q}{2}\right)\P HL-\frac{nq^2}{2}\frac{|\nabla\P|^4}{\P^2}.
	\end{equation}
	Multiplying both sides of \eqref{67} by $\P(\x)$ and  using the property of cut-off function in Lemma \ref{cut}, we see 
	\begin{eqnarray}\label{68}
		\max\limits_{B(x_0,2R)} A=A(\x)&\le& C(n,p,q)\left(L(\x)+\frac{1}{R^2}+K\right)\nonumber\\
		&\le& C(n,p,q)\left(\sup\limits_{B(x_0,2R)} L+\frac{1}{R^2}+K\right).
	\end{eqnarray}
Finally,	\eqref{68} implies desired \eqref{61} and we complete the proof.

\end{proof}

\begin{proof}[Proof of Theorem \ref{control2}]
We just prove $q>0$ case because $q=0$ case is similar and simpler.
Let $u$ be a positive solution of equation \eqref{pqle} on $B(x_0,2R)\subset\M$. Define $H,l,L$ as \eqref{21}, $Z$ as \eqref{22} on $\{x\in B(x_0,2R):|\nabla u|(x)>0\}$ and $F$ as \eqref{F}. By choosing $k=\g=\b=\delta=\sigma=\tau=0$, $d=1$ in Lemma \ref{kl} and the using condition $Ric\ge -Kg$, then we have
	\begin{eqnarray}\label{69}
	\d F
	&\ge&\frac{J\left(S_1H^5+S_2H^4L+S_3H^3L^2+S_4H^2L^3+S_5HL^4+S_6L^5\right)}{\left(H+\frac{dq}{2}L\right)^{3}}\nonumber\\
	&&-2K\left(H+\frac{q}{2}L\right) -\left(2H+q L\right) \frac{1+  Z}{J} \langle\nabla \ln F, \nabla \ln u\rangle\nonumber\\
	&&+\frac{1+Z}{J}L\left\{\frac{q}{2}\left(\frac{q}{2}-1\right)\frac{1+Z}{J}|\nabla\ln F|^2+lq\langle\nabla \ln F, \nabla \ln u\rangle\right.\nonumber\\
	&&\left.-q\left(\frac{q}{2}-1\right)\frac{lZ}{J}\langle\nabla \ln F, \nabla \ln u\rangle
	-\left(q+\frac{q^2}{2} Z\right)\langle\nabla \ln F, \nabla \ln u\rangle\right\}
\end{eqnarray}
on $\{x\in B(x_0,2R):|\nabla u|(x)>0\}$, where $ J=1+\frac{dq}{2}Z$ and $S_1,\cdots,S_6$ are defined as in Lemma \ref{kl}.

Now, we define
\begin{equation}
	A(x)=\P(x)F(x),\qquad\text{on $B(x_0,2R)$},
\end{equation}
where $\P(x)$ is an undetermined cut-off function as in Lemma \ref{cut}. Then
\begin{equation}\label{610}
	\d A=\frac{\d \P}{\P}A+2\left\langle\nabla \P,\nabla F\right\rangle+\P\d F,
\end{equation}
on $N:=\{x\in B(x_0,2R):A(x)>0\,\, \text{and}\,\, x\notin cut(x_0)\}$, where
$cut(x_0)$ is the cut locus of $x_0$.

Without loss of generality, we assume $A(x^{\star})=\max\limits_{B(x_0,2R)} A>0$ and the maximum value point $\x$ is outside of cut locus of $x_0$.
At the point $\x$, by \eqref{69} and \eqref{610}, we have (use $\nabla A(\x)=\left(\nabla\P H+\P\nabla H\right)(\x)=0$)

	\begin{eqnarray}\label{612}
	0\ge\d A&\ge&\frac{J\P\left(S_1H^5+S_2H^4L+S_3H^3L^2+S_4H^2L^3+S_5HL^4+S_6L^5\right)}{\left(H+\frac{q}{2}L\right)^{3}}\nonumber\\
	&&+\left(\frac{\d \P}{\P}-2\frac{|\nabla\P|^2}{\P^2}\right)A-2K\P\left(H+\frac{q}{2}L\right)+\left(2H+q L\right) \frac{1+  Z}{J} \langle\nabla \P, \nabla \ln u\rangle\nonumber\\
	&&+\frac{1+Z}{J}L\left\{\frac{q}{2}\left(\frac{q}{2}-1\right)\frac{1+Z}{J}\frac{|\nabla\P|^2}{\P}-lq\langle\nabla \P, \nabla \ln u\rangle\right.\nonumber\\
	&&\left.+q\left(\frac{q}{2}-1\right)\frac{lZ}{J}\langle\nabla \P, \nabla \ln u\rangle
	+\left(q+\frac{q^2}{2} Z\right)\langle\nabla \P, \nabla \ln u\rangle\right\}.
\end{eqnarray}

Notice that in current case, $S_1=\frac{2}{n}$ and $S_6=\frac{q^3}{4n}$\footnote{Notice that if $q=0$, then $S_4,S_5,S_6$ vanish and $S_3=\frac{2}{n}$ and the following arguments are also  valid.}. We do not give the formula of $S_2,S_3,S_4,S_5$ and just remenber that they are only dependent on $n,p,q$. Assume that
\begin{equation}\label{cd}
	\frac{L}{H}(\x)\ge \max\left(\frac{64n}{q^3}|S_5|,\left(\frac{64n}{q^3}|S_4|\right)^{\frac{1}{2}},\left(\frac{64n}{q^3}|S_3|\right)^{\frac{1}{3}},\left(\frac{64n}{q^3}|S_2|\right)^{\frac{1}{4}}\right),
\end{equation}
otherwise, we have
$$
\sup\limits_{B(x_0,R)}H+L\le (H+L)(\x)\le C(n,p,q)H(\x)\le C(n,p,q)\sup\limits_{B(x_0,2R)}H 
$$
and our proof is complete. For simplifying the calculations, in the following calculation, we use symbol $A\gtrsim B$ (or $A\lesssim B$) to represent $A\ge C(n,p,q)B$ (or $A\le C(n,p,q)B$) with some positive real number $C(n,p,q)$.

Under the condition \eqref{cd}, at the point $\x$, we have
\begin{eqnarray}\label{614}
	&&S_1H^5+S_2H^4L+S_3H^3L^2+S_4H^2L^3+S_5HL^4+S_6L^5\nonumber\\
	&\ge&\frac{2}{n}H^5+|S_2|H^4L+|S_3|H^3L^2+|S_4|H^2L^3+|S_5|HL^4+\frac{q^3}{8n}L^5.\nonumber
\end{eqnarray}
So, there exists $\e=\e(n,p,q)>0$ such that
\begin{eqnarray}\label{615}
	\frac{J\P\left(S_1H^5+S_2H^4L+S_3H^3L^2+S_4H^2L^3+S_5HL^4+S_6L^5\right)}{\left(H+\frac{q}{2}L\right)^{3}}\ge\e J\P(H+L)^2.
\end{eqnarray}
By Cauchy-Schwarz and Young inequalities, we derive
\begin{equation}\label{616}
	\left(2H+q L\right) \frac{1+  Z}{J} \langle\nabla \P, \nabla \ln u\rangle+\frac{\e}{10}J\P H(H+L)\gtrsim-\frac{|\nabla\P|^2}{\P}(H+L),
\end{equation}
\begin{equation}
	\frac{q}{2}\left(\frac{q}{2}-1\right)L\frac{(1+Z)^2}{J^2}\frac{|\nabla\P|^2}{\P}\gtrsim -\frac{|\nabla\P|^2}{\P}L,
\end{equation}
\begin{equation}
	-lq\frac{1+Z}{J}L\langle\nabla \P, \nabla \ln u\rangle+\frac{\e}{10}J\P HL\gtrsim-\frac{|\nabla\P|^2}{\P}L,
\end{equation}
\begin{equation}
	q\left(\frac{q}{2}-1\right)\frac{lZ(1+Z)}{J^2}L\langle\nabla \P, \nabla \ln u\rangle+\frac{\e}{10}J\P HL\gtrsim-\frac{|\nabla\P|^2}{\P}L,
\end{equation}
\begin{equation}\label{620}
	\left(q+\frac{q^2}{2} Z\right)\frac{1+Z}{J}L\langle\nabla \P, \nabla \ln u\rangle+\frac{\e}{10}J\P L(H+L)\gtrsim-\frac{|\nabla\P|^2}{\P}L.
\end{equation}
Substituting \eqref{615}, \eqref{616},$\cdots$,\eqref{620} into \eqref{612} yields
\begin{equation}\label{621}
\frac{\e}{2}J\P(H+L)^2\lesssim  \frac{|\nabla\P|^2}{\P}(H+L)+K\P(H+L)+\left(-\Delta\P+2\frac{|\nabla\P|^2}{\P}\right)(H+L).
\end{equation}
	Then, using the property of cut-off function in Lemma \ref{cut}, we obtain
\begin{equation}
	\max\limits_{B(x_0,2R)}A=A(\x)\lesssim \frac{1}{R^2}+K,\nonumber
\end{equation}
which implies
\begin{equation}
	\sup\limits_{B(x_0,R)}H+L\le C(n,p,q)\left(\frac{1}{R^2}+K\right).
\end{equation}
Therefore, whether the condition \eqref{cd} holds, the estimate \eqref{62} is valid and we complete the proof.

\end{proof}

\section{Local estimates from Liouville theorems}\label{S7}
	
	In this section, we shall prove Theorem \ref{LFG}. Due to Theorem \ref{t10} and \cite[Theorem 1.6]{LU1}, in the whole section, we always assume $p+q>1$ and $q>0$. As explained in Subsection \ref{ss3}, based on known Liouville theorems for \eqref{pqle}, we will derive these estimates by blow up arguments. In fact, our final contradictions can be attributed to a weaker Liouville theorem. So, first we define the weak Liouville  domain of \eqref{pqle} on $\mathbb{R}^n$ as follows.
	
	\begin{eqnarray}\label{WLD}
		\mathscr{D}_{WL}(n)=\left\{\begin{array}{l|l}
			(p,q) \in \mathbb{R}^2_{+} & \begin{array}{l}
				\text{any positive classical solution $u$ of \eqref{pqle} on } \\
				\text{ $\mathbb{R}^n$  with indices $p,q$ satisfying $|\nabla\ln u|\le 1$}\\
				\text{and $|\nabla u|^qu^{p-1}\le 1$ must be constant.}
			\end{array}
		\end{array}\right\}.
	\end{eqnarray}
	It is obvious that $\mathscr{D}_{L}(n)\subset \mathscr{D}_{WL}(n)$, where $\mathscr{D}_{L}(n)$ is defined by \eqref{LD}.
	
	\vspace{2mm}
	Now, we state the main result in this section, which  strengthen Theorem \ref{LFG}.
	
		\begin{thm}\label{SLFG}
		Let $\O\subsetneq \mathbb{R}^n$ be a domain and $u$ be a positive solution  of \eqref{pqle} with indices $p,q$ on $\O$. If $(p,q)\in \mathscr{D}_{WL}(n)$, where $\mathscr{D}_{WL}(n)$ is defined by \eqref{WLD},  then 
		\begin{eqnarray}
			\left(\frac{|\nabla u|^2}{u^2}+|\nabla u|^qu^{p-1}\right)(x)\le \frac{C(n,p,q)}{{\rm dist}^2(x,\partial\O)} \qquad \text{on\,\, $\O$}.
		\end{eqnarray}
	\end{thm}

	The direct corollary of Theorem \ref{SLFG} is the equivalence of Liouville theorem and weak Liouville theorem.
	
	\begin{cor}
		Let $n\in \mathbb{N}$, then $\mathscr{D}_{L}(n)=\mathscr{D}_{WL}(n)$.
	\end{cor}
	
	For proving Theorem \ref{SLFG}, we first recall the following important doubling lemma from Poláčik-Quittner-Souplet \cite{PQS}.%, which is fundamental for the proof of Theorem \ref{LFG}.
	\begin{lem}\cite[Lemma 5.1]{PQS}\label{DL}
		Let $\O\subset\mathbb{R}^n$ be a nonempty domain. Let $F: \O \rightarrow(0, \infty)$ be bounded on compact subsets of $\O$. For  a real number $k>0$, if $y \in \O$ is such that
		$$
		F(y) \operatorname{dist}(y, \partial\O)>2 k,
		$$
		then there exists $x \in \O$ such that
		$$
		F(x) \operatorname{dist}(x, \partial\O)>2 k, \quad F(x) \geq F(y),
		$$
		and
		$$
		F(z) \leq 2 F(x) \quad \text { for all } z \in\overline{B\left(x, \frac{k}{F(x)}\right)}\subset\O.
		$$
	\end{lem}

	Now, we give the proof of Theorem \ref{SLFG}.

	\begin{proof}[Proof of Theorem \ref{SLFG}]
		To make the proof clearer, we divide it into three steps.
		
		Step 1: setting of blow-up arguments.

		If the estimate \eqref{lgefromL} fails, then there exist sequences $\Omega_k, u_k, y_k \in \Omega_k$ such that $u_k$ solves equation \eqref{pqle} on $\Omega_k$ and the functions
		$$
		F_k:=\left|\nabla\ln u_k\right|+|\nabla u_k|^{\frac{q}{2}}u_k^{\frac{p-1}{2}}, \quad k=1,2, \ldots,
		$$
		satisfy
		$$
		F_k\left(y_k\right)>2 k \operatorname{dist}^{-1}\left(y_k, \partial \Omega_k\right) .
		$$
		By Lemma \ref{DL}, it follows that there exists $x_k \in \Omega_k$ such that
		$$
		F_k\left(x_k\right) \geq F_k\left(y_k\right), \quad F_k\left(x_k\right)>2 k \operatorname{dist}^{-1}\left(x_k, \partial \Omega_k\right),
		$$
		and
		$$
		F_k(z) \leq 2 F_k\left(x_k\right), \quad \operatorname{dist}(x_k,z) \leq \frac{k}{F_k\left(x_k\right)} .
		$$
		
		Now, define the rescaling of $u_k$ as 
		$$
		v_k(y):=\lambda_k^{\frac{2-q}{p+q-1}} u_k\left(x_k+\lambda_k y\right), \quad|y| \leq k, \text { with } \lambda_k=\frac{1}{F_k\left(x_k\right)} .
		$$
		The function $v_k$ solves
		$$
		\Delta v_k(y)+v_k^p(y)|\nabla v_k|^q(y)=0, \quad|y| \leq k,
		$$
		and
		\begin{eqnarray}\label{kk}
			\begin{cases}
				\left(\left|\nabla \ln v_k\right|+|\nabla v_k|^{\frac{q}{2}}v_k^{\frac{p-1}{2}}\right)(0)=\lambda_k F_k\left(x_k\right)=1, \\
				\\
				\left(\left|\nabla \ln v_k\right|+|\nabla v_k|^{\frac{q}{2}}v_k^{\frac{p-1}{2}}\right)(y) \leq 2, \quad\quad|y| \leq k .
			\end{cases}
		\end{eqnarray}
		Without loss of generality, we assume that 
		$$
		\alpha:=\lim\limits_{k\to \infty}\left|\nabla \ln v_k\right|(0),\qquad \beta:=\lim\limits_{k\to \infty}|\nabla v_k|^{\frac{q}{2}}v_k^{\frac{p-1}{2}}(0).
		$$
		Now we have reached the core of the proof, which is the following claim:
		$$ \alpha,\beta\in (0,1). $$

		Step 2: proof of the claim.
		
		If it does not hold, there will be two situations.
		\begin{itemize}
		\item Case 1: $ \alpha=1,\beta=0$;
		\item Case 2: $ \alpha=0,\beta=1. $
		\end{itemize}
		 Then we discuss these two cases in detail.
		If Case 1 holds, then
		$$
		\lim\limits_{k\to\infty}v_k(0)=0.
		$$
		For any fixed $R>0$, when $k>2R$ and $x\in B(0,2R)$, we have
		\begin{equation}
		\ln\left(\frac{v_k(x)}{v_k(0)}\right)=\ln\left(\frac{v_k(\gamma(1))}{v_k(\gamma(0))}\right)=\int_{0}^{1}\nabla\ln v_k(\gamma(t))\cdot \dot{\gamma}(t)dt\le 2|x|\le 4R,\nonumber
		\end{equation}
		where $\gamma$ is the segment connecting points $0$ and $x$ with constant speed $|x|$. Therefore, 
		\begin{equation}
		\limsup\limits_{k\to \infty}	\Vert v_k\Vert_{L^{\infty}(B(0,2R))}\le e^{4R}\lim\limits_{k\to \infty}v_k(0)=0,\nonumber
		\end{equation}
		and using \eqref{kk} derives
		\begin{eqnarray}\label{71}
				&&\limsup\limits_{k\to \infty}	\Vert |\nabla v_k|^{\frac{q}{2}}v_k^{\frac{p-1}{2}} \Vert_{L^{\infty}(B(0,2R))}\nonumber\\
				&\le& 	\limsup\limits_{k\to \infty}	\Vert |\nabla \ln v_k|^{\frac{q}{2}} \Vert_{L^{\infty}(B(0,2R))} \Vert v_k^{\frac{p+q-1}{2}} \Vert_{L^{\infty}(B(0,2R))}\nonumber\\
				&\le& 2^{\frac{q}{2}}\limsup\limits_{k\to \infty}	\Vert v_k^{\frac{p+q-1}{2}}\Vert_{L^{\infty}(B(0,2R))}=0.
				\end{eqnarray}
		Then using Lemma \ref{control1} and \eqref{71}, we have
		\begin{eqnarray}
			1&=&\lim\limits_{k\to \infty}\left|\nabla \ln v_k\right|(0)\nonumber\\
			&\le& 	\limsup\limits_{k\to \infty}	\Vert |\nabla \ln v_k| \Vert_{L^{\infty}(B(0,R))}\nonumber\\
			&\le& 	\limsup\limits_{k\to \infty} C(n,p,q)	\left(\Vert |\nabla v_k|^{\frac{q}{2}}v_k^{\frac{p-1}{2}} \Vert_{L^{\infty}(B(0,2R))}+\frac{1}{R}\right)\nonumber\\
			&=&\frac{C(n,p,q)}{R}.\nonumber
		\end{eqnarray}
		By the arbitrariness of $R$, this yields a contradiction and so we exclude the Case 1.

			If Case 2 holds, then
		$$
		\lim\limits_{k\to\infty}v_k(0)=\infty.
		$$
		For any fixed $R>0$, when $k>2R$ and $x\in B(0,2R)$, we have
		\begin{equation}
			\ln\left(\frac{v_k(x)}{v_k(0)}\right)=\ln\left(\frac{v_k(\gamma(1))}{v_k(\gamma(0))}\right)=\int_{0}^{1}\nabla\ln v_k(\gamma(t))\cdot \dot{\gamma}(t)dt\ge -2|x|\ge -4R,\nonumber
		\end{equation}
		where $\gamma$ is the segment connecting points $0$ and $x$ with constant speed $|x|$. Therefore, 
		\begin{equation}
			\liminf\limits_{k\to \infty} \inf\limits_{B(x_0,2R)} v_k\ge e^{-4R}\lim\limits_{k\to \infty}v_k(0)=\infty,\nonumber
		\end{equation}
		and using \eqref{kk} yields
		\begin{eqnarray}\label{72}
			&&\limsup\limits_{k\to \infty} \Vert |\nabla \ln v_k|^{\frac{q}{2}} \Vert_{L^{\infty}(B(0,2R))}\nonumber\\	&\le& 	\limsup\limits_{k\to \infty} \frac{\Vert |\nabla v_k|^{\frac{q}{2}}v_k^{\frac{p-1}{2}} \Vert_{L^{\infty}(B(0,2R))}}{\inf\limits_{B(x_0,2R)} v_k^{\frac{p+q-1}{2}}} 	 \nonumber\\
			&\le& \frac{2}{\liminf\limits_{k\to \infty}\inf\limits_{B(x_0,2R)} v_k^{\frac{p+q-1}{2}}} 	=0.
		\end{eqnarray}
			Then using Lemma \ref{control2} and \eqref{72}, we have
		\begin{eqnarray}
			1&=&\lim\limits_{k\to \infty} |\nabla v_k|^{\frac{q}{2}}v_k^{\frac{p-1}{2}}(0)\nonumber\\
			&\le& 	\limsup\limits_{k\to \infty}	\Vert |\nabla v_k|^{\frac{q}{2}}v_k^{\frac{p-1}{2}} \Vert_{L^{\infty}(B(0,R))}\nonumber\\
			&\le& 	\limsup\limits_{k\to \infty} C(n,p,q)	\left(\Vert |\nabla\ln v_k| \Vert_{L^{\infty}(B(0,2R))}+\frac{1}{R}\right)\nonumber\\
			&=&\frac{C(n,p,q)}{R}.\nonumber
		\end{eqnarray}
		By the arbitrariness of $R$, this yields a contradiction and so we exclude the Case 2.

		Step 3: end of the proof.
		
		 From the definitions of $\alpha,\b$, we have
		\begin{equation}
			\lim\limits_{k\to \infty}v_k(0)=\left(\frac{\b}{\alpha^{\frac{q}{2}}}\right)^{\frac{2}{p+q-1}}>0
		\end{equation}
		and 
		\begin{equation}
			\lim\limits_{k\to \infty}|\nabla v_k|(0)=\alpha\left(\frac{\b}{\alpha^{\frac{q}{2}}}\right)^{\frac{2}{p+q-1}}>0.
		\end{equation}
		For any fixed $R>0$, using the same arguments as Step 2, we see that when $k>R$ and $x\in B(0,R)$, 
		\begin{equation}
			\Vert v_k\Vert_{L^{\infty}(B(0,R))}\le e^{2R}v_k(0).\nonumber
		\end{equation}
		That is $\Vert v_k\Vert_{L^{\infty}(B(0,R))}$ is uniformly bounded and so $\Vert |\nabla v_k|^{q}v_k^{p} \Vert_{L^{\infty}(B(0,R))}$ is also.
		By elliptic $W^{2,d}$ estimates (for any $d>1$) and Sobolev imbedding, we deduce that some subsequence of $v_k$ (still recorded as $v_k$) converges in $C^1\left(B(0,R)\right)$ to a nonnegative classical solution $u_R$ of \eqref{pqle}.
		Notice that when $k>R$,
		\begin{equation}
			\inf\limits_{B(0,R)}v_k\ge e^{-2R}v_k(0).\nonumber
		\end{equation}
		Hence $u_R$ is positive. By diagonal arguments, there exists a subsequence of $v_k$ that $C^1_{loc}(\mathbb{R}^n)$ converges to a positive function $u$. Finally, $u$ is a global solution of \eqref{pqle} satisfying
		\begin{equation}
			|\nabla u|(0)=\alpha\left(\frac{\b}{\alpha^{\frac{q}{2}}}\right)^{\frac{2}{p+q-1}}>0\nonumber
		\end{equation}
	and 
	$$
	\sup\limits_{\mathbb{R}^n}\left(|\nabla\ln u|+|\nabla u|^{\frac{q}{2}}u^{\frac{p-1}{2}}\right)\le 2,
	$$
		which contradicts with the weak  Liouville theorem assumption.

	\end{proof}
	
%\section*{Appendices}
\appendix
\setcounter{equation}{0}

\setcounter{subsection}{0}

\setcounter{thm}{0}

\renewcommand{\theequation}{A.\arabic{equation}}

\renewcommand{\thesubsection}{A.\arabic{subsection}}

\renewcommand{\thethm}{A.\arabic{thm}}

\section*{Appendix A: Algebraic computations of coefficient functions}\label{SA}
We collect the coefficient functions which appear in the dominant terms of elliptic equations of auxiliary functions related to solutions. These results can be obtained through pure algebraic calculations and readers can check them by Wolfram Mathematica.
\begin{lem}\label{a1}
	The coefficients $\S_1,\cdots,\S_6$  in Lemma \ref{kl} is given by
	\begin{eqnarray}
		\S_1&=&\left(1+\frac{q\g}{2}\right)^3A(\b)-k(k+1)\left(1+\frac{q\g}{2}\right)^2\nonumber\\
		&&+\frac{q\g k^2}{2}\left(1+\frac{q\g}{2}\right)-2k(\b-1)\left(1+\frac{q\g}{2}\right)^2,\nonumber
	\end{eqnarray}
	
	\begin{eqnarray}
		\S_2&=&\left(1+\frac{q\g}{2}\right)^2\frac{dq}{2}(1+\g)(2A(\b)+A(\delta))\nonumber\\
		&&+\left(B(\b,\sigma)+(1-q+2\sigma)\b\right)\left(1+\frac{q\g}{2}\right)^3\nonumber\\
		&&+2(k-1)\b\left(1+\frac{q\g}{2}\right)^2\frac{dq}{2}(1+\g)-dq\b\g(k+l)\left(1+\frac{q\g}{2}\right)^2\nonumber\\
		&&+2\b\left(1+\frac{q\g}{2}\right)^3\left((k+l)d-\frac{dq}{2}(1+\g)\b\right)\nonumber\\
		&&-2\b(\b-1)\left(1+\frac{q\g}{2}\right)^2\frac{dq}{2}(1+\gamma)\nonumber\\
		&&+d(k+l)\left(1+\frac{q\g}{2}\right)^2\left(k+l+1-2\b\right)\nonumber\\
		&&+\left(1+\frac{q\g}{2}\right)^2\frac{dq}{2}(1+\g)\left(\frac{q(1+\g)}{2}-1\right)\b^2\nonumber\\
		&&+dq(1+\g)(k+l)\left(1+\frac{q\g}{2}\right)^2\b+dq(1+\g)\left(1+\frac{q\g}{2}\right)^2(\delta-1)\b,\nonumber
	\end{eqnarray}
	
	\begin{eqnarray}
		\S_3&=&\left(1+\frac{q\g}{2}\right)\left(\frac{dq}{2}(1+\g)\right)^2(A(\b)+2A(\delta))\nonumber\\
		&&+\left(1+\frac{q\g}{2}\right)^2\frac{dq}{2}(1+\g)\left(2B(\b,\sigma)+B(\delta,\tau)\right)+A(\sigma)\left(1+\frac{q\g}{2}\right)^3\nonumber\\
		&&+k\left((k-1)\left(\frac{dq}{2}(1+\g)\right)^2-2\left(1+\frac{q\g}{2}\right)\left(\frac{dq}{2}(1+\g)\right)\right)\nonumber\\
		&&+\frac{q\g}{2}\left(1+\frac{q\g}{2}\right)d^2(k+l)^2-2k(k+l)d\left(1+\frac{q\g}{2}\right)\left(\frac{dq}{2}(1+\g)\right)\nonumber\\
		&&-\left(1+\frac{q\g}{2}\right)\left[2(\b-1)d(k+l)\frac{dq}{2}(1+\g)\right.\nonumber\\
		&&\left.+(2\sigma-q)\left((k+l)d\left(1+\frac{q\g}{2}\right)+k\frac{dq}{2}(1+\g)\right)\right]\nonumber\\
		&&+d(k+l)\left(-\left(1+\frac{q\g}{2}\right)^2+2(k+l-1)\left(1+\frac{q\g}{2}\right)\frac{dq}{2}(1+\g)\right)\nonumber\\
		&&+k(k+l)d^2q(1+\g)\left(\frac{q}{2}(1+\g)-1\right)\nonumber\\
		&&-(k+l)dq(1+\g)\left(\frac{kdq}{2}(1+\g)+(k+l)d\left(1+\frac{q\g}{2}\right)\right)\nonumber\\
		&&-\frac{dq}{2}(1+\g)\left[(2\tau-q)k\left(1+\frac{q\g}{2}\right)\right.\nonumber\\
		&&\left.+2(\delta-1)\left(\frac{kdq}{2}(1+\g)+(k+l)d\left(1+\frac{q\g}{2}\right)\right)\right],\nonumber
	\end{eqnarray}
	
	\begin{eqnarray}
		\S_4&=&\left(\frac{dq}{2}(1+\g)\right)^3A(\delta)+\left(1+\frac{q\g}{2}\right)\left(\frac{dq}{2}(1+\g)\right)^2\left(B(\b,\sigma)+2B(\delta,\tau)\right)\nonumber\\
		&&+\left(1+\frac{q\g}{2}\right)^2\frac{dq}{2}(1+\g)\left(2A(\sigma)+A(\tau)\right)-k\left(\frac{dq}{2}(1+\g)\right)^2\nonumber\\
		&&-(2\sigma-q)(k+l)d\left(1+\frac{q\g}{2}\right)\frac{dq}{2}(1+\g)\nonumber\\
		&&+d(k+l)\left((k+l-1)\left(\frac{dq}{2}(1+\g)\right)^2-2\left(1+\frac{q\g}{2}\right)\frac{dq}{2}(1+\g)\right)\nonumber\\
		&&+\frac{dq}{2}(1+\g)\left(\frac{q\g}{2}(1+\g)-1\right)(k+l)^2d^2-2d(k+l)^2\left(\frac{dq}{2}(1+\g)\right)^2\nonumber\\
		&&-\frac{dq}{2}(1+\g)\left[2(\delta-1)(k+l)d\frac{dq}{2}(1+\g)\right.\nonumber\\
		&&\left.+(2\tau-q)\left(\frac{kdq}{2}(1+\g)+(k+l)d\left(1+\frac{q\g}{2}\right)\right)\right],\nonumber
	\end{eqnarray}
	\begin{eqnarray}
		\S_5&=&B(\delta,\tau)\left(\frac{dq}{2}(1+\g)\right)^3+\left(A(\delta)+2A(\tau)\right)\left(1+\frac{q\g}{2}\right)\left(\frac{dq}{2}(1+\g)\right)^2\nonumber\\
		&&+(q-1-2\tau)\left(k+l\right)d\left(\frac{dq}{2}(1+\g)\right)^2,\nonumber
	\end{eqnarray}
	\begin{equation}
		\S_6=\left(\frac{dq}{2}(1+\g)\right)^3A(\tau),\nonumber
	\end{equation}
	where the functions $A,B$ are expressed as  \eqref{A} and \eqref{B} in Lemma \ref{l2} respectively.
	
\end{lem}

\begin{lem}\label{a2}
	The coefficients $I_j$ ($1\le j\le 10 $) in Lemma \ref{eg} is given by 
	\begin{equation}
	I_1=\left(1-\frac{q}{2}\right)S_1,\nonumber
	\end{equation}
	\begin{eqnarray}
		I_2&=&\left(1-\frac{q}{2}\right) \rho  \varepsilon  A(\delta)+\frac{k^2 q \rho  \varepsilon }{2 \left(1-\frac{q}{2}\right)}-2 (\delta -1) k \rho  \varepsilon\nonumber\\
		&& +2 (\beta -1) \rho  \varepsilon  (k+l)+\frac{k l q \rho  \varepsilon }{1-\frac{q}{2}}+\frac{k q \rho  \varepsilon  (k+l)}{1-\frac{q}{2}}\nonumber\\
		&&-\frac{2 k \rho  \varepsilon  (k+l)}{1-\frac{q}{2}}+2 k \rho  \varepsilon  (k+l)+(1-l) l \rho  \varepsilon\nonumber\\
		&& +2 \left(1-\frac{q}{2}\right) S_1 (\rho  \varepsilon +\varepsilon )+2 \left(1-\frac{q}{2}\right) S_1 \varepsilon +\left(1-\frac{q}{2}\right) S_2,\nonumber
	\end{eqnarray}

	\begin{eqnarray}
		I_3&=&\left(1-\frac{q}{2}\right) S_3+6 \left(1-\frac{q}{2}\right) S_1 \varepsilon ^2+4 \left(1-\frac{q}{2}\right) S_2 \varepsilon \nonumber\\
		&&+	\rho ^2 \varepsilon ^2 \left(2 \left(1-\frac{q}{2}\right) A(\delta)+\frac{k^2 q}{1-\frac{q}{2}}\right.\nonumber\\
		&&-4 (\delta -1) k+2 (\beta -1) (k+l)\nonumber\\
		&&+2 (\delta -1) (k+l)+\frac{2 k l q}{1-\frac{q}{2}}\nonumber\\
		&&-\frac{2 k (k+l)}{1-\frac{q}{2}}+\frac{2 (k+l)^2}{1-\frac{q}{2}}\nonumber\\
		&&-\frac{l q (k+l)}{1-\frac{q}{2}}-\frac{q (k+l)^2}{2 \left(1-\frac{q}{2}\right)}\nonumber\\
		&&+2 k (k+l)-(k+l)^2\nonumber\\
		&&\left.-2 (l-1) l+\left(1-\frac{q}{2}\right) S_1\right)\nonumber\\
		&&+\rho \varepsilon  \left(\left(1-\frac{q}{2}\right) B(\delta ,\tau )+d \left(1-\frac{q}{2}\right) A(\delta)\right.\nonumber\\
		&&+\frac{d k^2 q}{2 \left(1-\frac{q}{2}\right)}-2 d (\delta -1) k\nonumber\\
		&&+4 (\beta -1) d (k+l)-2 d (\delta -1) (k+l)\nonumber\\
		&&+\frac{d k l q}{1-\frac{q}{2}}+\frac{3 d k q (k+l)}{1-\frac{q}{2}}\nonumber\\
		&&-\frac{4 d k (k+l)}{1-\frac{q}{2}}+\frac{d q (k+l)^2}{1-\frac{q}{2}}\nonumber\\
		&&+\frac{d l q (k+l)}{1-\frac{q}{2}}-\frac{2 d (k+l)^2}{1-\frac{q}{2}}\nonumber\\
		&&+4 d k (k+l)-d (l-1) l\nonumber\\
		&&+2 d \left(1-\frac{q}{2}\right) S_1+(k+l) (2 \sigma -q)\nonumber\\
		&&\left.-k (2 \tau -q)+l+2 \left(1-\frac{q}{2}\right) S_2\right)\nonumber\\
		&&+\rho \varepsilon ^2 \left(3 \left(1-\frac{q}{2}\right) A(\delta)+\frac{3 k^2 q}{2 \left(1-\frac{q}{2}\right)}\right.\nonumber\\
		&&-6 (\delta -1) k+6 (\beta -1) (k+l)\nonumber\\
		&&+\frac{3 k l q}{1-\frac{q}{2}}+\frac{3 k q (k+l)}{1-\frac{q}{2}}-\frac{4 k (k+l)}{1-\frac{q}{2}}\nonumber\\
		&&+2 (k+l) \left(l-\frac{k q}{2 \left(1-\frac{q}{2}\right)}\right)+6 k (k+l)\nonumber\\
		&&\left.-(k+l)^2+3(1-l) l+6 \left(1-\frac{q}{2}\right) S_1\right),\nonumber
	\end{eqnarray}

	\begin{eqnarray}
		I_4&=&4 \left(1-\frac{q}{2}\right) S_1 \varepsilon ^3+6 \left(1-\frac{q}{2}\right) S_2 \varepsilon ^2+4 \left(1-\frac{q}{2}\right) S_3 \varepsilon\nonumber\\
		&& +\rho \varepsilon  \left(d \left(1-\frac{q}{2}\right) B(\delta ,\tau )+\left(1-\frac{q}{2}\right) A(\tau )\right.\nonumber\\
		&&+2 (\beta -1) d^2 (k+l)-2 (\delta -1) d^2 (k+l)+\frac{5 d^2 q (k+l)^2}{2 \left(1-\frac{q}{2}\right)}\nonumber\\
		&&+\frac{2 d^2 k q (k+l)}{1-\frac{q}{2}}+\frac{d^2 l q (k+l)}{1-\frac{q}{2}}\nonumber\\
		&&-\frac{4 d^2 (k+l)^2}{1-\frac{q}{2}}-\frac{2 d^2 k (k+l)}{1-\frac{q}{2}}\nonumber\\
		&&+2 d^2 k (k+l)+2 d (k+l) (2 \sigma -q)-d (k+l) (2 \tau -q)\nonumber\\
		&&\left.-d k (2 \tau -q)+d l+2 d \left(1-\frac{q}{2}\right) S_2+2 \left(1-\frac{q}{2}\right) S_3\right)\nonumber\\
		&&+\rho\varepsilon ^2 \left(3 \left(1-\frac{q}{2}\right) B(\delta ,\tau)+3 d \left(1-\frac{q}{2}\right) A(\delta)\right.\nonumber\\
		&&+\frac{3 d k^2 q}{2 \left(1-\frac{q}{2}\right)}-6 d (\delta -1) k\nonumber\\
		&&+12 (\beta -1) d (k+l)-6 d (\delta -1) (k+l)\nonumber\\
		&&+\frac{3 d k l q}{1-\frac{q}{2}}+\frac{9 d k q (k+l)}{1-\frac{q}{2}}\nonumber\\
		&&-\frac{8 d k (k+l)}{1-\frac{q}{2}}+4 d (k+l) \left(l-\frac{k q}{2 \left(1-\frac{q}{2}\right)}\right)\nonumber\\
		&&+\frac{2 d q (k+l)^2}{1-\frac{q}{2}}+\frac{3 d l q (k+l)}{1-\frac{q}{2}}\nonumber\\
		&&-\frac{4 d (k+l)^2}{1-\frac{q}{2}}+12 d k (k+l)-3 d (k+l)^2\nonumber\\
		&&-3 d (l-1) l+6 d \left(1-\frac{q}{2}\right) S_1\nonumber\\
		&&+3 (k+l) (2 \sigma -q)-3 k (2 \tau -q)\nonumber\\
		&&\left.+3 l+6 \left(1-\frac{q}{2}\right) S_2\right)\nonumber\\
		&&+\rho \varepsilon ^3 \left(3 \left(1-\frac{q}{2}\right) A(\delta)+\frac{3 k^2 q}{2 \left(1-\frac{q}{2}\right)}\right.\nonumber\\
		&&-6 (\delta -1) k+6 (\beta -1) (k+l)\nonumber\\
		&&+\frac{3 k l q}{1-\frac{q}{2}}+\frac{3 k q (k+l)}{1-\frac{q}{2}}\nonumber\\
		&&-\frac{2 k (k+l)}{1-\frac{q}{2}}+4 (k+l) \left(l-\frac{k q}{2 \left(1-\frac{q}{2}\right)}\right)\nonumber\\
		&&+6 k (k+l)-2 (k+l)^2+(1-l) l\nonumber\\
		&&\left.-2 (l-1) l+6 \left(1-\frac{q}{2}\right) S_1\right)\nonumber\\
		&&+\rho^2\varepsilon ^2 \left(2 \left(1-\frac{q}{2}\right) B(\delta ,\tau)+4 d \left(1-\frac{q}{2}\right) A(\delta)\right.\nonumber\\
		&&+\frac{2 d k^2 q}{1-\frac{q}{2}}-8 d (\delta -1) k\nonumber\\
		&&+6 (\beta -1) d (k+l)+2 d (\delta -1) (k+l)\nonumber\\
		&&+\frac{4 d k l q}{1-\frac{q}{2}}+\frac{2 d k q (k+l)}{1-\frac{q}{2}}\nonumber\\
		&&-\frac{6 d k (k+l)}{1-\frac{q}{2}}+\frac{6 d (k+l)^2}{1-\frac{q}{2}}\nonumber\\
		&&-\frac{d l q (k+l)}{1-\frac{q}{2}}-\frac{3 d q (k+l)^2}{2 \left(1-\frac{q}{2}\right)}\nonumber\\
		&&+6 d k (k+l)-3 d (k+l)^2\nonumber\\
		&&-4 d (l-1) l+2 d \left(1-\frac{q}{2}\right) S_1\nonumber\\
		&&+(k+l) (2 \sigma -q)+(k+l) (2 \tau -q)\nonumber\\
		&&\left.-2 k (2 \tau -q)+2 l+\left(1-\frac{q}{2}\right) S_2\right)\nonumber\\
		&&+\rho^2\varepsilon ^3 \left(4 \left(1-\frac{q}{2}\right) A(\delta)+\frac{2 k^2 q}{1-\frac{q}{2}}\right.\nonumber\\
		&&	-8 (\delta -1) k+4 (\beta -1) (k+l)\nonumber\\
		&&+4 (\delta -1) (k+l)+\frac{4 k l q}{1-\frac{q}{2}}\nonumber\\
		&&-\frac{2 k (k+l)}{1-\frac{q}{2}}+\frac{2 (k+l)^2}{1-\frac{q}{2}}\nonumber\\
		&&+2 (k+l) \left(l-\frac{k q}{2 \left(1-\frac{q}{2}\right)}\right)-\frac{2 l q (k+l)}{1-\frac{q}{2}}\nonumber\\
		&&+4 k (k+l)-2 (k+l)^2+2 (1-l) l\nonumber\\
		&&\left.-2 (l-1) l+2 \left(1-\frac{q}{2}\right) S_1\right)\nonumber\\&&+\rho ^3 \varepsilon ^3 \left(\left(1-\frac{q}{2}\right) A(\delta)+\frac{k^2 q}{2 \left(1-\frac{q}{2}\right)}\right.\nonumber\\&&-2 (\delta -1) k+2 (\delta -1) (k+l)\nonumber\\&&+\frac{k l q}{1-\frac{q}{2}}-\frac{k q (k+l)}{1-\frac{q}{2}}\nonumber\\&&\left.+\frac{q (k+l)^2}{2 \left(1-\frac{q}{2}\right)}-\frac{l q (k+l)}{1-\frac{q}{2}}+(1-l) l\right),\nonumber
	\end{eqnarray}
	
\begin{eqnarray}
	I_5&=&\left(1-\frac{q}{2}\right) S_1 \varepsilon ^4+4 \left(1-\frac{q}{2}\right) S_2 \varepsilon ^3+6 \left(1-\frac{q}{2}\right) S_3 \varepsilon ^2\nonumber\\
	&&+\rho \varepsilon  \left(d \left(1-\frac{q}{2}\right) A(\tau )+\frac{3 d^3 q (k+l)^2}{2 \left(1-\frac{q}{2}\right)}\right.\nonumber\\
	&&-\frac{2 d^3 (k+l)^2}{1-\frac{q}{2}}+d^2 (k+l) (2 \sigma -q)\nonumber\\
	&&\left.-d^2 (k+l) (2 \tau -q)+2 d \left(1-\frac{q}{2}\right) S_3\right)\nonumber\\
	&&+\rho\varepsilon ^2 \left(3 d \left(1-\frac{q}{2}\right) B(\delta ,\tau)+3 \left(1-\frac{q}{2}\right) A(\tau)\right.\nonumber\\
	&&+6 (\beta -1) d^2 (k+l)-6 (\delta -1) d^2 (k+l)\nonumber\\
	&&+2 d^2 (k+l) \left(l-\frac{k q}{2 \left(1-\frac{q}{2}\right)}\right)+\frac{11 d^2 q (k+l)^2}{2 \left(1-\frac{q}{2}\right)}\nonumber\\
	&&+\frac{6 d^2 k q (k+l)}{1-\frac{q}{2}}+\frac{3 d^2 l q (k+l)}{1-\frac{q}{2}}\nonumber\\
	&&-\frac{8 d^2 (k+l)^2}{1-\frac{q}{2}}-\frac{4 d^2 k (k+l)}{1-\frac{q}{2}}\nonumber\\
	&&-3 d^2 (k+l)^2+6 d^2 k (k+l)\nonumber\\
	&&+6 d (k+l) (2 \sigma -q)-3 d (k+l) (2 \tau -q)\nonumber\\
	&&-3 d k (2 \tau -q)+3 d l\nonumber\\
	&&\left.+6 d \left(1-\frac{q}{2}\right) S_2+6 \left(1-\frac{q}{2}\right) S_3\right)\nonumber\\
	&&+\rho \varepsilon ^3 \left(3 \left(1-\frac{q}{2}\right) B(\delta ,\tau)+3 d \left(1-\frac{q}{2}\right) A(\delta)\right.\nonumber\\
	&&+\frac{3 d k^2 q}{2 \left(1-\frac{q}{2}\right)}-6 d (\delta -1) k\nonumber\\
	&&+12 (\beta -1) d (k+l)-6 d (\delta -1) (k+l)\nonumber\\
	&&+\frac{3 d k l q}{1-\frac{q}{2}}+\frac{9 d k q (k+l)}{1-\frac{q}{2}}\nonumber\\
	&&-\frac{4 d k (k+l)}{1-\frac{q}{2}}+8 d (k+l) \left(l-\frac{k q}{2 \left(1-\frac{q}{2}\right)}\right)\nonumber\\
	&&+\frac{d q (k+l)^2}{1-\frac{q}{2}}+\frac{3 d l q (k+l)}{1-\frac{q}{2}}\nonumber\\
	&&-\frac{2 d (k+l)^2}{1-\frac{q}{2}}+12 d k (k+l)\nonumber\\
	&&-6 d (k+l)^2-3 d (l-1) l\nonumber\\
	&&+6 d \left(1-\frac{q}{2}\right) S_1+3 (k+l) (2 \sigma -q)\nonumber\\
	&&\left.-3 k (2 \tau -q)+3 l+6 \left(1-\frac{q}{2}\right) S_2\right)\nonumber\\
	&&+\rho^2\varepsilon ^2 \left(4 d \left(1-\frac{q}{2}\right) B(\delta ,\tau)+2 d^2 \left(1-\frac{q}{2}\right) A(\delta)\right.\nonumber\\
	&&+2 \left(1-\frac{q}{2}\right) A(\tau)+\frac{d^2 k^2 q}{1-\frac{q}{2}}\nonumber\\
	&&-4 (\delta -1) d^2 k+6 (\beta -1) d^2 (k+l)\nonumber\\
	&&-2 (\delta -1) d^2 (k+l)+\frac{6 d^2 (k+l)^2}{1-\frac{q}{2}}\nonumber\\
	&&+\frac{2 d^2 k l q}{1-\frac{q}{2}}+\frac{4 d^2 k q (k+l)}{1-\frac{q}{2}}\nonumber\\
	&&+\frac{d^2 l q (k+l)}{1-\frac{q}{2}}-\frac{6 d^2 k (k+l)}{1-\frac{q}{2}}\nonumber\\
	&&-\frac{d^2 q (k+l)^2}{2 \left(1-\frac{q}{2}\right)}-3 d^2 (k+l)^2\nonumber\\
	&&+6 d^2 k (k+l)-2 d^2 (l-1) l\nonumber\\
	&&+d^2 \left(1-\frac{q}{2}\right) S_1+3 d (k+l) (2 \sigma -q)\nonumber\\
	&&+d (k+l) (2 \tau -q)-4 d k (2 \tau -q)+4 d l\nonumber\\
	&&\left.+2 d \left(1-\frac{q}{2}\right) S_2+\left(1-\frac{q}{2}\right) S_3\right)\nonumber\\
	&&+\rho^2\varepsilon ^3 \left(4 \left(1-\frac{q}{2}\right) B(\delta ,\tau )+8 d \left(1-\frac{q}{2}\right) A(\delta)\right.\nonumber\\
	&&+\frac{4 d k^2 q}{1-\frac{q}{2}}-16 d (\delta -1) k\nonumber\\
	&&+12 (\beta -1) d (k+l)+4 d (\delta -1) (k+l)\nonumber\\
	&&+\frac{8 d k l q}{1-\frac{q}{2}}+\frac{4 d k q (k+l)}{1-\frac{q}{2}}\nonumber\\
	&&-\frac{6 d k (k+l)}{1-\frac{q}{2}}+\frac{6 d (k+l)^2}{1-\frac{q}{2}}\nonumber\\
	&&+6 d (k+l) \left(l-\frac{k q}{2 \left(1-\frac{q}{2}\right)}\right)-\frac{2 d l q (k+l)}{1-\frac{q}{2}}\nonumber\\
	&&+12 d k (k+l)-6 d (k+l)^2\nonumber\\
	&&-8 d (l-1) l+4 d \left(1-\frac{q}{2}\right) S_1\nonumber\\
	&&+2 (k+l) (2 \sigma -q)+2 (k+l) (2 \tau -q)\nonumber\\
	&&\left.-4 k (2 \tau -q)+4 l+2 \left(1-\frac{q}{2}\right) S_2\right)\nonumber\\
	&&+\rho^2\varepsilon ^4 \left(2 \left(1-\frac{q}{2}\right) A(\delta)+\frac{k^2 q}{1-\frac{q}{2}}\right.\nonumber\\
	&&-4 (\delta -1) k+2 (\beta -1) (k+l)\nonumber\\
	&&+2 (\delta -1) (k+l)+\frac{2 k l q}{1-\frac{q}{2}}\nonumber\\
	&&+2 (k+l) \left(l-\frac{k q}{2 \left(1-\frac{q}{2}\right)}\right)+\frac{q (k+l)^2}{2 \left(1-\frac{q}{2}\right)}\nonumber\\
	&&-\frac{l q (k+l)}{1-\frac{q}{2}}+2 k (k+l)\nonumber\\
	&&\left.-(k+l)^2+2 (1-l) l+\left(1-\frac{q}{2}\right) S_1\right)\nonumber\\
	&&+\rho^3\varepsilon ^3 \left(\left(1-\frac{q}{2}\right) B(\delta ,\tau)+3 d \left(1-\frac{q}{2}\right) A(\delta )\right.\nonumber\\
	&&+\frac{3 d k^2 q}{2 \left(1-\frac{q}{2}\right)}-6 d (\delta -1) k\nonumber\\
	&&+6 d (\delta -1) (k+l)+\frac{3 d k l q}{1-\frac{q}{2}}\nonumber\\
	&&-\frac{3 d k q (k+l)}{1-\frac{q}{2}}+\frac{3 d q (k+l)^2}{2 \left(1-\frac{q}{2}\right)}\nonumber\\
	&&-\frac{3 d l q (k+l)}{1-\frac{q}{2}}-3 d (l-1) l\nonumber\\
	&&\left.+(k+l) (2 \tau -q)-k (2 \tau -q)+l\right)\nonumber\\
	&&+\rho^3\varepsilon ^4 \left(\left(1-\frac{q}{2}\right) A(\delta)+\frac{k^2 q}{2 \left(1-\frac{q}{2}\right)}\right.\nonumber\\
	&&-2 (\delta -1) k+2 (\delta -1) (k+l)+\frac{k l q}{1-\frac{q}{2}}\nonumber\\
	&&-\frac{k q (k+l)}{1-\frac{q}{2}}+\frac{q (k+l)^2}{2 \left(1-\frac{q}{2}\right)}\nonumber\\
	&&\left.-\frac{l q (k+l)}{1-\frac{q}{2}}+(1-l) l\right),\nonumber
\end{eqnarray}
	\begin{eqnarray}
		I_6&=&\left(1-\frac{q}{2}\right) S_2 \varepsilon ^4+4 \left(1-\frac{q}{2}\right) S_3 \varepsilon ^3\nonumber\\
		&&+\rho\varepsilon ^2 \left(3 d \left(1-\frac{q}{2}\right) A(\tau)+\frac{7 d^3 q (k+l)^2}{2 \left(1-\frac{q}{2}\right)}\right.\nonumber\\
		&&-\frac{4 d^3 (k+l)^2}{1-\frac{q}{2}}-d^3 (k+l)^2\nonumber\\
		&&+3 d^2 (k+l) (2 \sigma -q)-3 d^2 (k+l) (2 \tau -q)\nonumber\\
		&&\left.+6 d \left(1-\frac{q}{2}\right) S_3\right)\nonumber\\
		&&+\rho \varepsilon ^3 \left(3 d \left(1-\frac{q}{2}\right) B(\delta ,\tau )+3 \left(1-\frac{q}{2}\right) A(\tau)\right.\nonumber\\
		&&+6 (\beta -1) d^2 (k+l)-6 (\delta -1) d^2 (k+l)\nonumber\\
		&&+4 d^2 (k+l) \left(l-\frac{k q}{2 \left(1-\frac{q}{2}\right)}\right)+\frac{7 d^2 q (k+l)^2}{2 \left(1-\frac{q}{2}\right)}\nonumber\\
		&&+\frac{6 d^2 k q (k+l)}{1-\frac{q}{2}}+\frac{3 d^2 l q (k+l)}{1-\frac{q}{2}}\nonumber\\
		&&-\frac{4 d^2 (k+l)^2}{1-\frac{q}{2}}-\frac{2 d^2 k (k+l)}{1-\frac{q}{2}}\nonumber\\
		&&-6 d^2 (k+l)^2+6 d^2 k (k+l)\nonumber\\
		&&+6 d (k+l) (2 \sigma -q)-3 d (k+l) (2 \tau -q)\nonumber\\
		&&-3 d k (2 \tau -q)+3 d l\nonumber\\
		&&\left.+6 d \left(1-\frac{q}{2}\right) S_2+6 \left(1-\frac{q}{2}\right) S_3\right)\nonumber\\
		&&+\rho\varepsilon ^4 \left(\left(1-\frac{q}{2}\right) B(\delta ,\tau )+d \left(1-\frac{q}{2}\right) A(\delta)\right.\nonumber\\
		&&+\frac{d k^2 q}{2 \left(1-\frac{q}{2}\right)}-2 d (\delta -1) k\nonumber\\
		&&+4 (\beta -1) d (k+l)-2 d (\delta -1) (k+l)\nonumber\\
		&&+\frac{d k l q}{1-\frac{q}{2}}+\frac{3 d k q (k+l)}{1-\frac{q}{2}}\nonumber\\
		&&+4 d (k+l) \left(l-\frac{k q}{2 \left(1-\frac{q}{2}\right)}\right)+\frac{d l q (k+l)}{1-\frac{q}{2}}\nonumber\\
		&&+4 d k (k+l)-3 d (k+l)^2\nonumber\\
		&&-d (l-1) l+2 d \left(1-\frac{q}{2}\right) S_1+(k+l) (2 \sigma -q)\nonumber\\
		&&\left.-k (2 \tau -q)+l+2 \left(1-\frac{q}{2}\right) S_2\right)\nonumber\\
	    &&+\rho^2\varepsilon ^2 \left(2 d^2 \left(1-\frac{q}{2}\right) B(\delta ,\tau )+4 d \left(1-\frac{q}{2}\right) A(\tau)\right.\nonumber\\
	    &&+2 (\beta -1) d^3 (k+l)-2 (\delta -1) d^3 (k+l)\nonumber\\
	    &&+\frac{2 d^3 (k+l)^2}{1-\frac{q}{2}}+\frac{3 d^3 q (k+l)^2}{2 \left(1-\frac{q}{2}\right)}\nonumber\\
	    &&+\frac{2 d^3 k q (k+l)}{1-\frac{q}{2}}+\frac{d^3 l q (k+l)}{1-\frac{q}{2}}\nonumber\\
	    &&-\frac{2 d^3 k (k+l)}{1-\frac{q}{2}}-d^3 (k+l)^2\nonumber\\
	    &&+2 d^3 k (k+l)+3 d^2 (k+l) (2 \sigma -q)\nonumber\\
	    &&-d^2 (k+l) (2 \tau -q)-2 d^2 k (2 \tau -q)\nonumber\\
	    &&\left.+2 d^2 l+d^2 \left(1-\frac{q}{2}\right) S_2+2 d \left(1-\frac{q}{2}\right) S_3\right)\nonumber\\
	    &&+\rho^2\varepsilon ^3 \left(8 d \left(1-\frac{q}{2}\right) B(\delta ,\tau )+4 d^2 \left(1-\frac{q}{2}\right) A(\delta)\right.\nonumber\\
	    &&+4 \left(1-\frac{q}{2}\right) A(\tau)+\frac{2 d^2 k^2 q}{1-\frac{q}{2}}\nonumber\\
	    &&-8 (\delta -1) d^2 k+12 (\beta -1) d^2 (k+l)\nonumber\\
	    &&-4 (\delta -1) d^2 (k+l)+\frac{6 d^2 (k+l)^2}{1-\frac{q}{2}}\nonumber\\
	    &&+6 d^2 (k+l) \left(l-\frac{k q}{2 \left(1-\frac{q}{2}\right)}\right)+\frac{2 d^2 q (k+l)^2}{1-\frac{q}{2}}\nonumber\\
	    &&+\frac{4 d^2 k l q}{1-\frac{q}{2}}+\frac{8 d^2 k q (k+l)}{1-\frac{q}{2}}\nonumber\\
	    &&+\frac{2 d^2 l q (k+l)}{1-\frac{q}{2}}-\frac{6 d^2 k (k+l)}{1-\frac{q}{2}}\nonumber\\
	    &&-6 d^2 (k+l)^2+12 d^2 k (k+l)\nonumber\\
	    &&-4 d^2 (l-1) l+2 d^2 \left(1-\frac{q}{2}\right) S_1\nonumber\\
	    &&+6 d (k+l) (2 \sigma -q)+2 d (k+l) (2 \tau -q)\nonumber\\
	    &&-8 d k (2 \tau -q)+8 d l\nonumber\\
	    &&\left.+4 d \left(1-\frac{q}{2}\right) S_2+2 \left(1-\frac{q}{2}\right) S_3\right)\nonumber\\
		&&+\rho^2\varepsilon ^4 \left(2 \left(1-\frac{q}{2}\right) B(\delta ,\tau )+4 d \left(1-\frac{q}{2}\right) A(\delta)\right.\nonumber\\
		&&+\frac{2 d k^2 q}{1-\frac{q}{2}}-8 d (\delta -1) k\nonumber\\
		&&+6 (\beta -1) d (k+l)+2 d (\delta -1) (k+l)\nonumber\\
		&&+\frac{4 d k l q}{1-\frac{q}{2}}+\frac{2 d k q (k+l)}{1-\frac{q}{2}}\nonumber\\
		&&+6 d (k+l) \left(l-\frac{k q}{2 \left(1-\frac{q}{2}\right)}\right)+\frac{3 d q (k+l)^2}{2 \left(1-\frac{q}{2}\right)}\nonumber\\
		&&-\frac{d l q (k+l)}{1-\frac{q}{2}}+6 d k (k+l)\nonumber\\
		&&-3 d (k+l)^2-4 d (l-1) l\nonumber\\
		&&+2 d \left(1-\frac{q}{2}\right) S_1+(k+l) (2 \sigma -q)+(k+l) (2 \tau -q)\nonumber\\
		&&\left.-2 k (2 \tau -q)+2 l+\left(1-\frac{q}{2}\right) S_2\right)\nonumber\\
		&&+\rho^3\varepsilon ^3 \left(3 d \left(1-\frac{q}{2}\right) B(\delta ,\tau )+3 d^2 \left(1-\frac{q}{2}\right) A(\delta)\right.\nonumber\\
		&&+\left(1-\frac{q}{2}\right) A(\tau)+\frac{3 d^2 k^2 q}{2 \left(1-\frac{q}{2}\right)}\nonumber\\
		&&-6 (\delta -1) d^2 k+6 (\delta -1) d^2 (k+l)\nonumber\\
		&&+\frac{3 d^2 q (k+l)^2}{2 \left(1-\frac{q}{2}\right)}+\frac{3 d^2 k l q}{1-\frac{q}{2}}\nonumber\\
		&&-\frac{3 d^2 k q (k+l)}{1-\frac{q}{2}}-\frac{3 d^2 l q (k+l)}{1-\frac{q}{2}}\nonumber\\
		&&-3 d^2 (l-1) l+3 d (k+l) (2 \tau -q)\nonumber\\
		&&\left.-3 d k (2 \tau -q)+3 d l\right)\nonumber\\
		&&+\rho^3\varepsilon ^4 \left(\left(1-\frac{q}{2}\right) B(\delta ,\tau)+3 d \left(1-\frac{q}{2}\right) A(\delta)\right.\nonumber\\
		&&+\frac{3 d k^2 q}{2 \left(1-\frac{q}{2}\right)}-6 d (\delta -1) k\nonumber\\
		&&+6 d (\delta -1) (k+l)+\frac{3 d k l q}{1-\frac{q}{2}}\nonumber\\
		&&-\frac{3 d k q (k+l)}{1-\frac{q}{2}}+\frac{3 d q (k+l)^2}{2 \left(1-\frac{q}{2}\right)}\nonumber\\
		&&-\frac{3 d l q (k+l)}{1-\frac{q}{2}}-3 d (l-1) l\nonumber\\
		&&\left.+(k+l) (2 \tau -q)-k (2 \tau -q)+l\right),\nonumber
	\end{eqnarray}

\begin{eqnarray}
	I_7&=&\left(1-\frac{q}{2}\right) S_3 \varepsilon ^4\nonumber\\
	&&+\rho\varepsilon ^3 \left(3 d \left(1-\frac{q}{2}\right) A(\tau)+\frac{5 d^3 q (k+l)^2}{2 \left(1-\frac{q}{2}\right)}\right.\nonumber\\
	&&-\frac{2 d^3 (k+l)^2}{1-\frac{q}{2}}-2 d^3 (k+l)^2+3 d^2 (k+l) (2 \sigma -q)\nonumber\\
	&&\left.-3 d^2 (k+l) (2 \tau -q)+6 d \left(1-\frac{q}{2}\right) S_3\right)\nonumber\\
	&&+\rho \varepsilon ^4 \left(d \left(1-\frac{q}{2}\right) B(\delta ,\tau )+\left(1-\frac{q}{2}\right) A(\tau )\right.\nonumber\\
	&&+2 (\beta -1) d^2 (k+l)-2 (\delta -1) d^2 (k+l)\nonumber\\
	&&+2 d^2 (k+l) \left(l-\frac{k q}{2 \left(1-\frac{q}{2}\right)}\right)+\frac{d^2 q (k+l)^2}{2 \left(1-\frac{q}{2}\right)}\nonumber\\
	&&+\frac{2 d^2 k q (k+l)}{1-\frac{q}{2}}+\frac{d^2 l q (k+l)}{1-\frac{q}{2}}\nonumber\\
	&&-3 d^2 (k+l)^2+2 d^2 k (k+l)\nonumber\\
	&&+2 d (k+l) (2 \sigma -q)-d (k+l) (2 \tau -q)\nonumber\\
	&&-d k (2 \tau -q)+d l\nonumber\\
	&&\left.+2 d \left(1-\frac{q}{2}\right) S_2+2 \left(1-\frac{q}{2}\right) S_3\right)\nonumber\\
	&&+\rho^2\varepsilon ^2 \left(2 d^2 \left(1-\frac{q}{2}\right) A(\tau)+\frac{d^4 q (k+l)^2}{1-\frac{q}{2}}\right.\nonumber\\
	&&\left.+d^3 (k+l) (2 \sigma -q)-d^3 (k+l) (2 \tau -q)+d^2 \left(1-\frac{q}{2}\right) S_3\right)\nonumber\\
	&&+\rho^2 \varepsilon ^3 \left(4 d^2 \left(1-\frac{q}{2}\right) B(\delta ,\tau)+8 d \left(1-\frac{q}{2}\right) A(\tau)\right.\nonumber\\
	&&+4 (\beta -1) d^3 (k+l)-4 (\delta -1) d^3 (k+l)\nonumber\\
	&&+\frac{2 d^3 (k+l)^2}{1-\frac{q}{2}}+2 d^3 (k+l) \left(l-\frac{k q}{2 \left(1-\frac{q}{2}\right)}\right)\nonumber\\
	&&+\frac{4 d^3 q (k+l)^2}{1-\frac{q}{2}}+\frac{4 d^3 k q (k+l)}{1-\frac{q}{2}}\nonumber\\
	&&+\frac{2 d^3 l q (k+l)}{1-\frac{q}{2}}-\frac{2 d^3 k (k+l)}{1-\frac{q}{2}}\nonumber\\
	&&-2 d^3 (k+l)^2+4 d^3 k (k+l)\nonumber\\
	&&+6 d^2 (k+l) (2 \sigma -q)-2 d^2 (k+l) (2 \tau -q)\nonumber\\
	&&-4 d^2 k (2 \tau -q)+4 d^2 l\nonumber\\
	&&\left.+2 d^2 \left(1-\frac{q}{2}\right) S_2+4 d \left(1-\frac{q}{2}\right) S_3\right)\nonumber\\
	&&+\rho^2\varepsilon ^4 \left(4 d \left(1-\frac{q}{2}\right) B(\delta ,\tau)+2 d^2 \left(1-\frac{q}{2}\right) A(\delta)\right.\nonumber\\
	&&+2 \left(1-\frac{q}{2}\right) A(\tau )+\frac{d^2 k^2 q}{1-\frac{q}{2}}\nonumber\\
	&&-4 (\delta -1) d^2 k+6 (\beta -1) d^2 (k+l)\nonumber\\
	&&-2 (\delta -1) d^2 (k+l)+6 d^2 (k+l) \left(l-\frac{k q}{2 \left(1-\frac{q}{2}\right)}\right)\nonumber\\
	&&+\frac{5 d^2 q (k+l)^2}{2 \left(1-\frac{q}{2}\right)}+\frac{2 d^2 k l q}{1-\frac{q}{2}}\nonumber\\
	&&+\frac{4 d^2 k q (k+l)}{1-\frac{q}{2}}+\frac{d^2 l q (k+l)}{1-\frac{q}{2}}\nonumber\\
	&&-3 d^2 (k+l)^2+6 d^2 k (k+l)\nonumber\\
	&&-2 d^2 (l-1) l+d^2 \left(1-\frac{q}{2}\right) S_1\nonumber\\
	&&+3 d (k+l) (2 \sigma -q)+d (k+l) (2 \tau -q)\nonumber\\
	&&-4 d k (2 \tau -q)+4 d l\nonumber\\
	&&\left.+2 d \left(1-\frac{q}{2}\right) S_2+\left(1-\frac{q}{2}\right) S_3\right)\nonumber\\
	&&+\rho^3\varepsilon ^3 \left(3 d^2 \left(1-\frac{q}{2}\right) B(\delta ,\tau )+d^3 \left(1-\frac{q}{2}\right) A(\delta)\right.\nonumber\\
	&&+3 d \left(1-\frac{q}{2}\right) A(\tau )+\frac{d^3 k^2 q}{2 \left(1-\frac{q}{2}\right)}\nonumber\\
	&&-2 (\delta -1) d^3 k+2 (\delta -1) d^3 (k+l)\nonumber\\
	&&+\frac{d^3 q (k+l)^2}{2 \left(1-\frac{q}{2}\right)}+\frac{d^3 k l q}{1-\frac{q}{2}}\nonumber\\
	&&-\frac{d^3 k q (k+l)}{1-\frac{q}{2}}-\frac{d^3 l q (k+l)}{1-\frac{q}{2}}\nonumber\\
	&&-\left(d^3 (l-1) l\right)+3 d^2 (k+l) (2 \tau -q)\nonumber\\
	&&\left.-3 d^2 k (2 \tau -q)+3 d^2 l\right)\nonumber\\
	&&+\rho^3 \varepsilon ^4 \left(3 d \left(1-\frac{q}{2}\right) B(\delta ,\tau)+3 d^2 \left(1-\frac{q}{2}\right) A(\delta )\right.\nonumber\\
	&&+\left(1-\frac{q}{2}\right) A(\tau)+\frac{3 d^2 k^2 q}{2 \left(1-\frac{q}{2}\right)}\nonumber\\
	&&-6 (\delta -1) d^2 k+6 (\delta -1) d^2 (k+l)\nonumber\\
	&&+\frac{3 d^2 q (k+l)^2}{2 \left(1-\frac{q}{2}\right)}+\frac{3 d^2 k l q}{1-\frac{q}{2}}\nonumber\\
	&&-\frac{3 d^2 k q (k+l)}{1-\frac{q}{2}}-\frac{3 d^2 l q (k+l)}{1-\frac{q}{2}}\nonumber\\
	&&-3 d^2 (l-1) l+3 d (k+l) (2 \tau -q)\nonumber\\
	&&\left.-3 d k (2 \tau -q)+3 d l\right),\nonumber
\end{eqnarray}
	
	\begin{eqnarray}
		I_8&=&\rho  \varepsilon ^4 \left(d \left(1-\frac{q}{2}\right) A(\tau)+\frac{d^3 q (k+l)^2}{2 \left(1-\frac{q}{2}\right)}\right.\nonumber\\&&-d^3 (k+l)^2+d^2 (k+l) (2 \sigma -q)\nonumber\\&&
		\left.-d^2 (k+l) (2 \tau -q)+2 d \left(1-\frac{q}{2}\right) S_3\right)\nonumber\\&&
		\rho^2 \varepsilon ^3 \left(4 d^2 \left(1-\frac{q}{2}\right) A(\tau)+\frac{2 d^4 q (k+l)^2}{1-\frac{q}{2}}\right.\nonumber\\&&\left.+2 d^3 (k+l) (2 \sigma -q)-2 d^3 (k+l) (2 \tau -q)+2 d^2 \left(1-\frac{q}{2}\right) S_3\right)\nonumber\\&&+\rho^2\varepsilon ^4 \left(2 d^2 \left(1-\frac{q}{2}\right) B(\delta ,\tau)+4 d \left(1-\frac{q}{2}\right) A(\tau)\right.\nonumber\\&&+2 (\beta -1) d^3 (k+l)-2 (\delta -1) d^3 (k+l)\nonumber\\&&+2 d^3 (k+l) \left(l-\frac{k q}{2 \left(1-\frac{q}{2}\right)}\right)+\frac{5 d^3 q (k+l)^2}{2 \left(1-\frac{q}{2}\right)}\nonumber\\&&+\frac{2 d^3 k q (k+l)}{1-\frac{q}{2}}+\frac{d^3 l q (k+l)}{1-\frac{q}{2}}\nonumber\\&&+d^3 \left(-(k+l)^2\right)+2 d^3 k (k+l)\nonumber\\&&+3 d^2 (k+l) (2 \sigma -q)-d^2 (k+l) (2 \tau -q)\nonumber\\&&-2 d^2 k (2 \tau -q)+2 d^2 l\nonumber\\&&\left.+d^2 \left(1-\frac{q}{2}\right) S_2+2 d \left(1-\frac{q}{2}\right) S_3\right)\nonumber\\&&+\rho^3\varepsilon ^3 \left(d^3 \left(1-\frac{q}{2}\right) B_{\delta ,\tau }+3 d^2 \left(1-\frac{q}{2}\right) A_{\tau }\right.\nonumber\\
		&&\left.+d^3 (k+l) (2 \tau -q)-d^3 k (2 \tau -q)+d^3 l\right)\nonumber\\
		&&+\rho^3\varepsilon ^4 \left(3 d^2 \left(1-\frac{q}{2}\right) B(\delta ,\tau )+d^3 \left(1-\frac{q}{2}\right) A(\delta )\right.\nonumber\\&&+3 d \left(1-\frac{q}{2}\right) A(\tau )+\frac{d^3 k^2 q}{2 \left(1-\frac{q}{2}\right)}\nonumber\\&&-2 (\delta -1) d^3 k+2 (\delta -1) d^3 (k+l)\nonumber\\&&+\frac{d^3 q (k+l)^2}{2 \left(1-\frac{q}{2}\right)}+\frac{d^3 k l q}{1-\frac{q}{2}}\nonumber\\&&-\frac{d^3 k q (k+l)}{1-\frac{q}{2}}-\frac{d^3 l q (k+l)}{1-\frac{q}{2}}\nonumber\\&&-\left(d^3 (l-1) l\right)+3 d^2 (k+l) (2 \tau -q)\nonumber\\&&\left.-3 d^2 k (2 \tau -q)+3 d^2 l\right),\nonumber
	\end{eqnarray}
	\begin{eqnarray}
		I_9&=&\rho^2\varepsilon ^4 \left(2 d^2 \left(1-\frac{q}{2}\right) A(\tau)+\frac{d^4 q (k+l)^2}{1-\frac{q}{2}}\right.\nonumber\\
		&&\left.+d^3 (k+l) (2 \sigma -q)-d^3 (k+l) (2 \tau -q)+d^2 \left(1-\frac{q}{2}\right) S_3\right)\nonumber\\
		&&+\rho^3 \varepsilon ^3 d^3 \left(1-\frac{q}{2}\right)  A(\tau )+\rho^3 \varepsilon ^4\left(d^3 \left(1-\frac{q}{2}\right) B(\delta ,\tau )+3 d^2 \left(1-\frac{q}{2}\right) A(\tau )\right.\nonumber\\&&\left.+d^3 (k+l) (2 \tau -q)-d^3 k (2 \tau -q)+d^3 l\right),\nonumber
	\end{eqnarray}
	\begin{equation}
		I_{10}=\rho ^3\varepsilon ^4d^3 \left(1-\frac{q}{2}\right) A(\tau),\nonumber
	\end{equation}
	where $k=-\left(1-\frac{q}{2}\right)\b$, $l=p+q-1$, $S_1,S_2,S_3$ are defined as in Lemma \ref{cl} and  the functions $A,B$ are expressed as  \eqref{A} and \eqref{B} in Lemma \ref{l2} respectively.

\end{lem}

The following lemmas give the limit behavior of $I_i$ in Lemma \ref{a2} when $\e\to 0^{+}$ and $\rho\to 0^{+}$ respectively.

\begin{lem}\label{a3}
Let $I_j$ be expressed in Lemma \ref{a2} with $1\le j\le 10$.	For any $\e>0$, when $\rho\to 0^{+}$, then
	\begin{eqnarray}
		I_1&=&\left(1-\frac{q}{2}\right)S_1,\nonumber\\
		I_2&=&\left(1-\frac{q}{2}\right) S_2+4 \left(1-\frac{q}{2}\right) S_1 \varepsilon+O(\rho) ,\nonumber\\
		I_3&=&\left(1-\frac{q}{2}\right) S_3+4 \left(1-\frac{q}{2}\right) S_2 \varepsilon+6 \left(1-\frac{q}{2}\right) S_1 \varepsilon ^2+O(\rho),\nonumber\\
		I_4&=&4 \left(1-\frac{q}{2}\right) S_3 \varepsilon++6 \left(1-\frac{q}{2}\right) S_2 \varepsilon ^2+4 \left(1-\frac{q}{2}\right) S_1 \varepsilon^3+O(\rho),\nonumber\\
		I_5&=&6 \left(1-\frac{q}{2}\right) S_3 \varepsilon ^2+4 \left(1-\frac{q}{2}\right) S_2 \varepsilon ^3+\left(1-\frac{q}{2}\right) S_1 \varepsilon ^4+O(\rho),\nonumber\\
		I_6&=&4 \left(1-\frac{q}{2}\right) S_3 \varepsilon ^3+\left(1-\frac{q}{2}\right) S_2 \varepsilon ^4+O(\rho),\nonumber\\
		I_7&=&\left(1-\frac{q}{2}\right) S_3 \varepsilon ^4+O(\rho),\nonumber\\
		I_8&=&\rho  \varepsilon ^4 \left(d \left(1-\frac{q}{2}\right) A(\tau)+\frac{d^3 q (k+l)^2}{ \left(2-q\right)}-d^3 (k+l)^2\right.\nonumber\\&&\left.+2d^2 (k+l) (\sigma -\tau)+2 d \left(1-\frac{q}{2}\right) S_3\right)+O(\rho^2),\nonumber\\
		I_9&=&\rho^2\varepsilon ^4 \left(2 d^2 \left(1-\frac{q}{2}\right) A(\tau)+\frac{d^4 q (k+l)^2}{1-\frac{q}{2}}\right.\nonumber\\
		&&\left.+2d^3 (k+l) ( \sigma -\tau)+d^2 \left(1-\frac{q}{2}\right) S_3\right)+O(\rho^3),\nonumber\\
		I_{10}&=&\rho ^3\varepsilon ^4d^3 \left(1-\frac{q}{2}\right) A(\tau),\nonumber
	\end{eqnarray}
	where $k=-\left(1-\frac{q}{2}\right)\b$, $l=p+q-1$, $S_1,S_2,S_3$ are defined as in Lemma \ref{cl} and  the functions $A,B$ are expressed as  \eqref{A} and \eqref{B} in Lemma \ref{l2} respectively.
\end{lem}

\begin{lem}\label{a4}
	Let $I_j$ be expressed in Lemma \ref{a2} with $1\le j\le 10$.	For any $\rho>0$, when $\e\to 0^{+}$, then
	\begin{eqnarray}
		I_1&=&\left(1-\frac{q}{2}\right)S_1,\nonumber\\
		I_2&=&\left(1-\frac{q}{2}\right) S_2+O(\e) ,\nonumber\\
		I_3&=&\left(1-\frac{q}{2}\right) S_3+O(\e),\nonumber\\
		I_4&=&O(\e),\nonumber\\
		I_5&=&\rho \varepsilon  \left(d \left(1-\frac{q}{2}\right) A(\tau )+\frac{3 d^3 q (k+l)^2}{2 \left(1-\frac{q}{2}\right)}-\frac{2 d^3 (k+l)^2}{1-\frac{q}{2}}\right.\nonumber\\
		&&\left.+2d^2 (k+l) (\sigma-\tau)+2 d \left(1-\frac{q}{2}\right) S_3\right)+O(\e^2),\nonumber\\
		I_6&=&O(\e^2),\nonumber\\
		I_7&=&\rho^2\varepsilon ^2 \left(2 d^2 \left(1-\frac{q}{2}\right) A(\tau)+\frac{d^4 q (k+l)^2}{1-\frac{q}{2}}\right.\nonumber\\
		&&\left.+2d^3 (k+l) (\sigma -\tau)+d^2 \left(1-\frac{q}{2}\right) S_3\right)+O(\e^3),\nonumber\\
		I_8&=&O(\e^3),\nonumber\\
		I_9&=&\rho^3 \varepsilon ^3 d^3 \left(1-\frac{q}{2}\right)  A(\tau )+O(\e^4),\nonumber\\
		I_{10}&=&\rho ^3\varepsilon ^4d^3 \left(1-\frac{q}{2}\right)  A(\tau),\nonumber
	\end{eqnarray}
	where $k=-\left(1-\frac{q}{2}\right)\b$, $l=p+q-1$, $S_1,S_2,S_3$ are defined as in Lemma \ref{cl} and  the functions $A,B$ are expressed as  \eqref{A} and \eqref{B} in Lemma \ref{l2} respectively.
\end{lem}

\setcounter{thm}{0}

\renewcommand{\theequation}{B.\arabic{equation}}

\renewcommand{\thesubsection}{B.\arabic{subsection}}

\renewcommand{\thethm}{B.\arabic{thm}}

\section*{Appendix B: Estimates of non-sharp Liouville domain}

For \( q \in \left[0, 1 - \frac{1}{\sqrt{n-1}}\right] \) and \( q \ge \frac{5}{3} \), part (1) of Theorem \ref{lrm} is optimal. Next, we provide some estimates to the Liouville domain $\mathbb{L}(n)$ for the remaining case and compare them with existing partial results. All subsequent estimates can be derived through tedious mathematical analysis; however, to shorten the length of the paper, we directly provide the proofs of these facts using Mathematica. Readers may verify the correctness of these inequalities using the same code.

\begin{lem}\label{b1}
	Let $n\ge 3$, we have
	\begin{align*}
		&\left\{ (p,q) \in \mathbb{R}^2: p + q \leq \frac{n + 2}{n - 2} \right\} \cap \left( \mathbb{R} \times \left(0, \frac{5}{3} \right) \right) \subseteq \mathbb{L}(n), \\
		&\left( -\infty, \frac{3\sqrt{n + 6}}{2(n - 2)} \right] \times \left( 0, \frac{5}{3} \right) \subseteq \mathbb{L}(n).
	\end{align*}
This verifies the part $(4)$ and $(5)$ of Theorem \ref{m1}.
\end{lem}

\begin{proof}
	We divide the argument into three cases.
	
	Case 1. $q\in(0,1 - \frac{1}{\sqrt{n-1}}]$. It is easy to see 
	\[
	\left\{ (p,q) \in \mathbb{R}^2: p + q \leq \frac{n + 2}{n - 2} \right\} \subseteq 	\left\{ (p,q) \in \mathbb{R}^2: p + q < 1+ \frac{(2-q)^2}{(1-q)(n-2)} \right\}\subseteq \mathbb{L}(n).
	\]
	When $p\le \frac{3\sqrt{n + 6}}{2(n - 2)}$ and $q\in(0,1 - \frac{1}{\sqrt{n-1}}]$, we have
	\[
	p+q-1\le \frac{3\sqrt{n + 6}}{2(n - 2)}+q-1<\frac{(2-q)^2}{(1-q)(n-2)},
	\]
	and so $(p,q)\in\mathbb{L}(n)$. Here the second inequality is given by Mathematica:
	
	\begin{figure}[htbp]
		\centering
		\includegraphics[scale=0.36]{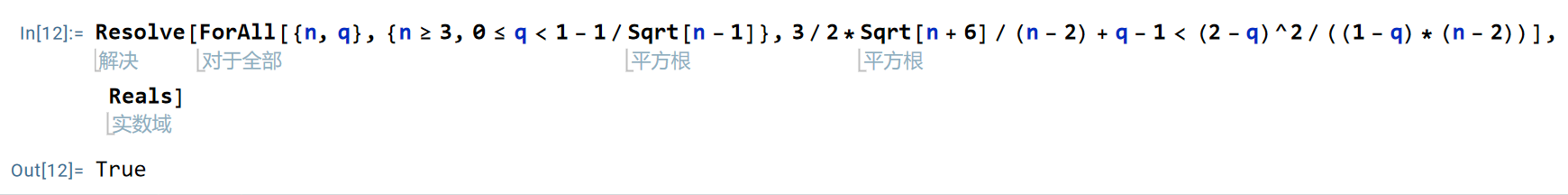}
	%	\caption{}
		\label{fig:ineq1}
	\end{figure}
	
	Case 2. $q\in(1 - \frac{1}{\sqrt{n-1}},1]$. In this case,% by Mathematica,  the constant $\mathcal{L}(n,q)$ given in Definition \ref{set} 	is 
	\begin{align*}
		&(p, q) \in \mathbb{L}(n) \Longleftrightarrow p + q - 1 < \mathcal{L}(n, q) \\
		&\qquad\qquad\quad\,\,\Longleftrightarrow p < \mathcal{L}(n, q) + 1 - q.
	\end{align*}
	So we only need to check that when $ q \in (1 - \frac{1}{\sqrt{n-1}},1], n \geq 3$, 
	\[
	\mathcal{L}(n, q) + 1 - q > \frac{3\sqrt{n + 6}}{2(n - 2)}
	\]
	and 
	\[
	\mathcal{L}(n, q)>\frac{4}{n-2}.
	\]
	That is, when $ q \in (1 - \frac{1}{\sqrt{n-1}},1], n \geq 3$, there exists a $y\in\mathbb{D}(n,q)$ (see Definition \ref{set}) such that
	\begin{equation}\label{B10}
		\frac{2-q}{n-2}+\frac{y(n-2)+2}{(n-2)(2-y)}+ 1 - q > \frac{3\sqrt{n + 6}}{2(n - 2)}
	\end{equation}
	and 
	\begin{equation}\label{B11}
		\frac{2-q}{n-2}+\frac{y(n-2)+2}{(n-2)(2-y)} > \frac{4}{n - 2}.
	\end{equation}
	By Mathematica, $\mathbb{D}(n,q)$ ($n\ge 4$) is given by 
	\begin{figure}[H]
	\centering
	\includegraphics[scale=0.36]{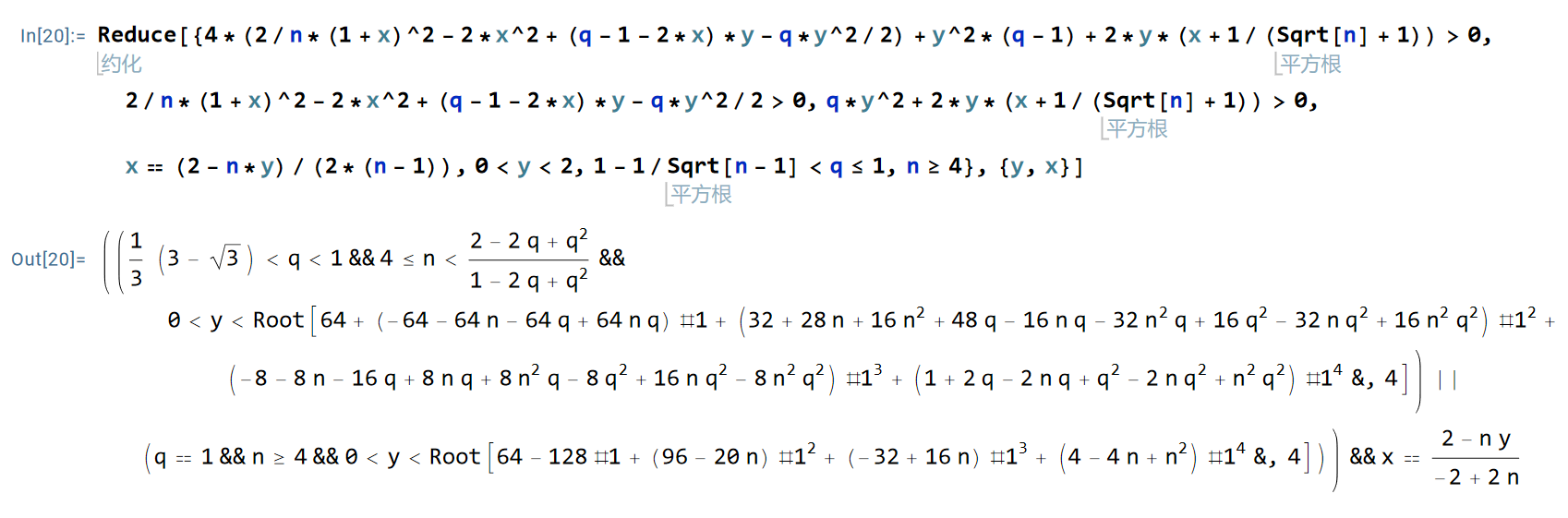}
	\end{figure}

	The inequality \eqref{B10}
	is given by
	
	\begin{figure}[H]
			\centering
		\includegraphics[scale=0.36]{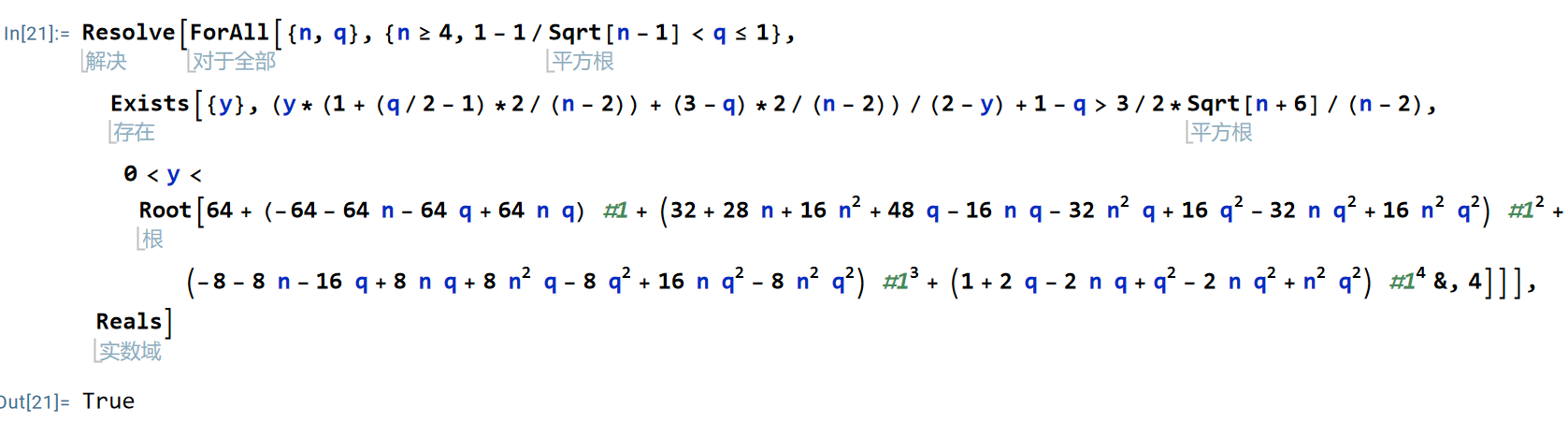}
	\end{figure}

	The inequality \eqref{B11} is given by

	\begin{figure}[H]
		\centering
		\includegraphics[scale=0.36]{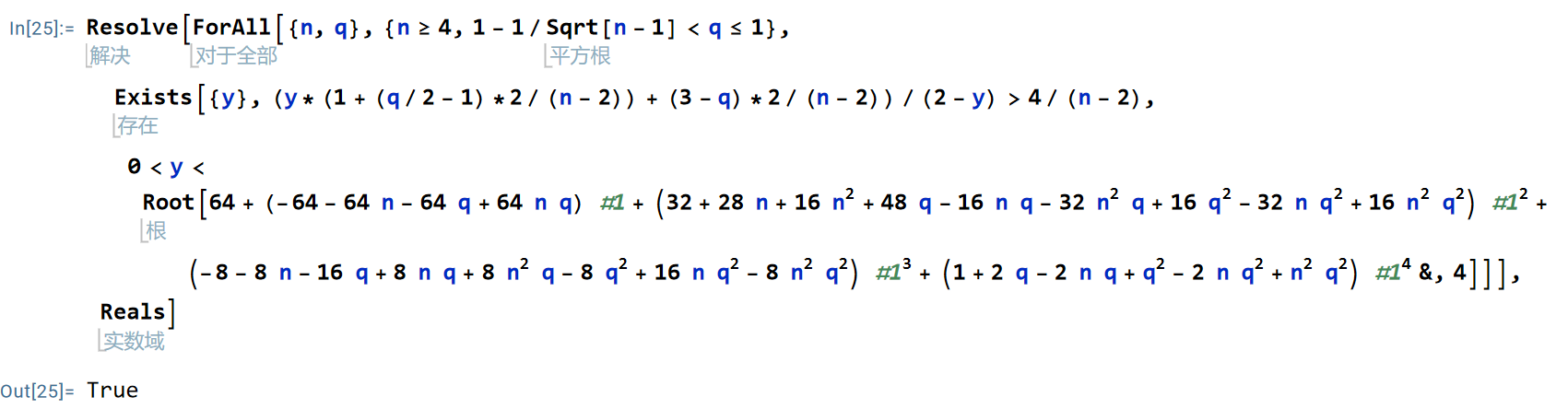}
		%	\caption{}
		%	\label{fig:ineq2}
	\end{figure}

When $n=3$, $\mathcal{L}(3,q)$ is given by
\begin{figure}[H]
	\centering
	\includegraphics[scale=0.36]{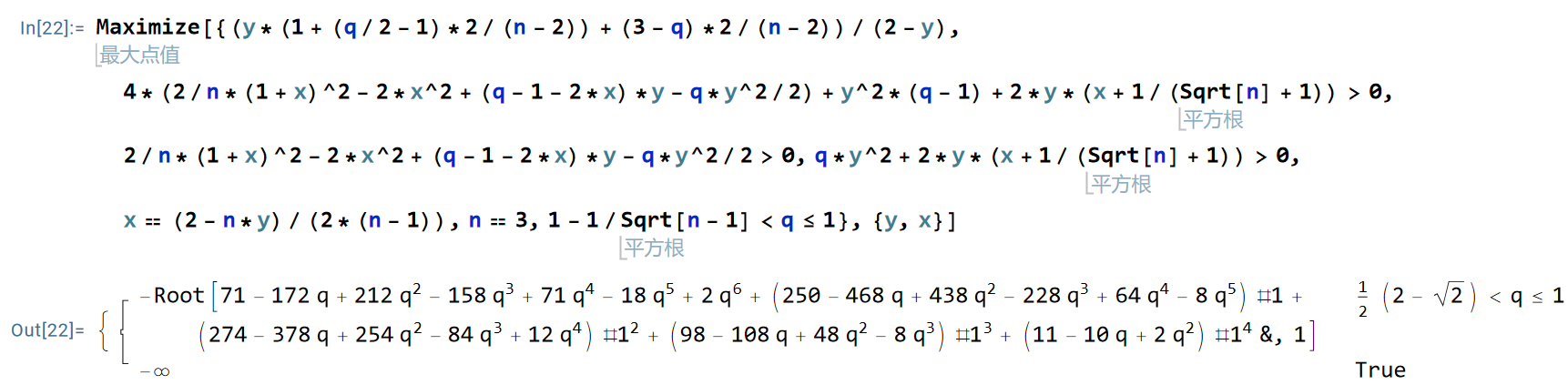}
	%	\caption{}
	%	\label{fig:ineq2}
\end{figure}

	The inequalities \eqref{B10} and \eqref{B11} are given by
	\begin{figure}[H]
		\centering
		\includegraphics[scale=0.36]{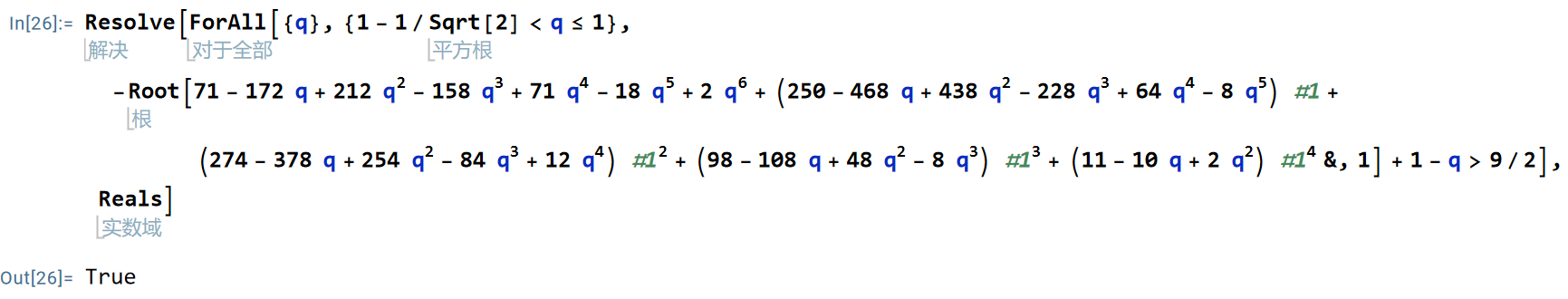}
		%	\caption{}
		%	\label{fig:ineq2}
	\end{figure}
	
	and
	
	\begin{figure}[H]
		\centering
		\includegraphics[scale=0.36]{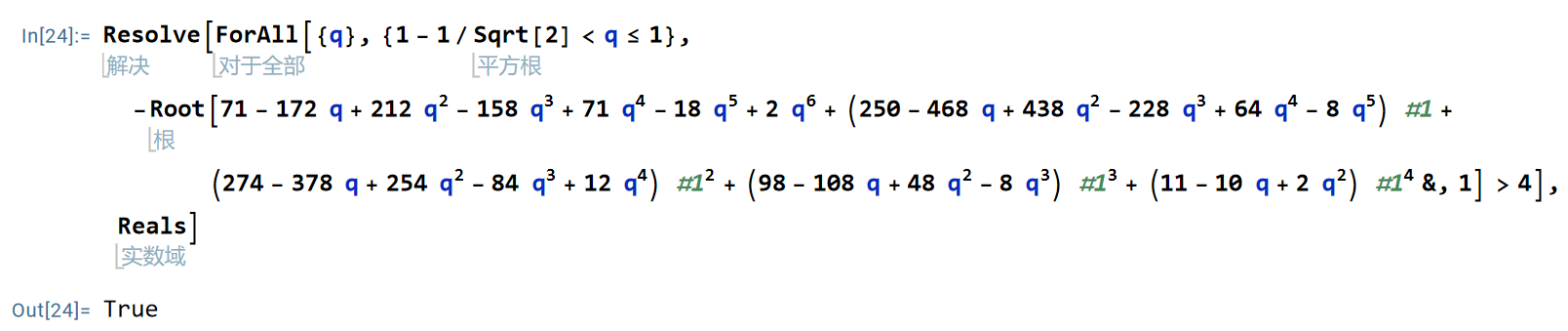}
		%	\caption{}
		%	\label{fig:ineq2}
	\end{figure}
	
	Thus we complete the proof of Case 2.
	
	\vspace{1mm}

	Case 3. $q\in(1,\frac{5}{3})$. In this case, %by Mathematica,  the constant $\mathcal{H}(n,q)$ given in Definition \ref{set} 	is 
	\begin{align*}
		&(p, q) \in \mathbb{L}(n) \Longleftrightarrow p + q - 1 < \mathcal{H}(n, q) \\
		&\qquad\qquad\quad\,\,\Longleftrightarrow p < \mathcal{H}(n, q) + 1 - q.
	\end{align*}
	So we only need to check that when $ q \in \left(1, \frac{5}{3}\right), n \geq 3$,
	\[
	\mathcal{H}(n, q) + 1 - q > \frac{3\sqrt{n + 6}}{2(n - 2)}.
	\]
	That is, when $ q \in \left(1, \frac{5}{3}\right), n \geq 3$, there exists a $y\in\mathbb{E}(n,q)$ (see Definition \ref{set}) such that
	\begin{equation}\label{B1}
		\frac{2-q}{n-2}+\frac{y(n-2)+2}{(n-2)(2-y)}+ 1 - q > \frac{3\sqrt{n + 6}}{2(n - 2)}.
	\end{equation}
	By Mathematica, $\mathbb{E}(n,q)$ is given by 
		\begin{figure}[htbp]
		\centering
		\includegraphics[scale=0.36]{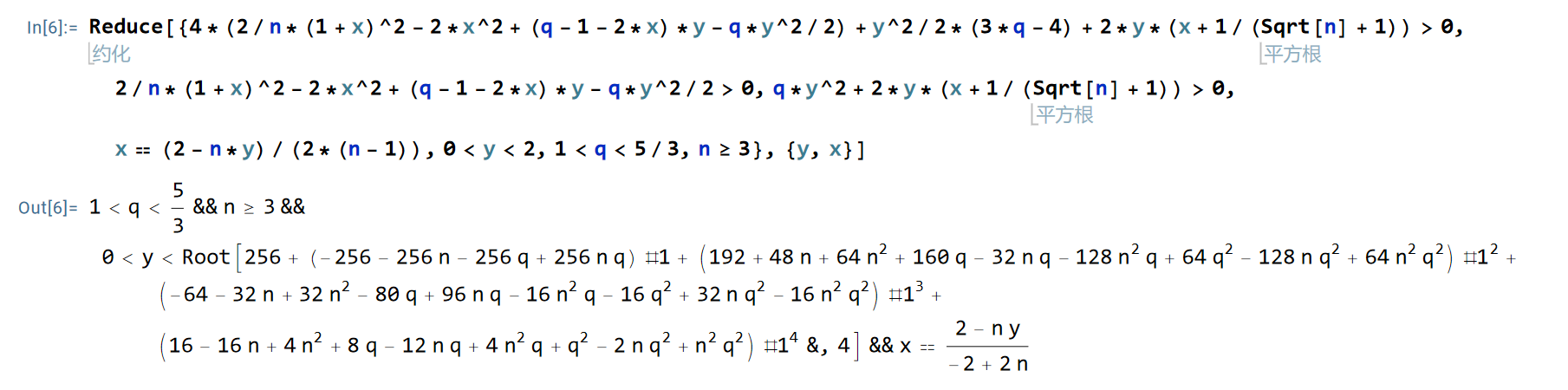}
		%	\caption{}
		\label{fig:ineq2}
	\end{figure}

	and inequality \eqref{B1}
is given by
		\begin{figure}[H]
		\centering
		\includegraphics[scale=0.36]{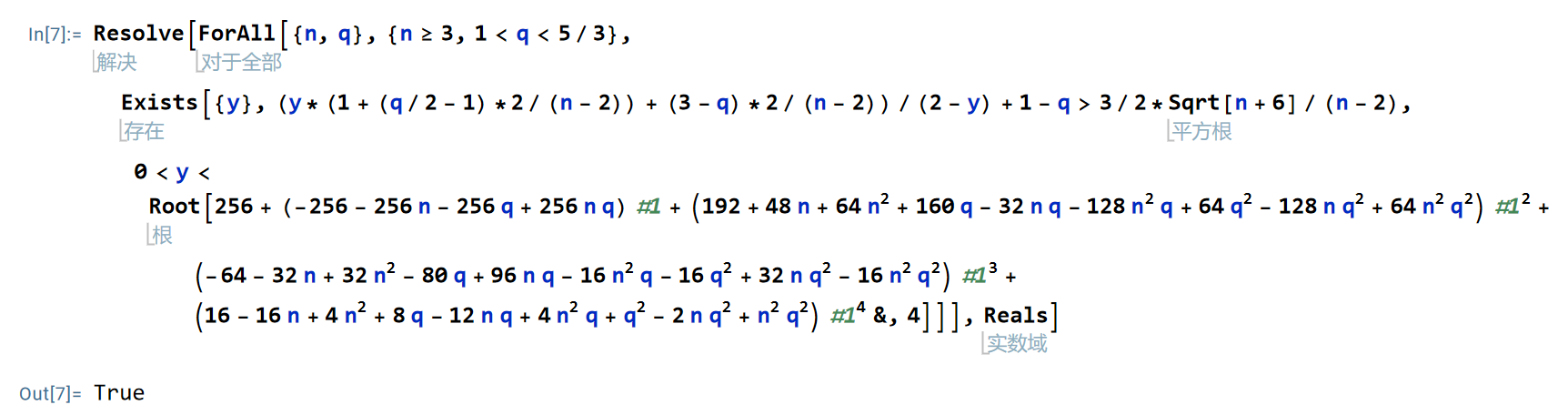}
		%	\caption{}
		\label{fig:ineq3}
	\end{figure}

	Notice that when $q>1$, $p+q\le\frac{n+2}{n-2}$ yields $p<\frac{4}{n-2}< \frac{3\sqrt{n + 6}}{2(n - 2)}$ and so when 
	$q\in(1,\frac{5}{5})$,
	\[
	p+q\le\frac{n+2}{n-2}\Longrightarrow (p,q)\in \mathbb{L}(n).
	\]
	We finish the proof of the lemma.

\end{proof}

The following lemma compares our results with Ma-Wu's result \cite[Theorem 1,3,Theorem 1.4]{MW}, which  improved the  Bidaut-Véron-García-Huidobro-Véron's Liouville theorem.  For $q\in(\frac{1}{n-1},2)$, they prove that if $p\ge 0$, $p+q-1>0$ and $\mathbb{H}(p,q)<0$, then $(p,q)\in \mathscr{D}_L(n)$, where
\begin{equation}\label{B4}
	\mathbb{H}(p, q):=p^2+\left[\frac{n-1}{n-2} q-\frac{n^2-3}{(n-2)^2}\right] p+\frac{1-(n-1) q}{(n-2)^2} .
\end{equation}
Actually, from Figure \ref{F02}, when $n=6$, we observe intuitively that the Liouville domain $\mathbb{L}(n)$ is larger than theirs. 

\begin{lem}\label{b2}
	When $n\ge 4$ and $q\in\left(1-\frac{1}{\sqrt{n-1}},\frac{5}{3}\right)$, if $(p,q)$ satisfies $\mathbb{H}(p, q)<0$, where $\mathbb{H}(p, q)$ is given by \eqref{B4}, then $(p,q)\in\mathbb{L}(n)$.
\end{lem}

\begin{proof}
	
	We assume $p>0$ and $z:=p+q>1$  in the proof or by Lemma \ref{b1}, we already have $(p,q)\in \mathbb{L}(n)$.
	
	Case 1. $q\in(1 - \frac{1}{\sqrt{n-1}},1]$. In this case, $\mathbb{H}(p, q)<0$ is equivalent to the following inequality of $z$:

	\begin{figure}[H]
		\centering
		\includegraphics[scale=0.36]{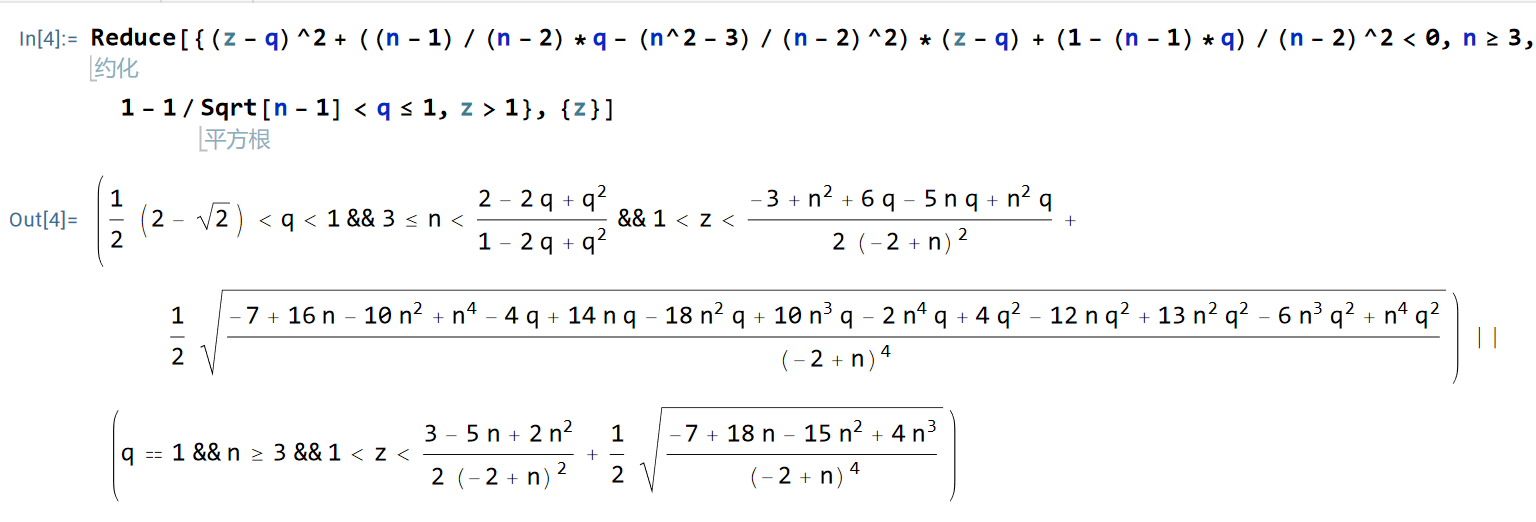}
			\caption{Upper bound of $z$ by $\mathbb{H}(p, q)<0$}
		\label{fig:upz}
	\end{figure}
	
		By the definition of 
		$\mathbb{L}(n)$,
		\[
		(p, q) \in \mathbb{L}(n) \Longleftrightarrow z=p + q  < \mathcal{L}(n, q)+1 .
		\]
		So we need to check the upper bound in Figure \ref{fig:upz} is smaller than $\mathcal{L}(n, q)+1$. When $n\ge 4$, it is given by ( the $\mathbb{D}(n,q)$ is given in Case 2 of the proof of Lemma \ref{b1})
		
			\begin{figure}[H]
			\centering
			\includegraphics[scale=0.36]{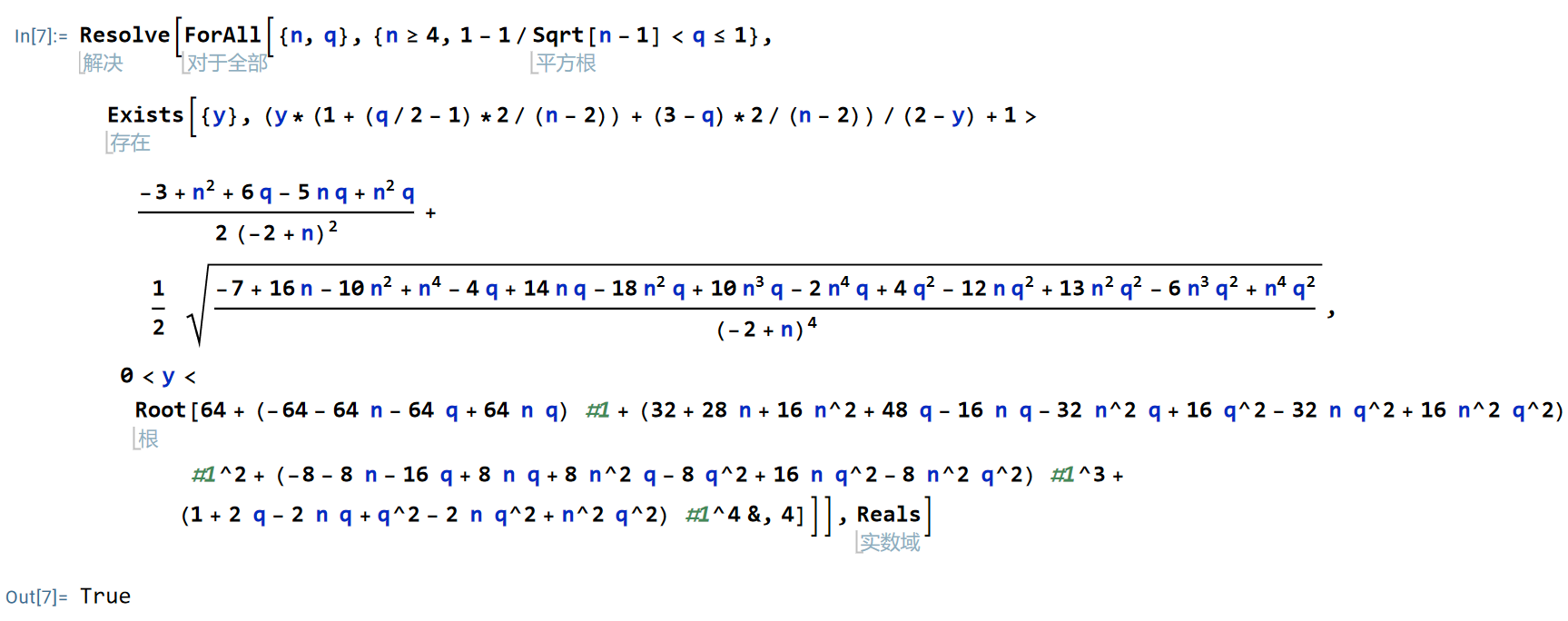}
		%	\caption{Upper bound of $z$ by $\mathbb{H}(p, q)<0$}
			%\label{fig:upz}
		\end{figure}

	Case 2. $q\in(1,\frac{5}{3})$.
	In this case, we actually have 
	\[
	\mathbb{H}(p, q)<0\Longrightarrow p<\frac{3\sqrt{n + 6}}{2(n - 2)},
	\]
	and so by Lemma \ref{b1}, $(p,q)\in \mathbb{L}(n)$. Above inequality is given by

		\begin{figure}[H]
		\centering
		\includegraphics[scale=0.36]{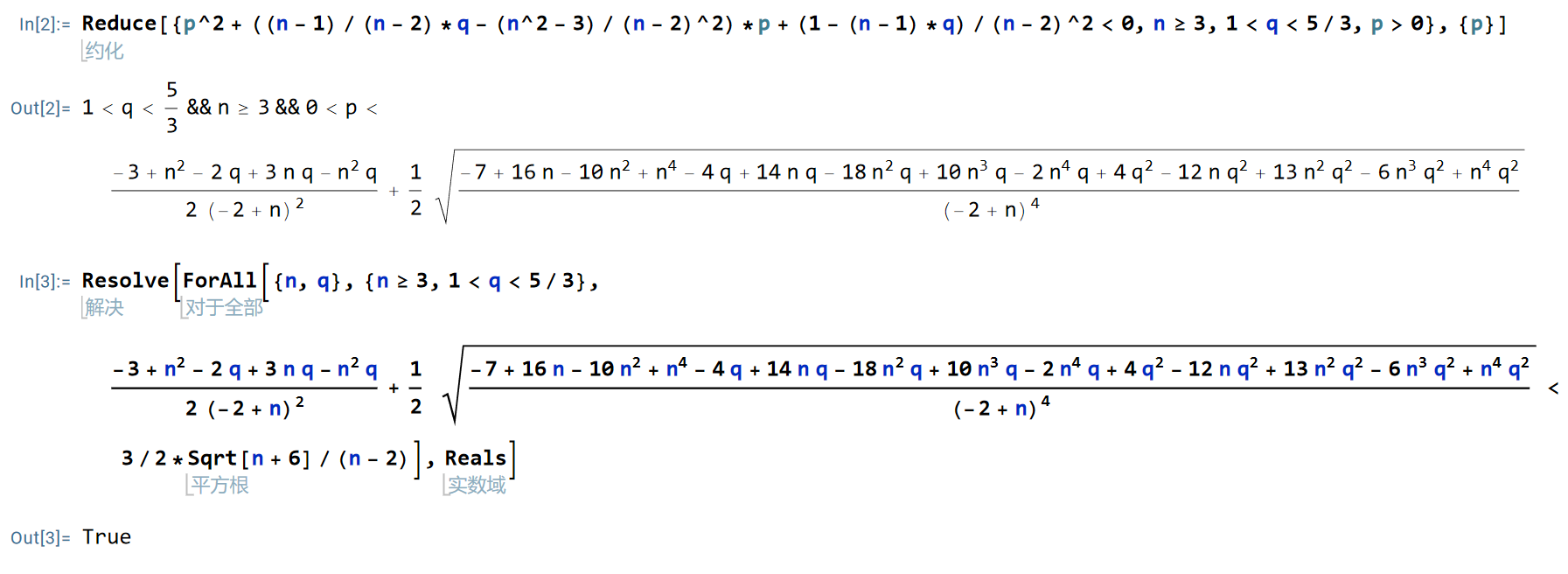}
		%	\caption{Upper bound of $z$ by $\mathbb{H}(p, q)<0$}
		%\label{fig:upz}
	\end{figure}
	
\end{proof}

\begin{rmk}
	From the Case 2 of the proof of Lemma \ref{b2}, we do not use $n\ge 4$ and so our Liouville domain is larger than known results when $q\ge 1$. However, when $n=3$,  Ma-Wu's sharp  Liouville domain is valid for $q\in[0,\frac{1}{2}]$, which is better than our $q\in[0,1-\frac{1}{\sqrt{2}}]$. By Mathematica, in this case, one can verify that our Liouville domain is also larger than theirs when $q\ge\frac{3}{4}$.
\end{rmk}

%	\noindent\textbf{Data availability}\\
%	\noindent Data sharing not applicable to this article as no datasets were generated or analysed during
%	the current study.\\
	\vspace{3mm}
	
	\noindent\textbf{Declarations}\\
	%\noindent	\textbf{conflict of interest} 
	The authors declare that they have no competing interests. 
	\vspace{5mm}
	
	\noindent\textbf{Data availability} \\
	Data sharing not applicable to this article as no datasets were generated or analysed during
	the current study.
	
	\vspace{5mm}
	\noindent\textbf{Acknowledgments}\\
We warmly thank Dr. Hua Zhu for his kind and helpful communications regarding their very recent interesting work \cite{DSWZ}. We also warmly thank Professor Laurent Véron for his kind and valuable communications regarding the present work.
	The author is supported by Scientific Research Startup Project of Jiangsu Normal University (Project No: 24XFRS051), the General Program of Basic Scientific Research in Institutions of Higher Education of Jiangsu Province (Grant No: 25KJB110002) and National Natural Science Foundation of China (Grant No: 12526552).

\end{document}